\documentclass[11pt]{preprint}
\usepackage[full]{textcomp}
\usepackage[osf]{newtxtext}
\usepackage{comment}

\usepackage{amssymb}
\usepackage{mathtools}
\usepackage{mhenvs}
\usepackage{mhequ}
\usepackage{booktabs}
\usepackage{tikz}
\usepackage{mathrsfs}
\usepackage{longtable}

\usepackage{microtype}
\usepackage{wasysym}
\usepackage{centernot}
\usepackage{enumitem}

\usepackage{hyperref}
\usepackage{cleveref}
\usepackage{mhsymb}

\usetikzlibrary{arrows}
\tikzset{
	partition/.style={line width=6mm,draw=red!20,line cap=round,line join=round},
    pairing/.style={thick,draw=black,*-*,shorten >=-2.9pt,shorten <=-2.9pt},
  }

\makeatletter
\newcommand{\globalcolor}[1]{%
  \color{#1}\global\let\default@color\current@color
}
\makeatother

\newif\ifdark
\IfFileExists{dark}{\darktrue}{\darkfalse}

\ifdark

\definecolor{darkred}{rgb}{0.9,0.2,0.2}
\definecolor{darkblue}{rgb}{0.7,0.3,1}
\definecolor{darkgreen}{rgb}{0.1,0.9,0.1}
\definecolor{pagebackground}{rgb}{.15,.21,.18}
\definecolor{pageforeground}{rgb}{.84,.84,.85}
\pagecolor{pagebackground}
\AtBeginDocument{\globalcolor{pageforeground}}

\else

\definecolor{darkred}{rgb}{0.7,0.1,0.1}
\definecolor{darkblue}{rgb}{0.4,0.1,0.8}
\definecolor{darkgreen}{rgb}{0.1,0.7,0.1}
\definecolor{pagebackground}{rgb}{1,1,1}
\definecolor{pageforeground}{rgb}{0,0,0}

\fi

\theoremnumbering{Alph}
\renewtheorem{theorem*}{Theorem}

\DeclareSymbolFont{timesoperators}{T1}{ptm}{m}{n}
\SetSymbolFont{timesoperators}{bold}{T1}{ptm}{b}{n}

\makeatletter
\renewcommand{\operator@font}{\mathgroup\symtimesoperators}
\makeatother

\def\emptyset{{\centernot\Circle}}

\DeclareMathAlphabet{\mathbbm}{U}{bbm}{m}{n}
\DeclareFontFamily{U}{BOONDOX-calo}{\skewchar\font=45 }
\DeclareFontShape{U}{BOONDOX-calo}{m}{n}{
  <-> s*[1.05] BOONDOX-r-calo}{}
\DeclareFontShape{U}{BOONDOX-calo}{b}{n}{
  <-> s*[1.05] BOONDOX-b-calo}{}
\DeclareMathAlphabet{\mcb}{U}{BOONDOX-calo}{m}{n}
\SetMathAlphabet{\mcb}{bold}{U}{BOONDOX-calo}{b}{n}

\let\epsilon\varepsilon

\definecolor{LB}{rgb}{0.29, 0.63, 0.73}

\def\f{\frac}
\def\1{\mathbf{1}}
\def\E{{\symb E}}

\def\L{{\mathbb L}}
\def\bbB{{\mathbb B}}
\def\Lip{{\mathrm{Lip}}}

\def\${|\!|\!|}

\def\<{\langle}
\def\>{\rangle}

\setlist{noitemsep,topsep=4pt}
\def\para_#1{/\!\!/_{\!#1}}

\def\WW{\mathbb{W}}

\def\ZZ{\mathbb{Z}}
\def\Cr{\mathscr{C}}
\def\Dom{\mathrm{Dom}}

\def\XX{{\mathbb X}}

\def\Z{{\mathbf Z}}
\def\ZZ{{\mathbb Z}}
\def\JJ{{\mathbb J}}
\def\A{\mathcal A}

\def\mft{\mathfrak{t}}

\def\K{\mathrm{K}}

\def\Ito{{\textnormal{\tiny Itô}}}

\def\slash{\leavevmode\unskip\kern0.18em/\penalty\exhyphenpenalty\kern0.18em}
\def\dash{\leavevmode\unskip\kern0.18em--\penalty\exhyphenpenalty\kern0.18em}

\makeatletter 
\newcommand*{\fat}{}
\DeclareRobustCommand*{\fat}{%
\mathbin{\mathpalette\bigcdot@{}}}
\newcommand*{\bigcdot@scalefactor}{.5}
\newcommand*{\bigcdot@widthfactor}{1.15}
\newcommand*{\bigcdot@}[2]{%
  \sbox0{$#1\vcenter{}$}
  \sbox2{$#1\cdot\m@th$}%
  \hbox to \bigcdot@widthfactor\wd2{%
    \hfil
    \raise\ht0\hbox{%
      \scalebox{\bigcdot@scalefactor}{%
        \lower\ht0\hbox{$#1\bullet\m@th$}%
      }%
    }%
    \hfil
  }%
}

\DeclareRobustCommand{\TitleEquation}[2]{\texorpdfstring{\StrLeft{\f@series}{1}[\@firstchar]$\if%
b\@firstchar\boldsymbol{#1}\else#1\fi$}{#2}}

\makeatother

\def\newoptest#1{\expandafter\DeclareMathOperator\csname#1\endcsname{#1}}
\newoptest{Ciao}

\newtheorem{assumption}[lemma]{Assumption}

\title{Generating diffusions with fractional Brownian motion}
\author{Martin~Hairer and Xue-Mei~Li}

\institute{Imperial College London, UK\\
\email{m.hairer@imperial.ac.uk, xue-mei.li@imperial.ac.uk}}

\begin{document}
\maketitle
\begin{abstract}
We study fast\slash slow systems  driven by  a fractional Brownian
motion $B$ with Hurst parameter $H\in (\f 13, 1]$.
Surprisingly, the slow dynamic converges on
suitable timescales to a limiting Markov process and we describe its generator.
More precisely, if $Y^\eps$ denotes a Markov process with sufficiently good mixing properties 
evolving on a fast
timescale $\eps \ll 1$, the solutions of the equation
\begin{equ}
dX^\eps = \epsilon^{\f 12-H} F(X^\eps,Y^\eps)\,dB+F_0(X^\eps,Y^\epsilon)\,dt\;
\end{equ}
converge to a regular diffusion without having to assume that $F$ averages to $0$, provided that $H< \f12$.
For $H > \f12$, a similar result holds, but this time it does require $F$ to average to $0$.
We also prove that the $n$-point motions converge to those of a Kunita type SDE. 

One nice interpretation of this result is that it provides a continuous interpolation between
the time homogenisation theorem for random ODEs with rapidly oscillating right-hand sides ($H=1$) 
and the averaging of diffusion processes ($H= \f12$). \\[.4em]
\noindent {\scriptsize The authors have no conflict of interest to disclose.}
\end{abstract}
\setcounter{tocdepth}{2}

\tableofcontents

\section{Introduction}

The setting considered in this article is as follows.
Consider a particle in a rapidly evolving random medium, so that it is 
governed by a stochastic differential equation of the type 
$dx_t= A(x_t, t/\eps)\,dt + \sigma(x_t,t/\eps)\,dB$\footnote{Following the tradition of probability theory, the subscript $t$ denotes dependence on the time parameter, not differentiation.} for a small
parameter $\eps>0$.
The situation we are interested in is where, in the ``static'' case (i.e.\ when $A$ and $\sigma$
have no explicit time dependence), the system is either super- or subdiffusive.
This is the case if the driving noise $B$ is modelled by fractional Brownian motion (fBM)
with Hurst parameter $H \neq \f12$.
Recall that fractional noises (i.e.\ the time derivative of fBM) can be obtained as scaling 
limits in statistical mechanics models \cite{Jona-Lasinio,Dobrushin,Sinai} and that 
fBM with Hurst parameter $H$ is a Gaussian 
process with stationary increments and self-similarity exponent $H$. It is therefore 
characterised (up to an irrelevant global shift) by the fact that 
$\E(B_t-B_s)^2=|t-s|^{2H}$, so that it is superdiffusive for $H > \f12$ and subdiffusive for
$H < \f12$. 
The covariance of its increments, $\E(B_{t+1}-B_{t})(B_{s+1}-B_s)$,  
decays at rate  $|t-s|^{2H-2}$ for large $|t-s|$ and therefore exhibits long-range dependence when
$H > \f12$.

We furthermore assume that the rapid time evolution of the environment is described by a hidden
Markov variable, thus leading to the model
\begin{equ}
\label{1}
x_t^\eps =x_0+ \int_0^t F(x_s^\eps,y_s^\eps)\,dB_s + \int_0^t F_0(x_s^\eps,y_s^\eps)\,ds,
\end{equ}
with $B$ an fBM with Hurst parameter $H\in (0,1)$ in $\R^m$ and $F(x,y)\in \L(\R^m, \R^d)$.
The stochastic integral appearing in the first term is problematic when $H < \f12$: one
should really interpret this equation as $x_t^\eps = \lim_{\delta \to 0} x_t^{\eps,\delta}$ with
 $x_t^{\eps,\delta}$ driven by a smooth approximation $B^\delta$ to $B$ with relevant timescale
 $\delta \ll \eps \ll 1$, see Section~\ref{sec:description} below.
Regarding the fast Markov variable, a prototypical situation is that of a system of the type
\begin{equ}[e:startingSDE]
dy_t^\eps =  \sigma(x_t^\eps,y_t^\eps)\,{dW_t \over\sqrt\eps} + b(x_t^\eps,y_t^\eps)\,{dt\over\eps},  \;
\end{equ}
where  $W$ is a Wiener processes independent of the fBM $B$ appearing in (\ref{1}). 
This allows for the case where the variable $x$ feeds back into the evolution of $y$, but for most
of this article we assume that there is no $x$-dependence in \eqref{e:startingSDE}.
We also assume that $y_t$ admits a unique invariant probability measure~$\mu$. In the case
with feedback, we have a family of invariant measure $\mu_x$ obtained by ``freezing'' the value of the 
variable $x$ in \eqref{e:startingSDE}.
 
It was recently shown by the authors in \cite{AveragingHigh} that in the case $H > \f12$
the process $x_\eps$ converges \textit{in probability}  to the solution to
\begin{equ}[e:limitODE]
dx = \bar F(x)\,dB + \bar F_0(x)\,dt \;,
\end{equ}
where the average of any function $h$ is given by $\bar h(x) = \int h(x,y)\,\mu_x(dy)$.
The aim of the present article is to investigate the two cases left out by the aforementioned analysis,
namely what happens when either $H < \f12$ or when $H > \f12$ but $\bar F = 0$ in \eqref{e:limitODE}?

\subsection{Description of the model}
\label{sec:description}

It turns out that the effect of the rapid oscillatory motion described by the fast variable $y$
is to slow down the motion of $x$ in the superdiffusive case and to speed it up in the 
subdiffusive case. This can be explained by the following heuristics. 
For times of order $t \lesssim \eps$, the process $Y$ doesn't evolve much so that,
by the scaling property of the driving fBM, one expects the 
process $x$ to move by about $\eps^H$ in a time of order $\eps$. On large times $t \gtrsim \eps$ 
on the other hand
we will see that the limiting process is actually Markovian, even in the case with long-range
dependence. This suggests that over times of order $t$ the process $x$ performs about $t/\eps$
steps of a random walk with step size $\eps^H$ and therefore moves by about
$\eps^H \sqrt{t/\eps}$. This suggests that one should multiply $F$ by $\eps^{\f12-H}$
in order to obtain a non-trivial limit.

As a consequence, the equations we actually study in this article are of the form:
\begin{equ}[e:SDEbasic]
dX^\eps = \eps^{\f12 - H} F_i(X^\eps,Y^\eps)\,dB_i + F_0(X^\eps,Y^\eps)\,dt,
 \;, \qquad Y^\eps(t) = Y(t/\eps)\;,
\end{equ}
(summation over $i$ is implied),  where $B$ is a fractional Brownian motion with 
Hurst parameter $H$ ranging from $\f 13$ to $1$\footnote{We could probably deal with $H \in (\f14,\f13]$ with our techniques, but this would obscure some of the arguments for relatively little gain. For $H \le \f14$, there exists no solution theory even in the absence of~$Y$.}, and $Y$ is an independent stationary Markov process with values in some Polish space $\CY$, invariant measure $\mu$
 and generator $-\CL$\footnote{The convention of adding a minus sign to 
 the generator simplifies our expressions later on.}. At the moment, we are unfortunately 
 unable to cover the
case when $X$ feeds back into the dynamics of $Y$. When $H > \f12$, we furthermore assume that 
$\int F_i(x,y)\,\mu(dy) = 0$ for every $i \neq 0$ and every~$x$.

Our main result is that, as $\eps \to 0$, solutions to \eqref{e:SDEbasic} converge in law
to a limiting Markov process and we provide an expression for its generator. In fact, we 
have an even stronger form of convergence, namely we
show that the \textit{flow} generated by \eqref{e:SDEbasic} converges to the one generated by
a limiting stochastic differential equation of Kunita type (i.e.\ driven by an infinite-dimensional
noise).

\begin{remark}
Of course, (\ref{e:SDEbasic}) is not quite the same as \eqref{e:startingSDE} which was our starting point. 
One way of relating them more directly is to perform a time change and set $X^\eps = x_\eps(\eps^{1-2H}t)$ with $x_\eps$ solving \eqref{e:startingSDE}.
Then $X^\eps$ solves the equation
$dX^\eps =\tilde \eps^{{1\over 2}-H} F_i(X^\eps,Y^{\tilde\eps})\,dB_i+  \tilde \eps^{\f 1{2H}-1}F_0(X^\eps,Y^{\tilde \epsilon})\,dt$
where we have set $\tilde \epsilon=\eps^{2H}$. 
When $H < \f12$, this then converges to the same limit as \eqref{e:SDEbasic}, but of course 
with $F_0$ in  \eqref{e:SDEbasic} set to $0$. When $H > \f12$, then one would need to take $F_i$ centred in 
\eqref{e:startingSDE} in order to obtain a non-trivial limit and our results imply that one
again converges to the same limit as \eqref{e:SDEbasic}, at least in the case $F_0 = 0$.
\end{remark}

The special case when $F_0 = 0$ and the $F_i$ are independent of the $x$-variable yields a functional
central limit theorem for stochastic integrals against fractional Brownian motion. This
already appears to be new by itself and might be of independent interest.

As already hinted at, the map $t \mapsto F_i(\cdot, Y^\eps_t)$ is 
too irregular to fit into the standard 
theory of differential equations driven by a fractional Brownian motion, especially when $H < \f12$, 
so that it is not even completely clear a priori how to  interpret \eqref{e:SDEbasic} for fixed $\eps > 0$.
These questions will be addressed in more detail in Section~\ref{sec:precise} below.
Let us put these aside for the moment and consider  the following ordinary differential equation
\begin{equ}[e:deltaeps] 
\dot X^{\eps,\delta} _s = \f {\delta^{H-1}}{ \sqrt \eps}  v\Big(\f s{\eps \delta}\Big) F(X_s^{\eps,\delta},Y_s^\eps) + F_0(X_s^{\eps,\delta},Y_s^\epsilon)\;,
\end{equ}
where $v$ is a smooth stationary Gaussian random process with covariance $C$ such that
$C(t) \sim |t|^{2H-2}$ for $|t|$ large. When $H < \f12$ we furthermore assume that 
$\int C(t)\,dt = 0$ and, when $H = \f12$, we assume that $C$ decays exponentially and
satisfies $\int C(t)\,dt = 1$.
One way of obtaining such a process $v$ is to set 
$v = \phi * \dot B$ for $\phi$ a Schwartz test function integrating to $1$
 (and $*$ denoting convolution in time).
This in particular shows that, at least in law, one has
$(\eps\delta)^{H-1}v\big({t \over \eps\delta}\big)
=( \phi_{\eps\delta}* \dot B)(t)$, where we set $\phi_\eps(t) = \eps^{-1}\phi(t/\eps)$.
Since this converges in law to $\dot B$ as $\eps \delta \to 0$,
we can view \eqref{e:deltaeps} as an approximation to \eqref{e:SDEbasic}.

It is then possible to show that the limit $X^\eps = \lim_{\delta \to 0} X^{\eps,\delta}$
exists and our results hold with $X^\eps$ interpreted in this way. Furthermore, we will see
that all our results hold uniformly over $\delta \in (0,1]$ as $\eps \to 0$.
This in particular shows that the converse limit obtained by first sending $\eps \to 0$
and then $\delta \to 0$ is the same, as are all limits obtained by other ways
of jointly sending $\eps,\delta \to 0$.


\subsection{Description of the main results}

We now give a precise formulation of our main results, albeit with a simplified 
set of assumptions. The reason is that while the simplified assumptions are straightforward
to state, they are very stringent regarding the Markov process $Y$. The more realistic
set of assumptions used in the remainder of the article however is quite technical to 
formulate. We first recall the following standard definition of the fractional powers 
of the generator of the process $Y$.



\begin{definition}\label{def:fracPower}
We write $\CH = L^2(\mu)$ with $\mu$ the invariant measure of $Y$ and 
$\scal{\cdot,\cdot}_\mu$ for its scalar product.
For $\alpha \in (0,1)$, we then say that $f\in \Dom(\CL^{\alpha} )$ if, for 
every $g \in \CH$, the integral
$${1\over \Gamma(-\alpha)}\int_0^\infty t^{-\alpha-1} \<P_tf-f, g\>_{\mu} dt<\infty\;,$$
converges and determines a bounded functional on $\CH$ (which we then call $\CL^\alpha f$).
Recall that the generator of the process $Y$ is $-\CL$, so that $\CL$ is indeed a positive
operator in the reversible case and $\CL^\alpha$ does then coincide with the definition
using functional calculus.

Similarly, for $\alpha \in (-1,0)$, we write $\CL^\alpha$ for the operator given by
$$\CL^\alpha f = {1\over \Gamma(-\alpha)}\int_0^\infty t^{-\alpha-1} P_tf\, dt\;.
$$
Since $t \mapsto t^{-\alpha-1}$ is locally integrable, 
it follows from the first point of Assumption~\ref{simple} below that 
$\CL^\alpha$ is a bounded operator on the 
subspace of $\Lip(\CY)$ consisting of mean zero functions. 
\end{definition}

Assuming that $X^\eps$ takes values in $\R^d$, we then define the $d\times d$ matrix-valued function
\begin{equ}[e:defSigma]
\Sigma(x,\bar x) = {1\over 2}\Gamma(2H+1)\sum_{k=1}^m\int F_k(x,y) \otimes \big(\CL^{1-2H} F_k\big)(\bar x,y)\,\mu(dy)\;,
\end{equ}
where $\CL$ acts on the second argument of $F_k$. As we will see in Remark~\ref{rem:main}, 
the expression  (\ref{e:defSigma})  is naturally interpreted as the limit $\delta \to 0$ of a 
``local'' Green--Kubo formula associated to the fluctuations of \eqref{e:deltaeps}.
 
 Note that the condition $\bar F_k = 0$ is necessary in the case $H > \f12$ since the negative power of $\CL$ appearing in this expression does not make sense otherwise, see also Remark~\ref{rem:fracL} below.
We shall assume mixing conditions and H\"older continuity of the $Y$ variable, see Assumptions~\ref{ass:ergodic}--\ref{ass:integrability} below, as well as a regularity condition on  
$x \mapsto F(x,\cdot)$ (and also $F_0$) as spelled out in Assumption~\ref{ass:C3}. 
A simpler set of conditions is as follows, the first  of which  is a strengthening of Assumptions~\ref{ass:ergodic} and~\ref{ass:integrability}, the second  is a strengthening of Assumption~\ref{ass:continuous}, and the last is just a restatement of Assumption~\ref{ass:C3}
in this context.

 \begin{assumption}[Simplified Assumptions]\label{simple}
 The functions $F_i$ appearing in \eqref{e:SDEbasic} as well as the Markov process $Y$ satisfy the following.
 \begin{enumerate}
 \item 
The Markov semigroup associated to the 
process $Y$ is strongly continuous and has a spectral gap in $\Lip(\CY)$, 
the space of bounded Lipschitz continuous functions on $\CY$.
\item In the case $H < 1/2$ we assume that, for any $\alpha < H$, 
the process $t \mapsto Y_t$ admits $\alpha$-Hölder continuous trajectories 
and its Hölder seminorm (over intervals of length $1$ say) has bounded moments of all orders.
\item When $H > \f12$, we also assume that 
$\int F_i(x,y)\,\mu(dy) = 0$ for every $i \neq 0$ and every~$x$.
\item There exists $\kappa > 0$ such that, for every $i \ge 0$, $x\mapsto F_i(x,\cdot)$ is $\CC^4$ with values in
$\Lip(\CY)$ and its derivatives of order at most $4$ are bounded by $C (1+|x|)^{-\kappa}$
for some $C>0$. 
 \end{enumerate} \end{assumption}

\begin{remark}
Recall that a Markov semigroup $(P_t)_{t \ge 0}$ admits a spectral gap in any given Banach space 
$E \subset L^2(\mu)$ if $P_t\colon E \to E$ is a bounded linear operator for every $t$ and if there exist 
constants $c,C > 0$ such that $\|P_t f - \mu(f)\|_E\le Ce^{-ct}\|f\|_E$ for all $f \in E$. For this 
definition to make sense, $E$ of course needs to contain all constant functions.
\end{remark}

The reason why we are aiming for a more general result at the expense of a much more technical set
of assumptions is that having a spectral gap in $\Lip(\CY)$
is a very restrictive condition which is not even satisfied for the Ornstein--Uhlenbeck process.\footnote{But on the other hand it \textit{is} satisfied for systems with superlinear dissipation. 
This even includes the 
Allen--Cahn equation on the torus in dimension $d \le 3$ driven by space-time white 
noise, as can be deduced from the results in \cite{Michael,Mourrat,Jonathan,Pavlos,Philipp}!}

\begin{theorem}\label{theo:main}
Let $H\in (\f 13, 1)$ and let Assumption~\ref{simple} hold.
For fixed $\eps > 0$, $\alpha < H$, and $T>0$,  the process $X^{\eps,\delta}$ converges in 
law in $\CC^\alpha([0,T])$ as $\delta \to 0$
to a limit $X^\eps$ which we interpret as the solution to \eqref{e:SDEbasic}.

The solution flow of \eqref{e:SDEbasic} converges in law to that of the
Kunita-type stochastic differential equation written in It\^o form as
\begin{equ}[e:KunitaSDE]
dX_t=W(X_t, dt)+ G(X_t)\,dt + \bar F_0(X_t)dt\;,
\end{equ}
where $\bar F_0(x)=\int F_0(x,y)\mu(dy)$, $G_i(x) =  (\d^{(2)}_j\Sigma_{ji})(x,x)$,
 $W$ is a Gaussian random field with correlation
\begin{equ}[e:defWGaussian]
\E  (W_i(x,t)W_j(\bar x,\bar t) )= (t \wedge \bar t) \bigl(\Sigma_{ij}(x,\bar x) + \Sigma_{ji}(\bar x,x)\bigr)\;,
\end{equ}
and where $\d^{(2)}_j$ denotes differentiation in the $j$th direction of the second argument.
\end{theorem}
 \begin{proof}
As already suggested, this is a special case of our main result, Theorem~\ref{theo:mainReal} 
below. The fact that Assumptions~\ref{ass:ergodic}--\ref{ass:integrability} and \ref{ass:C3} are
implied by Assumption~\ref{simple} is immediate. (Take $E_n = \Lip(\CY)$ for every $n$.)
 \end{proof}
 
As a consequence, we also have the following functional CLT.
 
 \begin{corollary}
 Let $H\in (\f 13, 1)$ and let  Assumption~\ref{simple} hold (or let Assumptions~\ref{ass:ergodic}--\ref{ass:integrability} hold
 and when $H > \f12$, let  $\int F_i(y)\,\mu(dy) = 0$ for every $i \geq 1$.)

Then the stochastic process $Z^\eps_t = \sqrt \eps \int_0^{t/\eps} F(Y_r)\, dB_r$ 
 converges to a Wiener process $W$, weakly in  $\CC^\alpha([0,T])$ for any ${\alpha<\f 12\wedge H}$.
Furthermore, defining the random smooth function 
$Z^{\eps,\delta}_t = \sqrt \eps \int_0^{t/\eps} F(Y_r)\, dB_r^\delta$ with $B^\delta = \phi_\delta * B$, its iterated integral satisfies
$$
\lim_{\eps \to 0}\lim_{\delta \to 0} \int_s^t \bigl(Z^{\eps,\delta}_r-Z^{\eps,\delta}_s\bigr)\otimes dZ^{\eps,\delta}_r =  \int_s^t  (W_r-W_s)\otimes dW_r+\Sigma (t-s)\;,
$$
where the matrix $\Sigma$ is given by \eqref{e:defSigma} (which is independent of $x,\bar x$ 
in this case).
\end{corollary}

\begin{remark}\label{rem:main}
\Cref{theo:main}
characterises $\lim_{\eps \to 0}\lim_{\delta \to 0}X^{\eps,\delta}$
and shows that it is a Markov process with generator $\A$ given by 
\begin{equ}[e:defLimit]
(\A g)(x) = \sum_{i,j=1}^d \d_j \big(\Sigma_{ji}(x,\cdot)\d_i g\big)(x) + \sum_{i=1}^d\bar F_0^i(x)\d_i g(x) \;.
\end{equ}
Our proof actually carries over with minor modifications to the case when $\eps\to 0$ for 
fixed $\delta$ (but with convergence bounds that are uniform in $\delta$!), in which case the limit 
is given by the same expression \eqref{e:KunitaSDE},
but with the matrix $\Sigma$ given by
\begin{equ}[e:generalSigma]
\Sigma_\delta(x,\bar x) = \sum_{k=1}^m \int_0^\infty R_\delta(t) \int F_k(x,y) \otimes \big(P_t F_k\big)(\bar x,y)\,\mu(dy)\,dt\;,
\end{equ}
where $R_\delta(t) = \delta^{2H-2} \E v(0)v(t/\delta)$ and $P_t = e^{-\CL t}$ denotes  the Markov semigroup for $Y$. We will derive this formula in Section~\ref{sec:Heuristics} where we will also see that, 
for frozen values of $x$, it is a special case of 
the Green--Kubo formula \cite{Kubo,PapaKo,KestenPap}. Note that 
\eqref{e:KunitaSDE}--\eqref{e:generalSigma} (in particular the convergence of the flow)
is also consistent with \cite[Theorem~4.3]{Gautam} where a somewhat analogous
situation is considered.
It follows from \Cref{def:fracPower} that $\Sigma_\delta \to \Sigma$ as
$\delta \to 0$, so that the two limits commute (in law).
\end{remark}

\begin{remark}\label{rem:refs} 
There has recently been a surge in interest in the study of slow\slash fast systems involving
fractional Brownian motion. We already mentioned the averaging result \cite{AveragingHigh} which
considers the case $H > \f12$ but with $\bar F \neq 0$. The work \cite{Pei-Inahama-Xu} considers the case $H \in (\f13,1)$ like the present article,
but with the very strong assumption that $F$ is independent of the fast variable, in which case 
only $F_0$ exhibits rapid fluctuations and one essentially recovers classical averaging results.
 In \cite{Bourguin-Gailus-Spiliopoulos-typical}, the authors consider the case $H > \f12$, but 
with $F$ independent of the slow variable $x$ and, as in \cite{AveragingHigh}, not necessarily 
averaging to zero.
They obtain a description of the fluctuations for (a generalisation of) such systems in the regime where there is an 
additional small parameter in front of $F$.
\end{remark}
  
Formula (\ref{e:defSigma}) holds for the continuum of parameters $H\in (\f 13, 1)$. There are two
special cases that were previously known. The case $H = \f12$ reduces of course to the
classical stochastic averaging results \cite{Stratonovich-rhs,Has68,Freidlin-76,Skorohod} which
state that the generator of the limiting diffusion is obtained by averaging the generator for the
slow diffusion with the $x$ variable frozen against the invariant measure for the fast process.
Note that for this to match \eqref{e:defLimit} one needs to interpret the stochastic integral in
\eqref{e:SDEbasic} in the Stratonovich sense. This is natural given that this is the interpretation that
one obtains when replacing $B$ by a smooth approximation, which is consistent with Remark~\ref{rem:main}. 
The fact that one also has convergence of flows however (in the case without feedback considered here) appears to be 
new even in this case.

Another set of closely related classical results deals with ``time homogenisation'',
also known as the Kramers--Smoluchowski limit or diffusion creation \cite{Kramers,PapaKo}.
There, one considers random ODEs of the type
\begin{equ}[e:homog]
	\frac{dX_t^\epsilon}{dt}=\f 1 {\sqrt \epsilon} F(X_t^\epsilon, Y_t^\epsilon)
	+ F_0(X_t^\epsilon, Y_t^\epsilon)\;, 
\end{equ}
with
 $F$ averaging to zero against the stationary measure $\mu$ for the fast process $Y$.  
In this case, one also obtains a Markov process in the limit $\eps \to 0$ and its generator
coincides with \eqref{e:defLimit} if one sets $H=1$.
This can be understood by noting that, at least formally, fractional Brownian motion with
Hurst parameter $H=1$ is given by $B(t) = ct$ with $c$ a normal random variable,
so that \eqref{e:SDEbasic} reduces to \eqref{e:homog}, except for the random constant $c$,
which then appears quadratically in \eqref{e:defLimit} and therefore disappears when
averaged out.

The standard proofs of averaging \slash homogenisation results found in the literature
tend to fall roughly into two groups. The first contains functional analytic proofs based on
general methods for studying singular limits of the form $\exp(t \CL_\eps)$ for
$\CL_\eps = \eps^{-1} \CL_0 + \CL_1$. This of course requires the full process
(slow plus fast) to be Markovian and completely breaks down in our situation. 
The second group consists of more probabilistic arguments, which typically rely on
using corrector techniques to construct sufficiently many martingales to be able to
exploit the well-posedness of the martingale problem for the limiting Markov process.
The latter are in principle more promising in our situation since the limiting process is still
Markovian, but the lack of Markov property makes it unclear how to construct martingales 
from our process. (But see \Cref{subsection-convergence} for a construction that does
go in this direction.)

Instead, our proof relies on rough paths theory \cite{Lyons,Book}, which has recently 
been used to recover homogenisation
results (formally corresponding to the case $H=1$), for example in \cite{KellyMelb}. See also 
\cite{Bail-Cat, Chevyrev-Friz-Korepanov-Melbourne-Zhang,Deu-Ore-Per, Friz-Gassiat-Lyons} for more recent results with a similar flavour.
In the case when the \textit{fast} dynamics is non-Markovian and solves an equation driven by a fractional Brownian motion,
a collection of homogenisation results were obtained in \cite{Gehringer-Li-homo, Gehringer-Li-tagged, Gehringer-Li-fOU, Gehringer-Li-Sieber}, while stochastic averaging results with non-Markovian fast motions are  obtained 
in \cite{Li-Sieber, Li-Sieber-2} for the case $H>\f 12$. The former group of results are proved 
using rough path techniques, but there is of course an extensive literature on functional limit 
theorems based on  either central 
or non-central limit theorems, see for example \cite{Bai-Taqqu, BenHariz, Breuer-Major,  Maejima-Ciprian, Dobrushin-Major,Pipiras-Taqqu, Rosenblatt1956}.

Finally,  note that many physical systems can be regarded as a slow \slash fast systems, this includes second order Langevin 
equations and tagged particles in a turbulent  random field \cite{
 Cogburn-Hersh, PapaKo,   KestenPap,   Komorowski-Novikov-Ryzhik-14, Sethuraman-Varadhan-Yau,  Birrell-Hottovy-Volpe-Wehr, Fannjiang-Komorowski}. 
They also arise in the context of perturbed completely integrable Hamiltonian systems \cite{FW, averaging} and geometric stochastic systems  \cite{geodesics,  LB, GIR,  Perruchaud}. See also \cite{Kurtz92, E,Veretennikov} for some review articles \slash monographs.


%
%

\begin{remark}
It may be surprising that, when $H < \f12$, even though $X_\eps$ is driven by a \textit{fractional} Brownian motion
and $F(x,y)$ isn't assumed to be centred in the $y$ variable,
the limit $\bar X$ is a regular diffusion. This is unlike the case $H > \f12$ \cite{AveragingHigh, Li-Sieber}
where a non-centred $F$ leads to an averaging result with a process driven by fBm in the limit.
 This change in behaviour can be understood heuristically as follows. With $\eta$ as in \eqref{e:defEta},
the covariance of $t\mapsto f(Y_t^\eps) \dot B_t$ is given by 
$\zeta_\eps(t-s) = \eta ''(t-s) g((t-s)/\eps)$ for some  function $g(t)=\<f, P_{t}f\>$,
 that would typically converge quite
fast to a non-zero limit. The scaling properties of $\eta$ then show that 
\begin{equ}
\int_\R \eta ''(t) g(t/\eps)\,dt = \eps^{2H-1} \int_\R \eta ''(t) g(t)\,dt = C_g \eps^{2H-1}\;,
\end{equ}
for some constant $C_g$ which has no reason to vanish in general. As a consequence, $\eps^{1-2H} \zeta_\eps$
converges pointwise to $0$ while its integral remains constant, suggesting that $\eps^{\f12-H} f(Y_t^\eps) \dot B_t$ 
indeed converges to a white noise. When $H > \f12$ however, $\eta''$ is not absolutely integrable at infinity
and one needs to assume that $g$ vanishes there, which leads to a centering condition.
A similar transition from diffusive to super-diffusive behaviour at $H = \f12$ was observed
in a different context in \cite{Komorowski-Novikov-Ryzhik-14}.
\end{remark}

\begin{remark} 
As explained, our result implies more,  namely that the (random) flow induced by
the SDE \eqref{e:SDEbasic} converges in law to that induced by the Kunita-type SDE \cite{Kunita}
\begin{equ}[e:limitKunita]
dx_i = W_i(x,dt) +  (\d^{(2)}_j\Sigma_{ji})(x,x)\,dt + \bar F_0^i(x)\,dt\;.
\end{equ}
In other words, the flows $\psi_{s,t}^\epsilon$, where  $\psi_{s,t}^\eps (x)$ denotes the solution at time $t$ to the 
$x$-component of (\ref{e:SDEbasic}) with initial condition $x$ at time $s$, converge to a 
limit $\psi_{s,t}$ which is Markovian in the sense that $\psi_{s,t}$ and $\psi_{u,v}$ are independent
whenever $[s,t) \cap [u,v) = \emptyset$. 
This remark appears to be novel even when $H=\f12$, but it is unclear
whether it extends to the case when $x$ feeds back into the dynamic of $y$ as in \eqref{e:startingSDE}.
\end{remark}

\begin{remark}\label{rem:Strat}
The term $\d^{(2)}_j\Sigma_{ji}$ appearing in \eqref{e:limitKunita} looks ``almost'' like an
Itô-Strato\-no\-vich correction.
In fact, when $\CL$ is self-adjoint on $L^2(\mu)$, one has
$\Sigma_{ij}(x,\bar x) = \Sigma_{ji}(\bar x,x)$ in which case \eqref{e:limitKunita}
is equivalent to $dx_i = W_i(x,{\circ}\, dt)+ \bar F_0^i(x)\,dt$.
\end{remark}

%

\subsection{Heuristics for general slow\slash fast random ODEs} 
\label{sec:Heuristics}

We now show how to heuristically derive \eqref{e:generalSigma}.
Consider a random ODE of the form 
\begin{equ}[e:genODE]
	\frac{dX_t^\epsilon}{dt}=\f 1 {\sqrt \epsilon} \hat F(X_t^\epsilon, Z_t^\epsilon)\;, 
\end{equ}
where $Z_t^\eps = Z(t/\eps)$ for some stationary (but not necessarily Markovian!)
stochastic process $Z$ and $\hat F(x,\cdot)$ is assumed to be centred with respect to the
stationary measure of $Z$. In the case when $\hat F(x,z) = \hat F(z)$ does not depend on $x$, 
it follows from the Green--Kubo formula \cite{Kubo,PapaKo,KestenPap} that, at least 
when $Z$ has sufficiently nice mixing properties, $X^\eps$ converges as $\eps \to 0$ 
to a Wiener process with covariance $\Sigma + \Sigma^\top$, where
\begin{equ}
\Sigma = \int_0^\infty \E\big( \hat F(Z_0)\otimes \hat F(Z_t)\big)\,dt\;.
\end{equ}
This suggests that a natural quantity to consider in the general case is
\begin{equ}[e:SigmaGen]
\Sigma(x,\bar x) = \int_0^\infty \E \big(\hat F(x,Z_0)\otimes \hat F(\bar x,Z_t)\big)\,dt\;,
\end{equ}
and that the limit of $X^\eps$ as $\eps \to 0$ is a diffusion with generator
of the form 
\begin{equ}[e:Agen]
\big(\CA g\big)(x) = \Sigma_{ij}(x,x)\d^2_{ij} g(x) + b_i(x)\d_i g(x)\;,
\end{equ}
for some drift term $b$. 

To derive the correct expression for the drift $b$, we note that one expects
\begin{equ}
\E \big(X^\eps_{t+\delta t} - X^\eps_t\,|\, \CF_t\big)
\approx \delta t\, b(X^\eps_t)\;,
\end{equ}
in the regime $\eps \ll \delta t \ll 1$. The left-hand side of this expression is given by
\begin{equ}[e:LHS]
\f 1 {\sqrt \epsilon} \int_t^{t+\delta t} \E \big(\hat F(X_s^\epsilon, Z_s^\epsilon)\,|\,\CF_t\big)\,ds\;.
\end{equ}
To lowest order, one can approximate this expression by replacing $X_s^\epsilon$ by $X_t^\epsilon$,
but the resulting expression vanishes rapidly for $s \gtrsim t+\eps$ due to the 
centering condition on $\hat F$.
To the next order, one has
\begin{equs}
\E \big(\hat F(X_s^\epsilon&, Z_s^\epsilon)\,|\,\CF_t\big)
\approx \E\Big(\hat F\Big(X_t^\epsilon + \f1{\sqrt\eps}\int_t^s \hat F(X_t^\epsilon, Z_r^\epsilon)\,dr , Z_s^\epsilon\Big)\,\Big|\,\CF_t\Big)\label{e:Taylor} \\
&\approx \E\big(\hat F(X_t^\epsilon, Z_s^\eps)\,|\,\CF_t\big) + \f1{\sqrt\eps}\int_t^s \E \big(D\hat F(X_t^\epsilon,Z_s^\epsilon) \hat F(X_t^\epsilon, Z_r^\epsilon)\,|\,\CF_t\big)\,dr\\
&\approx \sqrt\eps\int_{0}^\infty \E \big(D\hat F(X_t^\epsilon,Z_u) \hat F(X_t^\epsilon, Z_0)\,|\,\CF_t\big)\,du\;,
\end{equs}
where the last identity follows from the substitution $u = (s-r)/\eps$ combined with
the fact that, provided that $Z$ is sufficiently rapidly mixing, 
we expect the main contribution from
this integral to come from $|u| \approx 1$, while typical values of 
$s$ are such that $(s-t)/\eps \approx \delta t/\eps \gg 1$.
Combining this with \eqref{e:LHS} eventually yields the expression
\begin{equ}
b(x) = \int_0^\infty D\hat F(x,Z_s) \hat F(x, Z_0)\,ds\;.
\end{equ}
Comparing this with \eqref{e:SigmaGen}, we conclude that
\begin{equ}
b_i(x) = \big(\d_j\Sigma_{ji}(x,\cdot)\big)(x)\;,
\end{equ}
(summation over repeated indices is implied) so that \eqref{e:Agen} can be written as
\begin{equ}
\big(\CA g\big)(x) = \d_j \big(\Sigma_{ji}(x,\cdot)\d_i g\big)(x)\;,
\end{equ}
which does coincide with the expression \eqref{e:defLimit} as desired.

In order to link this calculation with the setting of the previous section, 
we note that \eqref{e:deltaeps} (with $F_0 = 0$ for simplicity) can be coerced 
into the form \eqref{e:genODE} by setting
$Z_t = (\delta^{H-1}v(t/\delta), Y_t)$ as well as $\hat F(x,(v,y)) = F(x,y) v$. In this case, one has
\begin{equs}
\E \big(\hat F(x,Z_0)\otimes \hat F(\bar x,Z_t)\big)
&= \delta^{2H-2} R(t/\delta)\sum_{k=1}^m\E \big(F_k(x,Y_0)\otimes F_k(\bar x,Y_t)\big)\\
&= R_\delta(t)\sum_{k=1}^m\int F_k(x,y)\otimes \big(P_t F_k\big)(\bar x,y)\,\mu(dy)\;,
\end{equs}
so that one does indeed recover the expression \eqref{e:generalSigma} for any
fixed $\delta$.

\begin{remark}
The eagle-eyed reader will have spotted that since the stationary measure of $Z$ is
$\CN(0,C) \otimes \mu$ for some multiple $C$ of the identity matrix and since $\hat F(x,(v,y))$
is linear in $v$, the centering condition for $\hat F$ is \textit{always} satisfied, independently of
the choice of $F$. This explains why our main result does not require any centering condition
when $H \le \f12$. When $H > \f12$ however, the covariance function $R$ decays too slowly for
the heuristic derivation just given to apply. The centering condition for $F$ then guarantees
that correlations decay sufficiently fast to justify the second step in \eqref{e:Taylor}.
\end{remark}

The remainder of this article is structured as follows.  In Section~\ref{sec:precise} 
we introduce the assumptions on the nonlinearities $F_i$ as well as the fast process $Y$, 
we discuss a few examples, and we give provide the statement of our main result.
In Section~\ref{sec:convSmooth}
we then show that solutions to \eqref{e:deltaeps} converge as $\delta \to 0$, which yields in particular
a precise interpretation of what we mean by \eqref{e:SDEbasic} when $H < \f12$. The strategy of
proof is as follows. Given a smooth mollification $B^\delta$ of $B$,
we first show convergence of $\int_s^t \int_s^r   f(u)\, \dot B^\delta(u) du\, g(r)\, \dot B^\delta(r)dr$
as $\delta \to 0$
for any deterministic $H$-H\"older continuous functions $f,g$. While we are able to 
reduce this to existing criteria
for canonical rough path lifts of Gaussian processes \cite{FVGaussian,Coutin-Qian} 
in the case where the two fractional Brownian motions appearing in 
this expression are independent, the case where they are equal requires a bit
more care and relies on a simple trick given in Proposition~\ref{prop:contractWiener}, which is of independent interest.
This then allows us to build an infinite-dimensional rough path $\Z^\eps$  (taking values in a space of vector fields
on $\R^d$) associated to \eqref{e:SDEbasic} in a similar way as in \cite[Sec.~1.5]{KellyMelb} (see also
the ``nonlinear rough paths'' of \cite{NonLin} and \cite{Gehringer-Li-Sieber}) and to reformulate 
\eqref{e:SDEbasic} as an RDE driven by $\Z^\eps$ with 
nonlinearity given by point evaluation. Section~\ref{sec:RD} provides details of the construction of $\Z^\eps$,
while Section~\ref{sec:formMain} then uses it to 
formulate our main technical result, namely Theorem~\ref{theo:mainRP} which shows that $\Z^\eps$
converges to a certain rough path lift of an infinite-dimensional Wiener process 
with covariance function given by $\Sigma$. The remainder of the article is devoted to the
proof of this convergence statement.
Section~\ref{sec:tight} shows tightness of the family $\{\Z^\eps\}_{\eps \le 1}$,
while we identify its limit in Section~\ref{sec:limit}. In both sections, the cases $H < \f12$ and
$H > \f12$ are treated in a completely different way.

The fact that we have convergence of the full 
infinite-dimensional rough path allows us to conclude that we do not just have convergence of solutions for fixed initial
conditions, but of the full solution flow.
One point of note is that there are two separate sources
of randomness, namely the Markov process $Y$ and the fractional Brownian motion $B$. Our convergence result
is ``annealed'' in the sense that our convergence in law requires both sources, but a number of intermediate results 
are ``quenched'' in the sense that they hold for almost every realisation of $Y$. It is an open question whether
our final convergence result also holds in the quenched sense.

\subsection*{Acknowledgements}

{\small
XML acknowledges partial support from the EPSRC (EP/S023925/1 and EP/V026100/1), while MH gratefully acknowledges support from the Royal Society through a research professorship.
}


\section{Precise formulation and results}
\label{sec:precise}

In this section, we collect the precise assumptions on the functions $F_i$ as well as
the Markov process $Y$. 

\noindent{\bf Convention.}
We write $A \lesssim B$ as shorthand for $A \le KB$ with a constant $K$
that will differ from statement to statement.

\subsection{Technical assumptions on the fast variable \TitleEquation{Y}{Y}}
\label{sec:processY}
 
Throughout the article we fix $H \in (\f13,1)$ as well as a sequence $(E_n)_{n \ge 0}$ of Banach spaces
such that $E_n \subset E_{n+1}$ and $E_n \subset L^1(\CY,\mu)$ for every $n \ge 0$, and such that
pointwise multiplication is a continuous operation from $E_0 \times E_n$ into $E_{n+1}$ for every $n\ge 0$.
We also write simply $E$ instead of $E_0$ and assume $E$ contains constant functions.
See Section~\ref{sec:class} below for two classes of examples showing what type of spaces
we have in mind here.

 First, we impose that $Y$ has ``nice'' ergodic properties in the following sense, which in particular
implies that $\mu$ is its unique invariant measure on $\CY$.
\begin{assumption}\label{ass:ergodic}
Let $N = \infty$ for $H > \f12$ and $N=2$ for $H \in (\f13,\f12]$.
For every $n \in [1,N)$, the semigroup $P_t$ extends to a strongly continuous semigroup on $E_n$ and there
exist constants $C$ and $c> 0$ (possibly depending on $n$) such that, 
for every $f \in E_n$ with $\int_\CY f d\mu=0$, one has
\begin{equ}[e:spectralGap]
\|P_t f\|_{E_n} \le C e^{-ct} \|f\|_{E_n}\;.
\end{equ}
\end{assumption}

In the low regularity case, we also assume that the process $Y$ has some sample path continuity
when composed with a function in $E_2$.
\begin{assumption}\label{ass:continuous}
For $H \in (\f13,\f12)$ there exists $p_\star > \max\{4d,12/(3H-1)\}$ such that 
for every $f \in E_2$
\begin{equ}[e:boundContY]
\|f(Y_t)-f(Y_0)\|_{L^{p_\star}} \le c \|f\|_{E_2}(t^H \wedge 1) \,\qquad \forall t \ge 0\;,
\end{equ}
for some constant $c > 0$.
\end{assumption}

We also need some integrability.

\begin{assumption}\label{ass:integrability}
For $H \ge \f12$, one has $E_n \subset L^2(\CY,\mu)$ for every $n \ge 0$. For
$H < \f12$, one has $E_2 \subset L^2(\CY,\mu)$ and $E \subset L^{p_\star}(\CY,\mu)$.
\end{assumption}
\begin{remark}
When combining it with the inclusion of the product, Assumption~\ref{ass:integrability}  implies that $E \subset \bigcap_{p \ge 1} L^p(\CY,\mu)$  for $H\ge \f 12$.
\end{remark}
Another consequence of these assumptions is as follows. 

\begin{remark}\label{rem:fracL}
As a consequence of Assumption~\ref{ass:continuous}, we conclude that if $f \in E_2$ and $H < \f12$, then
\begin{equs}
\|P_t f - f\|_\mu^2 &= \E \bigl|\E \bigl(f(Y_t) - f(Y_0)\,|\,Y_0\bigr)\bigr|^2 
\le \E |f(Y_t) - f(Y_0)|^2\\ &\le \|f\|_{E_2}^2  \big(t^{2H} \wedge 1\bigr) \;.
\end{equs}
Recalling the definition of $\CL^\alpha$ from Definition~\ref{def:fracPower}, it
 follows that $E_2 \subset \Dom(\CL^\alpha)$ for every $\alpha < H$ so that
\eqref{e:defSigma} is indeed well defined provided that $F_k(x,\cdot) \in E_2$ for
every $x$. This will be guaranteed by Assumption~\ref{ass:C3} below.
\end{remark}

\subsection{Examples of fast variables }
\label{sec:class}

One possible concrete framework is as follows. Fix  two weights $V\colon \CY \to [1,\infty]$ and $W \colon \CY \to (0,\infty)$ 
and a metric $d$ on $\CY$ generating its topology
with the property that there exists $C>0$ such that, for all $x,y \in \CY$ with 
$d(x,y) \le 1$, one has
\begin{equ}[e:boundRatio]
V(x) \le C V(y)\;,\qquad W(x) \le C W(y)\;.
\end{equ}
For $n \ge 1$, we then let $\CB_{V,W}$ be the Banach space of functions 
$f \colon \CY \to \R$ such that $$\|f\|_{V,W} \eqdef \sup_{x \in \CY}{|f(x)|\over V(x)} + \sup_{x,y \in \CY \atop d(x,y) \le 1}{|f(x) - f(y)|\over d(x,y) W(x) V(x)} < \infty\;.$$
One choice of scale of function spaces that is suitable for a large class of Markov processes is
to take $E_n = \CB_{V,W}$ for every $n \ge 1$ (and suitably chosen $V$ and $W$), while 
$E_0$ is chosen be the space of bounded Lipschitz continuous functions, namely $\CB_{1,1}$.

This framework is relatively general since it allows for a wide variety of choices of 
$V$, $W$, and of distance functions on $\CY$, see \cite{Michael,Andy}. 
For example, it was shown in \cite[Thm.~1.4]{HM08} that 
the 2D stochastic Navier--Stokes equations exhibit a spectral gap in such spaces 
under extremely weak conditions on the driving noise.
More precisely, for every $\eta$ small enough there exist constants $C$ and $\gamma$ such that
$$\Big\|P_tf-\int f d\mu\Big\|_\eta\le C e^{-\gamma t}\|f\|_\eta\;,$$
for every Fr\'echet differentiable function  $f$ for every $t\ge 0$, where
\begin{equ}[e:normJonathan]
\|f\|_\eta = \sup_x e^{-\eta |x|^2} \bigl(|f(x)| + |Df(x)|\bigr)\;.
\end{equ}
This at first sight appears to fall outside our framework, but one notices that if
one sets
\begin{equ}[e:funnyd]
d(x,y) = \inf_{\gamma: x \to y} \int_0^1 (1+|\gamma(t)|)|\dot\gamma(t)|\,dt\;,
\end{equ} 
then the norm $\|\cdot\|_{V,W}$ with $V(x) = \exp(\eta |x|^2)$ and $W(x) = 1/(1+|x|)$ is equivalent
to the norm \eqref{e:normJonathan}. The reason for the choice of $d$ as
in \eqref{e:funnyd}, which is then ``undone'' by our choice of $W$, is to guarantee that \eqref{e:boundRatio} holds for $V$,
which would not be the case for $|x-y| \le 1$ in the Euclidean distance.

To verify Assumption~\ref{ass:continuous} one can then for example make use of the following.
\begin{lemma}
Suppose that $\int \big(V(x)(1+W(x))\big)^{p_\star}\,\mu(dx) < \infty$  
 and there exists a constant $c$ such that, for some $\alpha_0>1-2H$,
  \begin{equ}[e:boundContY1]
\|d(Y_t, Y_0)\|_{L^{p_\star}} \le c\bigl(t^{\alpha_0} \wedge 1\bigr)\,\qquad \forall t \ge 0\;.
\end{equ} 
Let $f \in \CB_{V,W}$ and  $2p \le p_\star$, then
$$\|f(Y_t) - f(Y_0)\|_{L^p}\lesssim \|f\|_{V,W}(1\wedge t^{\alpha_0})\;.$$
In particular, on any fixed time interval, we have $\E \|f(Y_\cdot)\|_{\CC^\alpha}^p < \infty$
provided that $p(\alpha_0-\alpha) > 1$.
\end{lemma}
\begin{proof}
For $f \in E$ with $\|f\|_E \le 1$ and for $p \le \f 12 p_\star$, one has
\begin{equs}
\big\|f(Y_t) - f(Y_0)\|_{L^p} &\le \|W(Y_0) V(Y_0) d(Y_t,Y_0) +\one_{d(Y_t,Y_0) > 1}\|f\|_E \big(V(Y_0)+V(Y_t)\big)\big\|_{L^p} \\
&\lesssim \|(1+W(Y_0)) V(Y_0) \|_{L^{2p}} \Bigl((1\wedge t^{\alpha}) + \P(d(Y_t,Y_0) > 1)^{\f1{2p}}\Bigr) \\
&\lesssim 1\wedge t^{\alpha}\;,
\end{equs}
where we combined \eqref{e:boundContY1} with Markov's inequality in the last step.
\end{proof}

When $H \ge \f12$, Assumption~\ref{ass:continuous} is empty, so only integrability conditions 
are required on the spaces $E_n$. This allows for example to use Harris's theorem \cite{HarrisOrig,MeynTweedie,Harris} to verify
Assumption~\ref{ass:ergodic} for spaces of functions with weighted supremum norms.
More precisely, one would then take $E$ to be the space of all bounded Borel measurable 
functions and $E_n = \CB_V$, the Banach space of functions 
$f \colon \CY \to \R$ such that $$\|f\|_{V} \eqdef \sup_{x \in \CY}{|f(x)|\over V(x)} < \infty\;.$$
In order to verify our assumptions, it then suffices that $V$ is 
a square integrable Lyapunov function for the Markov process $Y$ and that 
the sublevel sets of $V$ satisfy a `small set' condition for the transition probabilities
of $Y$ \cite{MeynTweedie}.

\subsection{Main results}

One final assumption we need is
that the nonlinearities $F$ and $F_0$ appearing in \eqref{e:SDEbasic} are 
sufficiently nice $E$-valued functions of their first argument. More precisely, we assume the following.

\begin{assumption}\label{ass:C3}
The map $x \mapsto F(x,\cdot)$ is of class $\CC^4$ with values in $E$ and there exists
an exponent $\kappa > \f{16 d}{p_\star}$ with $p_\star$ as in \Cref{ass:continuous} (and simply $\kappa > 0$ when $H > \f12$) such that, for every multi-index $\ell$ of length at most $4$, 
$$\|D_x^\ell F(x,\cdot)\|_{E} \lesssim (1+|x|)^{-\kappa}.$$
The same is assumed to hold true for $F_0$.
When $H > \f12$, we further assume that 
$\int F_i(x,y)\,\mu(dy) = 0$ for every $i \neq 0$ and every~$x$.
\end{assumption}

The condition $F \in \CC^4$ is of course suboptimal and could probably be 
lowered to $F \in \CC^\beta$ for $\beta > \max\{H^{-1},2\}$ and $F_0 \in \CC^\beta$
for $\beta > 1$, at least if enough integrability is assumed in Assumption~\ref{ass:integrability}.
We also now fix  a Schwartz function $\phi$ integrating to $1$ and
set $\phi_\delta(t)=\f 1\delta \phi(t/ \delta)$.
We then write $B^\delta$ for the convolution of $B$ with this mollifier, namely
$$B^\delta(t) = \f 1\delta (\phi_\delta* B)(t)= \f 1\delta \int_\R  \phi\Big(\f{t-s} \delta \Big) B(s)\, ds\;.$$
With this notation, the solutions to \eqref{e:deltaeps} are equal in law to the process given by
\begin{equ}[e:deltaeps2] 
\dot X^{\eps,\delta} _t = \eps^{\f12-H} F(X_t^{\eps,\delta},Y_t^\eps)\,\dot B^{\eps\delta} + F_0(X_t^{\eps,\delta},Y_t^\epsilon)\;.
\end{equ}
Since $B^{\eps\delta}$ is smooth, this equation should be interpreted as an ordinary differential
equation that just happens to have random coefficients. With all these preliminaries at hand, our
main result is the following.

\begin{theorem}\label{theo:mainReal}
For $H \in (\f13,1]$ and under Assumptions~\ref{ass:ergodic}--\ref{ass:integrability}, 
and~\ref{ass:C3}, the conclusions of Theorem~\ref{theo:main} hold. 
With $X^{\eps,\delta}$ defined in \eqref{e:deltaeps2}, the convergence $X^{\eps,\delta} \to X^\eps$
furthermore holds in probability.
\end{theorem}

\begin{proof}
The convergence in probability of the flow 
$X^{\eps,\delta} \to X^{\eps}$ is the content of \Cref{prop:convDeltaConsequence} below.
The proof of the conclusion of Theorem~\ref{theo:main}, namely the convergence in law of the 
flow for \eqref{e:SDEbasic} as $\eps \to 0$ is the content of \Cref{cor:main}.
\end{proof}

\section{Convergence of smooth approximations}
\label{sec:convSmooth}


We first address the question of the convergence in probability of solutions to \eqref{e:deltaeps}
to those of \eqref{e:SDEbasic} as $\delta \to 0$ for $\eps > 0$ fixed.
In fact, we will directly provide an interpretation of \eqref{e:SDEbasic} and show that this
interpretation is sufficiently stable to allow for the approximation \eqref{e:deltaeps}.

Our convergence proof relies on the theory of rough paths; we refer to \cite{Book}
for an introduction. The main insight of this theory is that even though, for $H \le \f12$, the solution map $B \mapsto X$
for equations of the type \eqref{e:SDEbasic} isn't continuous when viewing $B$ as an element of
any classical function space large enough to contain typical sample paths of fractional Brownian motion,
it does become continuous when enhancing $B$ with its iterated integrals $\bbB = \int B\otimes dB$
and endowing the space of pairs $(B,\bbB)$ with a suitable topology.

For this, consider for any $x \in \R^d$ the processes
\begin{equ}[e:defZ]
Z^{\eps,\delta}_{s,t}(x) = \epsilon^{\f 12-H}\int_s^t F(x,Y^\eps_r)\,dB^\delta(r)\;,\quad \bar Z^{\eps}_{s,t}(x) = \int_s^t F_0(x,Y^\eps_r)\,dr\;.
\end{equ}
Here, the first integral is interpreted as a Wiener integral which makes sense also when $\delta = 0$
and, when $\delta > 0$, coincides with the Riemann--Stieltjes integral.
Recall that the Wiener integral of a deterministic (or independent) integrand against any Gaussian
process $B$ is well-defined provided that the integrand belongs to the reproducing kernel Hilbert space 
$\CH_B$ of $B$ and provides an isometric embedding $\CH_B \ni f \mapsto \int f\,dB \in L^2(\Omega,\P)$.  
In the case of fractional Brownian motion, it is known that 
$L^2 \subset \CH_B$ when $H \ge \f12$ while for $H < \f12$ one has $\CC^\alpha \subset \CH_B$ for
every $\alpha > \f12-H$. The fact that for fixed $x$ and $\eps > 0$, 
$t \mapsto F(x,Y^\eps_t)$ belongs to $\CH_B$ for all $H > \f13$ is then a simple
consequence of Assumptions~\ref{ass:continuous} and~\ref{ass:integrability} combined with
Kolmogorov's continuity criterion (when $H < \f12$).

Write $\CB = \CC_b^3(\R^d,\R^d)$ and $\CB_k = \CC_b^3(\R^{d\cdot k},(\R^{d})^{\otimes k})$,   so that 
one has canonical inclusions of the algebraic tensor product $\CB_k \otimes_0 \CB_\ell \subset \CB_{k+\ell}$ with
the usual identification  $(f\otimes g)(x,y)= f(x) g(y)$ thanks to the fact that
$\|f\otimes g\|_{\CB_{k+\ell}}\le \|f\|_{\CB_\ell }\|g\|_{\CB_k}$.
Given a final time $T>0$ and $\alpha \in (\f13,\f12)$, we define the space 
$\Cr^\alpha([0,T], \CB \oplus \CB_2)$ of $\alpha$-Hölder rough paths in the
usual way \cite[Def.~2.1]{Book}, but with all norms of level-$2$ objects in $\CB_2$.
Recall that an $\alpha$-Hölder rough path $(X, \XX)$ is a pair of functions where  $X\in \Cr^\alpha([0,T], \CB)$
with $X_0=0$ and $\XX: \Delta_T \to \CB_2$ where $\Delta_T :=\{(s,t): 0\le s\le t \le T\}$ is the two-simplex  with
$|\XX|_{2 \alpha}:=\sup_{(s,t)\in\Delta_T}\frac{\|\XX_{s,t}\|_{\CB_2}}{|t-s|^{2 \alpha}}<\infty$.
In addition, Chen's relation is imposed, namely $\XX_{s,t}-\XX_{s,u}-\XX_{u,t}=X_{s,u}\otimes X_{u,t}$.

We define the second-order processes $\ZZ^{\eps,\delta}$ and $\bar \ZZ^\eps$ by
\begin{equ}
\ZZ^{\eps,\delta}_{s,t}(x,\bar x) = 
\int_s^t Z^{\eps,\delta}_{s,r}(x)\, dZ^{\eps,\delta}_{s,r}(\bar x)\;,\qquad
\bar\ZZ^{\eps}_{s,t}(x,\bar x) = 
\int_s^t \bar Z^{\eps}_{s,r}(x)\, d\bar Z^{\eps}_{s,r}(\bar x)\;,
\end{equ}
(the differentials are taken in the $r$ variable) and we define $\Z^{\eps,\delta} = (Z^{\eps,\delta},\ZZ^{\eps,\delta})$, $\bar \Z^\eps = (\bar Z^\eps,\bar \ZZ^\eps)$. Note here that 
$r \mapsto Z_{s,r}^{\eps,\delta}(x)$ is smooth and $r \mapsto \bar Z_{s,r}^{\eps}(x)$
is Hölder continuous  for any exponent strictly less than $1$, so these integrals
should be interpreted as regular Riemann--Stieltjes integrals.
In Section~\ref{sec:RD} below we will give a proof of the following result.

\begin{proposition}\label{prop:convDelta}
Let $H \in (\f13,1]$, let $\alpha \in (\f13,H\wedge \f12)$ and $\beta \in (1-\alpha,1)$, and let
Assumptions~\ref{ass:ergodic}--\ref{ass:integrability}, 
and~\ref{ass:C3} hold. Then, 
$\Z^{\epsilon,\delta}$ and $\bar \Z^\eps$ admit versions that are random elements in 
$\Cr^\alpha([0,T], \CB \oplus \CB_2)$ and $\Cr^\beta([0,T], \CB \oplus \CB_2)$ respectively.
Furthermore, $\Z^{\epsilon,\delta}$ converges in probability in 
$\Cr^\alpha([0,T], \CB \oplus \CB_2)$ as $\delta \to 0$ to the random rough path $\Z^{\epsilon}$
characterised in \Cref{prop:defZZeps} below. (In particular, the first order component $Z^{\epsilon}$
of $\Z^{\epsilon}$ is given, for any fixed $x$, by the Wiener integral \eref{e:defZ} with $\delta = 0$.)
\end{proposition}

For now, we take this result as granted.
With this result in place, we obtain the following convergence result as $\delta \to 0$.

\begin{proposition}\label{prop:convDeltaConsequence}
The second claim of Theorem~\ref{theo:mainReal} holds. 
\end{proposition}

\begin{proof}
With the space $\CB$ as above, let $\delta \colon \R^d \to L(\CB,\R^d)$ be the function 
given by $\delta(x)(f) = f(x)$.
We then claim that, for any $\eps, \delta > 0$, \eqref{e:deltaeps2} can be rewritten as the 
rough differential equation (RDE) driven by the infinite-dimensional rough paths
$\Z^{\epsilon,\delta}$ and $\bar \Z^{\epsilon}$ defined above and given by
\begin{equation}\label{e:infinite}
dX = \delta(X)\,d\Z^{\epsilon,\delta} + \delta(X)\,d\bar \Z^{\epsilon}\;.
\end{equation}
Note that since $\alpha + \beta > 1$, there is no need to specify
cross-integrals between $\Z^{\eps,\delta}$ and $\bar \Z^{\eps}$ since they can be defined in a canonical
way using Young integration \cite{Young}.

To check that this RDE is well-posed for any rough paths $\Z^{\epsilon,\delta}$ and $\bar \Z^\eps$
belonging to $\Cr^\alpha([0,T], \CB \oplus \CB_2)$ and $\Cr^\beta([0,T], \CB \oplus \CB_2)$ respectively,
we note first that one readily verifies that the map
$\delta$ is Fréchet differentiable, and actually even $\CC^3_b$. 
Its differential $D\delta$ at $x\in \R^d$ in the direction of $y\in \R^d$ is given by
$(D\delta)_x(y)(h) = (Dh)_x(y)$ where $h\in  L(\CB,\R^d)$.  Morerover $|(D\delta)_x(h)|\le |h|_\CB$.
 In particular we may consider the map 
$D\delta \cdot \delta \colon \R^d \to L(\CB \otimes \CB,\R^d)$ which for  $h = h_1 \otimes  h_2\in \CB_2$ is given by
\begin{equs}[e:Ddelta]
\bigl(D\delta \cdot \delta\bigr)(x) h &= (D\delta)_x \bigl(\delta(x)(h_1)\bigr)(h_2)=(D\delta)_x \bigl(h_1(x)\bigr)(h_2) \\
&= (Dh_2)_x(h_1(x)) = (\tr D^{(2)} h)(x,x) \;,
\end{equs}
for a suitable partial trace $\tr$.
This shows that $D\delta \cdot \delta$ extends continuously to a $\CC^2_b$ map from $\R^d$ into $L(\CB_2, \R^d)$.
(Since $\CB_2$ differs from the projective tensor product of $\CB$ with itself, this doesn't automatically
follow from the fact that $\delta$ itself is $\CC^3_b$.)
Retracing the standard existence and uniqueness proof for RDEs, \cite[Sec.~8.5]{Book} then shows that
\eqref{e:infinite} admits unique (global) solutions for every 
initial condition and every driving path.
Furthermore, if the sample path $t \mapsto Y(t)$ is given by any continuous 
$\CY$-valued function then, under the stated 
regularity conditions on $F, F_0$, it is straightforward to verify that the solutions to \eqref{e:infinite} 
coincide with those of \eqref{e:deltaeps2}.

Since the RDE solution is a jointly locally Lipschitz continuous function of both the
initial data $x_0$ and the driving paths $\Z^{\eps,\delta}$ and $\bar \Z^{\eps}$ into
$\CC^\alpha(\R_+,\R^d)$, the claim that $X^{\eps,\delta} \to X^\eps$ in probability
then follows immediately from \Cref{prop:convDelta}.
\end{proof}

\begin{remark}
This shows that it is consistent to \textit{define} solutions to \eqref{e:SDEbasic} in the general case
as simply being a shorthand for solutions to \eqref{e:infinite} driven by the 
pair of rough paths $\Z^\eps$ and $\bar \Z^\eps$. This is the interpretation that we will use
from now on. The fact that in the special case $H =\f12$ this coincides with the Stratonovich
interpretation of the equation follows as in \cite[Thm~9.1]{Book}.
\end{remark}

\subsection{Preliminary results}

In this section we present a few general results that will be used in the proof of 
\Cref{prop:convDelta}.
We start with the following elementary 
property of the second Wiener chaos.

\begin{proposition}\label{prop:contractWiener}
Let $\CH \subset \CE$ be an abstract Wiener space and let $B, \tilde B$ be two
i.i.d.\ 
 Gaussian random variables on $\CE$ with Cameron--Martin space $\CH$.
Let $K^\delta \colon \CE \times \CE \to \R$ be  continuous bilinear maps such that the limit
$\tilde K = \lim_{\delta \to 0} K^\delta(B,\tilde B)$ exists in $L^2$. Then, the limit
\begin{equ}[e:limK]
K = \lim_{\delta \to 0} \bigl(K^\delta(B,B)- \E K^\delta(B,B)\bigr)
\end{equ}
exists in $L^2$. Furthermore, the limit in \eqref{e:limK} depends only on the limit $\tilde K$
and not on the approximating sequence $K^\delta$ and
one has the bound $\E  K^2 \le 2 \E \tilde K^2$.
\end{proposition}

\begin{proof}
The Gaussian probability space generated by the pair $(B,\tilde B)$ has Cameron--Martin space $\CH \oplus \tilde \CH$ where $\tilde \CH$ is a copy of $\CH$.
Since $K^\delta$ is bilinear and $K^\delta(B,\tilde B)$ has vanishing expectation, it 
belongs to the second homogeneous Wiener chaos, so that there exists $\hat \K^\delta \in (\CH \oplus \tilde \CH) \otimes_s (\CH \oplus \tilde \CH)$
with $K^\delta(B,\tilde B) = \CI_2(\hat \K^\delta)$, where $\CI_k$ denotes the usual isometry between $k$th symmetric tensor power
and $k$th homogeneous Wiener chaos, see \cite{NualartBook}.

Note now that, interpreting $\hat \K^\delta$ as a
Hilbert--Schmidt operator on $\CH \oplus \tilde\CH$, there exists
$\K^\delta \in \tilde \CH \otimes  \CH$ such that $\hat \K^\delta = \iota \K^\delta$, where
$\iota \colon  \tilde\CH \otimes \CH \to (\CH \oplus \tilde \CH) \otimes_s (\CH \oplus \tilde \CH)$ is given by
\begin{equ}
\iota \K = \f12 \begin{pmatrix} 0 & \tau  \K \cr  \K & 0 \end{pmatrix}\;, 
\end{equ}
with the obvious matrix notation and $\tau \colon \tilde\CH \otimes \CH \to \CH \otimes\tilde  \CH$ the
transposition operator. 

This is because the first diagonal block is obtained by testing against
 $\<k, (B, \tilde B)\>$ where $k=(f, 0) \otimes_s (g,0)$ with $f,g \in \CH$,   yielding 
\begin{equ}
\E \bigl(K^\delta(B,\tilde B) \bigl(\scal{f,B}\scal{g,B} - \scal{f,g}\bigr)\bigr) = 0\;.
\end{equ}
The second diagonal block vanishes for the same
reason with the roles of $B$ and $\tilde B$ exchanged.
(Here we denote by $x\mapsto \<x,f\>$ the unique element of $L^2(\CE)$ 
which is linear on a set of full measure containing $\CH$ and coincides with $\scal{f,\cdot}$ there. In fact,
$\<x,f\>=f^*(x)$ when $f^*\in\CE^*$ and $\<B,f\>=f^*\circ B$. )

On the other hand, one has 
\begin{equ}
K^\delta_\diamond(B,B) \eqdef K^\delta(B,B)- \E K^\delta(B,B) = \f12 \CI_2(\K^\delta + \tau \K^\delta)\;,
\end{equ} 
where this time $\CI_2$ refers to the isometry between $\CH\otimes_s\CH$ and the second chaos generated by $B$ only.
Since $\CI_2$ and $\tau$ are both isometries, it immediately follows that 
$\E (K^\delta_\diamond(B,B))^2  \le \|\K^\delta\|^2 = 2 \|\iota \K^\delta\|^2 = 2\E (K^\delta(B,\tilde B))^2$, and similarly for 
differences $K^\delta - K^{\delta'}$, so that the claim follows.
\end{proof}

\begin{remark}\label{rem:sym}
It follows immediately from \eqref{e:limK} that if we replace $K^\delta$ by the same sequence of
bilinear maps, but with their two arguments exchanged, the limit one obtains is the same.
\end{remark}
Before we turn to the precise statement,
we introduce the following notation which will
be used repeatedly in the sequel. We write $\eta$ for the distribution on $\R$ given by
\begin{equ}[e:defEta]
\scal{\eta,\phi} =\int_\R |t|^{2H}\phi(t)\,dt\;,
\end{equ}
and $\eta''$ for its second distributional derivative.
For $a < 0 < b$, 
we will then make the abuse of notation 
$\int_a^b \eta''(t)\,\phi(t)\,dt$ as a shorthand for 
$\lim_{\eps \to 0} \scal{\eta'', \phi \one^\eps_{[a,b]}}$, where $\one^\eps_{[a,b]}$ denotes a
mollification of the indicator function $\one_{[a,b]}$. The following is elementary.
\begin{lemma}\label{distributional-derivative}
Let $a<0<b$ and $H \in (\f13,\f12)$. 
Setting  $\alpha_H= H(1-2H)$ and $\phi_0 = \phi(0)$, the limit above is given by
\begin{equ}[e:ideneta]
\int_a^b \eta''(t)\,\phi(t)\,dt = -2\alpha_H \int_a^b |t|^{2H-2}\,\big(\phi(t) - \phi_0\big)\,dt 
+ 2H \phi_0 \bigl(|a|^{2H-1} + |b|^{2H-1}\bigr)\;,
\end{equ}
independently of the choice of mollifier, thus justifying the notation. For $a = 0$, we set
\begin{equ}
\int_0^b \eta''(t)\,\phi(t)\,dt = -2\alpha_H \int_0^b |t|^{2H-2}\,\big(\phi(t) - \phi_0\big)\,dt 
+ 2H \phi_0 |a|^{2H-1}\;,
\end{equ}
which can be justified in a similar way, provided that the mollifier one uses is symmetric.
(This in turn is the case if we view $\eta$ as the limit of covariances of smooth approximations
to fractional Brownian motion.) 

For $H = \f12$, one similarly has $\int_a^b \eta''(t)\,\phi(t)\,dt = \phi_0$ and $\int_0^b \eta''(t)\,\phi(t)\,dt = \f12 \phi_0$, while for  $H \in (\f12,1)$, $\eta''$ is given by the locally integrable function 
$t \mapsto -2\alpha_H |t|^{2H-2}$.
\qed
\end{lemma}


We now show that for any fixed $\eps > 0$, the processes $Z^\eps$ satisfy a suitable form of Hölder
regularity. To keep notations shorter, we define the collection of processes
\begin{equ}[e:defZf]
Z^f_{s,t} \eqdef \int_s^t f(r)\,dB_i(r)=\sum_{i=1}^m\int_s^t\< f(r), e^i\>_{\R^m} dB_r^i\;,
\end{equ}
indexed by $\R^m$-valued functions $f$
that belong to the reproducing kernel space of the fractional Brownian motion $B$.  Here $\{e^i\}$ is an o.n.b. of $\R^m$ and $B_r=(B_r^1, \dots, B_r^m)$. We start our analysis with some preliminary result for the
irregular case $H < \f12$.

\begin{lemma}\label{prop:boundCovZ}
Let $H \in (\f13,\f12)$ and let $f\in \CC^\beta([0,1],\R^m)$ for some $\beta > (1-2H) \vee 0$. The processes $Z^f$ satisfy the Coutin--Qian conditions
\cite[Def.~14]{Coutin-Qian,FVGaussian} in the sense that 
\begin{equs}
\E \big(Z^f_{s,t}(x)^2\bigr) &\lesssim \|f\|_\infty\|f\|_{\CC^\beta} |t-s|^{2H}\;,\\ 
\big|\E \big(Z^f_{s,s+h}(x)\, Z^f_{t,t+h}(x)\bigr)\big| &\lesssim \|f\|_\infty^2 |t-s|^{2H-2} h^2\;, 
\end{equs}
for all $0 < s < t < 1$ and all $h \in (0, t-s]$.
\end{lemma}

\begin{proof}
The mixed second order distributional derivative of $ \E(B^\delta_sB_t^\delta)$  is given by $\f 12 \eta_\delta''(s-t)$,
the convolution of $\f12\eta''$ with a symmetric mollifier at scale $\delta$.
 Mollifying $B$ and taking limits shows that we have the identity
\begin{equs}
\E (Z^f_{s,t})^2 &=  \lim_{\delta\to 0}  \int_s^t\int_s^t \E (\dot B_u^\delta \dot B_v^\delta) f(u)f(v)\,du\,dv\;\\
&= \f 12\int_s^t\int_s^t \eta''(x-y)f(x)f(y)\,dx\,dy\;,
\end{equs}
(with summation over the components of $f$ implied).  
For $H < \f12$, this yields the bound
\begin{equs}
\E (Z^f_{s,t})^2 &=\f 14 \int_{s-t}^{t-s}\eta''(v) \int_{2s+|v|}^{2t-|v|}  \Bigl(f\Big(\f{u+v}2\Big)f\Big(\f{u-v}2\Big)\Bigr)\,du\,dv \\
&= -\f 12\alpha_H\int_{s-t}^{t-s} |v|^{2H-2}  \int_{2s+|v|}^{2t-|v|}  \Bigl(f\Big(\f{u+v}2\Big)f\Big(\f{u-v}2\Big) - f\Big(\f u2\Big)^2\Bigr)\,du\,dv \\
&\qquad  +\f{\alpha_H}{1-2H}\int_0^{t-s}u^{2H-1} \Bigl(f\Big(t-\f u2\Big)^2+ f\Big(s+\f u2\Big)^2\Bigr)\,du \label{e:covarIst}\\
&\lesssim \|f\|_{\CC^\beta} \|f\|_\infty |t-s|^{2H+\beta}
+ \|f\|_\infty^2 |t-s|^{2H}\;,
\end{equs}
as required. 

 Regarding the covariance,  we have 
\begin{equs}
\big|\E \bigl(Z^f_{s,s+h}Z^f_{t,t+h}\bigr)\bigr| 
&\lesssim \|f\|_\infty^2 \int_0^h \int_0^h|t-s+v-u|^{2H-2} \,du\,dv \\
&\lesssim \|f\|_\infty^2 h^2 |t-s|^{2H-2}\;,
\end{equs}
when $0<h\le t-s$ so the intervals overlap only at most one point, as required. \end{proof}

If $\tilde B$ is an independent copy of $B$, we can combine this result
with those of \cite{FVGaussian} to conclude that there is a canonical rough path associated with 
the path $\big(\int f(r)\,dB(r),  \int g(r)\,d\tilde B(r)\big)$ for any $f,g\in  \CC^\beta([0,1],\R^m)$. 
We now show that there also exists a canonical lift for $(Z^f, Z^g)$, where the integrals are 
defined with respect to the same fractional Brownian motion $B$.
With this notation, we then have the following result.

\begin{proposition}\label{prop:boundRP}
For $H \in (\f13,\f12)$, we set $\alpha = H$ and $U = \CC^\beta([0,1],\R^m)$ for some $\beta \in (\f13 , H)$.
For $H \in [\f12, 1)$, we fix $p > 3/(3H-1)$ and set $\alpha = H-\f1p$ and $U = L^p([0,1],\R^m)$.

Then, for any finite collection $\{f_i\}_{i \le N} \subset U$, there is a ``canonical lift'' of 
$$Z^f = (Z^{f_1},\ldots,Z^{f_N})$$ 
given by \eqref{e:defZf} to a geometric rough path 
$\Z^f = (Z^f, \ZZ^f)$ on $\R^N$ such that
\begin{equ}[e:wantedBoundZ]
\sup_{s \neq t} |t-s|^{-q\alpha} \E \Bigl(|Z^f_{s,t}|^q + |\ZZ^f_{s,t}|^{\f q2}\Bigr) \lesssim \sum_i\|f_i\|_{U}^{q}\;,
\end{equ}
for every $q \ge 1$. This is obtained by taking the limit as $\delta \to 0$
of the canonical lift of the smooth paths $Z^{\delta,f}$ defined as in \eqref{e:defZf} 
but with $B$ replaced by $B^\delta$.
\end{proposition}

\begin{remark}\label{rem:mollifier}
As usual, ``geometric'' here means that $\Z^f$ is the limit of canonical lifts of smooth functions.
Indeed, for $B^\delta$ the convolution of $B$ with a mollifier at scale $\delta > 0$, $\ZZ^{f}$ is given by 
$$
(\ZZ^{f}_{s,t})_{ij} =\lim_{\delta\to 0}  \int_s^t \int_s^r   f_i(u)\, \dot B^\delta(u) du\, f_j(r)\, \dot B^\delta(r)dr\;,
$$
and this limit is independent of the choice of mollifier (and therefore ``canonical'').
\end{remark}

\begin{proof}
We only need to show \eqref{e:wantedBoundZ} for $q=2$ since $Z$ and $\ZZ$ belong to a Wiener chaos of fixed
order. (Recall that the $f_i$ are considered deterministic here.)

We start with the case $H \in (\f13,\f12)$.
Let $\tilde B$ denote an independent copy of the fractional Brownian motion $B$ and 
let $\tilde Z^{f,g} = (Z^f, \tilde Z^g)$
where $\tilde Z^g$ is defined like $Z^{g}$ but with $ B$ replaced by $\tilde B$. By \Cref{prop:boundCovZ}, for any $f,g \in \CC^\beta$, we
can then apply \cite[Thm~35]{FVGaussian} to construct a second-order 
process $\tilde \ZZ^{f,g}_{s,t}$  satisfying the Chen identity
\begin{equ}
\tilde \ZZ^{f,g}_{s,t} - \tilde \ZZ^{f,g}_{s,u} - \tilde \ZZ^{f,g}_{u,t} = Z^f_{s,u} \, \tilde Z^g_{u,t}\;.
\end{equ}
It furthermore coincides with the Wiener integral
\begin{equ}[e:interpretZZ]
\tilde \ZZ^{f,g}_{s,t} = \int_s^t Z^f_{s,r}\,d\tilde Z^g_{s,r} = \int_s^t \int_s^r   f(u)\,dB(u)\, g(r)\,d\tilde B(r)\;,
\end{equ}
which makes sense since the Coutin--Qian condition guarantees that $r \mapsto Z^f_{s,r}$ belongs to the 
reproducing kernel space of $Z^g$.
It is furthermore such that smooth approximations to \eqref{e:interpretZZ} 
(replace $B$ and $\tilde B$ by $B^\delta$ and 
$\tilde B^\delta$, obtained by convolution with a mollifier at scale $\delta \to 0$) converge to it in $L^2$.
In particular, $\tilde \ZZ^{f,g}$ belongs to the second Wiener chaos generated by $(B,\tilde B)$ and
is of the form of the limits considered in Proposition~\ref{prop:contractWiener}.

We now want to replace $\tilde B$ by $B$. 
For an approximation $B^\delta$ as mentioned above, setting
\begin{equ}[e:constructZZ]
\ZZ^{\delta,f,g}_{s,t} =  \int_s^t \int_s^r   f(u)\,dB^\delta(u)\, g(r)\,dB^\delta(r)\;,
\end{equ}
we have
\begin{equ}[e:expZZdelta]
\E \ZZ^{\delta,f,g}_{s,t} ={\f 12} \int_s^t \int_s^r   f(u) g(r) \, \eta_\delta''(u-r)\,du\,dr\;,
\end{equ}
where $\eta_\delta$ is an even $\delta$-mollification of $t \mapsto |t|^{2H}$, so that in particular
$\int_0^\infty \eta_\delta''(s)\,ds=0$, since $\eta_\delta'(0)=0$. Similarly to \eqref{e:covarIst},
we can then rewrite \eqref{e:expZZdelta} as
\begin{equs}
\E \ZZ^{\delta,f,g}_{s,t} &={\f 12} \int_{s}^{t} g(r) \int_{s}^{r} \bigl(f(u) - f(r)\bigr) \eta_\delta''(r-u)\,du\,dr \\
&\qquad -{\f 12}  \int_{s}^{t} g(r) \int_{-\infty}^{s} f(r) \eta_\delta''(r-u)\,du\,dr\;.
\end{equs}
%
It follows immediately that one has
\begin{equs}
\lim_{\delta \to 0} \E \ZZ^{\delta,f,g}_{s,t} &= H(2H-1)\int_{s}^{t} g(r) \int_{s}^{r} \bigl(f(u) - f(r)\bigr) |r-u|^{2H-2}
\,du\,dr \\
&\qquad - H\int_{s}^{t} g(r)f(r) |r-s|^{2H-1}\,dr\;, \label{e:limitExpZ}
\end{equs}
which is bounded by some multiple of
\begin{equ}
\|f\|_{\CC^\beta} \|g\|_\infty |t-s|^{2H+\beta} + \|f\|_\infty \|g\|_\infty |t-s|^{2H}\;.
\end{equ}
Combining this with Proposition~\ref{prop:contractWiener}, we conclude that 
\begin{equ}[e:convProb]
\ZZ^{f,g}_{s,t} = \lim_{\delta \to 0} \ZZ^{\delta,f,g}_{s,t}\;,
\end{equ}
exists in probability, is independent of the choice of mollification, and satisfies the bound
\begin{equ}[e:boundZZ]
|\E \ZZ^{f,g}_{s,t}| \lesssim \|f\|_{\CC^\beta} \|g\|_\infty |t-s|^{2H+\beta} + \|f\|_\infty \|g\|_\infty |t-s|^{2H}\;.
\end{equ}
It now suffices to set $\ZZ^{f,ij}_{s,t} = \ZZ^{f_i,f_j}_{s,t}$. Both the fact that Chen's 
relation holds and the fact that the resulting rough path is geometric follow at once from
the fact that these properties hold for the smooth approximations.

Combining \eqref{e:boundZZ} with the fact that the rough path obtained from 
\cite[Thm~35]{FVGaussian} satisfies the bound \eqref{e:wantedBoundZ} with $\alpha = H$ 
as a consequence of \Cref{prop:boundCovZ}, the claim follows.

We now turn to the case $H = \f12$ where it is well known that 
$\ZZ^{\delta,f,g}_{s,t}$ defined in \eqref{e:constructZZ}
converges to the Stratonovich integral  $ \int_s^t \int_s^r   f(u)\,dB(u)\, g(r)\, \circ dB(r)$, so that
\begin{equ}
\ZZ^{f,g}_{s,t} = \int_s^t \int_s^r   f(u)\,dB(u)\, g(r)\,dB(r)
+ \f12 \int_s^t \scal{f(u),g(u)}\,du\;.
\end{equ}
A simple consequence of Hölder's inequality then leads to the bounds
\begin{equ}
\E |Z_{s,t}^f|^2 \lesssim \|f\|_{L^p} |t-s|^{1-\f2p}\;,\qquad
\E |\ZZ_{s,t}^{f,g}|^2 \lesssim \|f\|_{L^p}\|g\|_{L^p}|t-s|^{2-\f4p}\;,
\end{equ}
where  $ \|f\|_{L^p}= \|f\|_{L^p([0,1])}$. This shows again that the bound \eqref{e:wantedBoundZ} holds, this time with 
$\alpha = \f12-\f1p$ (and $q$ arbitrary), and our condition on $p$ guarantees that this is greater than $\f13$.

For $H > \f12$, the first identity in \eqref{e:covarIst} above combined with the positivity of 
the distribution $-\eta''$ and Hölder's inequality yields the bound
\begin{equ}
\E (Z^f_{s,t})^2 \lesssim |t-s|^{1-\f2p}\|f\|_{L^p}^2 \Big| \int_{s-t}^{t-s}\eta''(v)\,dv\Big|
\lesssim \|f\|_{L^p}^2 |t-s|^{2H-\f2p}\;.
\end{equ}
Similarly, again as a consequence of the positivity of $-\eta''$, we have the bound
\begin{equ}
\E (\ZZ^{f,g}_{s,t})^2 \le \E \big((Z^{|f|}_{s,t})^2(Z^{|g|}_{s,t})^2\big)
\le 3\E (Z^{|f|}_{s,t})^2\E (Z^{|g|}_{s,t})^2\lesssim \|f\|_{L^p}^2\|g\|_{L^p}^2 |t-s|^{4H-\f4p}\;,
\end{equ}
which shows again that \eqref{e:wantedBoundZ} holds.
\end{proof}

\subsection{Construction and convergence of the rough driver as \TitleEquation{\delta \to 0}{delta goes to 0}}
\label{sec:RD}

The aim of this section is to construct the rough path $\Z^\eps$ (this is the content of \Cref{prop:defZZeps})
and to show that this construction enjoys good stability properties. 
This is done by  stitching together the ``canonical'' rough path lift for the collection 
$\{Z^\eps(x)\}_{x \in \R^d}$ obtained in \Cref{prop:boundRP}. For $H>\f 12$, this is just iterated Young integrals.
For $H=\f 12$ the iterated integrals are considered in the Stratonovich sense and, 
thanks to the independence of $Y$ and $B$, the first order process can be interpreted indifferently as
either an It\^o or a Stratonovich integral. 

In order to make use of \Cref{prop:boundRP}, we use the following lemma, where $Y^\eps_t$ denotes the 
Markov process from \Cref{sec:processY}.

\begin{lemma}\label{lem:boundCF}
Let $U$ be as in Proposition~\ref{prop:boundRP} and let Assumption~\ref{ass:continuous} hold for some $p_\star$.   When $H < \f12$, we further assume $E\subset L^{p_\star}$ and $\beta < H - \f1{p_\star}$, where $U = \CC^\beta$.
Then, given $f \in E$ and setting $\hat f(t) = f(Y^\eps_t)$, one has for every $p < p_\star$  the bound  
\begin{equ}
\E \|\hat f\|_U^p \lesssim \|f\|_E^p\;,
\end{equ}
uniformly over $E$. (Here we use the convention $p_\star = \infty$ when $H>\f12$.)
\end{lemma}

\begin{proof}
For $H \in (\f13,\f12)$, the assertion follows immediately from Kolmogorov's continuity test,
using Assumption~\ref{ass:continuous} and $E\subset L^{p}$.
For $H \ge \f12$, it suffices to note that if $f \in L^p(\CY,\mu)$, then for any fixed $\eps > 0$ the map
$\hat f \colon t \mapsto f(Y^\eps_t)$ belongs to $L^p([0,1])$ almost surely and $\E \|\hat f\|_{L^p}^p \le \|f\|_{L^p}^p$.
\end{proof}

We now show how to collect these objects into one ``large'' Banach space-valued rough path. The process
itself will take values in $\CB = \CC_b^3(\R^d,\R^d)$, with
the second order processes taking value in $\CB_2 = \CC_b^3(\R^{d}\times \R^{d},(\R^{d})^{\otimes 2})$.
Our aim is then to define a $\CB\oplus \CB_2$-valued rough path $(Z^\epsilon,\ZZ^\epsilon)$ which is the 
canonical lift (in the sense of \Cref{prop:boundRP} for any finite collection of $x$'s) of
\begin{equ}[e:ZepsGeneral]
(Z^\eps_{s,t})(x) = \epsilon^{\f 12-H}\int_s^t F(x,Y^\eps_r) \,dB(r)\;.
\end{equ}
Let $\{f_{i,x}^\eps\}_{i\le d}$ be the collection of maps from $\R_+$ to $ \R^m$ determined by
\begin{equ}
\scal{f_{i,x}^\eps(t), e} = (F(x,Y^\eps_t)e)_i\;,\qquad \forall e \in \R^m\;.
\end{equ}  
With this  notation at hand and recalling the construction of $Z^f$ and $\ZZ^{f,g}$ as in the proof of 
Proposition~\ref{prop:boundRP}, 
it is then natural to look for a $\CB\oplus \CB_2$-valued rough path $(Z^\epsilon,\ZZ^\epsilon)$
such that, for every $x, \bar x \in \R^d$, the identities
\begin{equ}[e:defZZ]
\big(Z^\eps_{s,t}(x)\big)_i = \eps^{\f12 -H} Z^{f_{i,x}^\eps}_{s,t}\;,\quad
\big(\ZZ^\eps_{s,t}(x,\bar x)\bigr)_{ij} \eqdef \eps^{1 -2H} \ZZ^{f_{i,x}^\eps, f_{j,\bar x}^\eps}_{s,t}\;,
\end{equ}
hold almost surely. (Provided that $F(x,\cdot) \in E$ for every $x$, 
the right-hand sides make sense by combining 
Lemma~\ref{lem:boundCF} with Proposition~\ref{prop:boundRP}.)
We claim that this does indeed define a \textit{bona fide} infinite-dimensional rough path

%

Recall also that a rough path $\Z = (Z,\ZZ)$ is \textit{weakly geometric} if the identity
\begin{equ}[e:weakgeo]
Z_{s,t} \otimes Z_{s,t} = \ZZ_{s,t} +  \ZZ^\top_{s,t}\;,
\end{equ}
holds, where the transposition map 
$(\cdot)^\top \colon \CB \otimes \CB \to \CB \otimes \CB$ swapping the two factors is continuously
extended to $\CB_2$. With these notations, we have the following.

\begin{proposition}\label{prop:defZZeps}
Let Assumptions~\ref{ass:continuous}, \ref{ass:integrability}, and~\ref{ass:C3} hold and let $\alpha \in (\f13 ,H - \f1{p_\star})$ if $H < \f12$ and 
$\alpha \in (\f13,\f12)$ otherwise.
Then, for any $\eps > 0$, there exists a random rough path $\Z^\eps = (Z^\eps,\ZZ^\eps)$ in 
$\Cr^\alpha([0,T], \CB \oplus \CB_2)$ that is weakly geometric and such that, for every $x$, \eqref{e:defZZ} 
holds almost surely. 
\end{proposition}

\begin{proof}
Since the Chen relations and \eqref{e:weakgeo} are obviously satisfied for 
 smooth approximations as in \eqref{e:constructZZ}, we 
 only need to show that the analytic constraints hold. In other words, for
any fixed $T >0$ and $\eps > 0$, 
we look for an almost surely finite random variable $C_\eps$ such that 
\begin{equ}
\|Z^\eps_{s,t}\|_{\CB} \le C_\eps |t-s|^\alpha\;, \qquad
\|\ZZ^\eps_{s,t}\|_{\CB_2} \le C_\eps |t-s|^{2\alpha}\;, 
\end{equ}
holds uniformly over all $0\le s < t \le T$.
By the Kolmogorov criterion for 
rough paths \cite[Thm~3.1]{Book}, it suffices to show that, for some
$\beta > 0$ and $p \ge 1$ such that $\gamma - \f1p > \alpha$, one has the bounds
\begin{equ}[e:wantedbound]
\E\|Z^\eps_{s,t}\|_{\CB}^p \le C_{\eps,p} |t-s|^{p\gamma}\;, \qquad
\E \|\ZZ^\eps_{s,t}\|_{\CB_2}^{p/2} \le C_{\eps,p} |t-s|^{p\gamma}\;.
\end{equ}

By Lemma~\ref{lem:compact} below, it suffices to show that 
\minilab{e:boundWanted}
\begin{equs}
\sup_{x\in \R^d} (1+|x|)^{\kappa p} \E |D^\ell Z^\eps_{s,t}(x)|^p &\lesssim |t-s|^{p\gamma}\;, \label{e:boundWanted1}\\
\sup_{x,\bar x\in \R^d} (1+|x|+|\bar x|)^{\kappa p \over 2} \E |D_x^k D_{\bar x}^\ell\ZZ^\eps_{s,t}(x,\bar x)|^{p\over 2}&\lesssim |t-s|^{p\gamma}\;,\label{e:boundWanted2}
\end{equs}
for $k + \ell \le 4$ and some $p$ such that $p > (4d/\kappa) \vee (\gamma-\alpha)^{-1}$.

Since for $\ell \le 4$ we have $D_x^\ell F_i^*(x,\cdot) \in E$ with $\|D_x^\ell F_i^*(x,\cdot)\|_E \lesssim (1+|x|)^{-\kappa}$ 
by Assumption~\ref{ass:C3}, it follows immediately from Proposition~\ref{prop:boundRP} combined
with Lemma~\ref{lem:boundCF} that the bound \eqref{e:boundWanted1} holds
for $\gamma = H$ and $p \le p_\star$ when $H < \f12$ and for any $\gamma < H$ and $p \ge 1$ when
$H \ge \f12$.
The bound \eqref{e:boundWanted2} follows in the same way.
(These arguments are somewhat formal, but can readily be justified by taking limits of smooth 
approximations.)
\end{proof}


In order to prove Proposition~\ref{prop:convDelta}, we make use of the following 
variant ``in probability'' of the usual tightness criterion for convergence in law.

\begin{proposition}\label{prop:tightProba}
Let $(\CZ,d)$ be a complete separable metric space and 
let $\{L_k\,:\, k \in \N\}$ be a countable collection
of continuous maps $L_k \colon \CZ \to \R$ that separate elements of $\CZ$ in the sense that,
for every $x,y \in \CZ $ with $x\neq y$ there exist $k$ such that $L_k(x) \neq L_k(y)$.

Let $\{Z_n\}_{n \ge 0}$ and $Z_\infty$ be $\CZ$-valued random variables such that 
the collection of their laws is tight and such that $L_k(Z_n) \to L_k(Z_\infty)$ in probability for
every $k$. Then, $Z_n \to Z_\infty$ in probability.
\end{proposition}

\begin{proof}
Let $\hat d\colon \CZ^2 \to \R_+$ be the continuous distance function given by
\begin{equ}
\hat d(x,y) = \sum_{k \ge 0} 2^{-k} \bigl(1\wedge |L_k(x) - L_k(y)|\bigr)\;,
\end{equ}
and note first that our assumption implies that $\hat d(Z_n, Z_\infty) \to 0$ in probability.
Given $\eps > 0$, tightness implies that there exists $K_\eps \subset \CZ$ compact such that 
$\P(Z_n \not \in K_\eps) \le \eps$ for every $n \in \N \cup \{\infty\}$. Furthermore, the set
$\{(x,y) \in K_\eps\times K_\eps\,:\, d(x,y) \ge \eps\}$ is compact, so that $\hat d$ attains
its infimum $\delta$ on it. Since $\hat d$ only vanishes on the diagonal, one has $\delta > 0$ and,
since $\hat d(Z_n, Z_\infty) \to 0$ in probability, we can find $N>0$ such that
$\P(\hat d(Z_n, Z_\infty) \ge \delta) \le \eps$ for every $n \ge N$.

It follows that, for every $n \ge N$ one has
\begin{equ}
\P(d(Z_n, Z_\infty) \ge \eps) \le \P(Z_n \not \in K_\eps) + \P(Z_\infty \not \in K_\eps) + 
\P(\hat d(Z_n, Z_\infty) \ge \delta)
\le 3\eps\;,
\end{equ}
which implies the claim.
\end{proof}

Regarding tightness itself, the following lemma is a slight variation of well known results.
\begin{lemma}\label{lem:tightBasic}
Let $\hat \CB \subset \CB$ and $\hat \CB_2 \subset \CB_2$
be compact embeddings of Banach spaces such that $\hat \CB \otimes \hat \CB \subset \hat \CB_2$
with $\|v\otimes w\|_{\hat \CB_2} \lesssim \|v\|_{\hat \CB}\|w\|_{\hat \CB}$.
Let $\CA$ be a collection
of random $\hat \CB \oplus \hat \CB_2$-valued rough paths such that for some $\alpha_0>\f 13$
and every $\Z = (Z,\ZZ)\in \CA$, 
\begin{equ}[e:condTight]
\E \big(\|Z_{s,t}\|_{\hat \CB}^p + \|\ZZ_{s,t}\|_{\hat \CB_2}^{p/2}\bigr) \le |t-s|^{p\alpha_0}\;,
\end{equ} 
for some $p > 3/(3\alpha_0-1)$. Then, the laws of the $\Z$'s in $\CA$ are tight 
in $\Cr^\alpha([0,T], \CB \oplus \CB_2)$ for any $T > 0$ and $\alpha \in(\f13, \alpha_0 - \f1p)$.
\end{lemma}

\begin{proof}
Write $\CG$ for the metric space given by $\CB \oplus \CB_2$ endowed with the metric 
\begin{equ}
d(a\oplus b, \bar a\oplus \bar b) = \|\bar a-a\|_{\CB} \vee \big\|\bar b- b -\textstyle{\f12} (\bar a + a)\otimes (\bar a-a)\big\|_{\CB_2}^{1/2}\;.
\end{equ}
Recall then that $\Cr^\alpha([0,T], \CB \oplus \CB_2)$ can be identified with the usual
space of $\alpha$-Hölder functions with values in $\CG$ by identifying $\Z = (Z,\ZZ)$ with
the function $t \mapsto \Z_t \eqdef Z_{0,t} \oplus \ZZ_{0,t}$ and noting that, thanks to Chen's relations,
\begin{equ}
d \bigl(\Z_s,\Z_t\bigr) = \|Z_{s,t}\|_{\CB} \vee \|\ZZ_{s,t}+ \textstyle{\f12} Z_{s,t}\otimes Z_{s,t}\|^{1/2}_{\CB_2}\;.
\end{equ}
(See \cite[Sec.~7.5]{FVBook} for more details and motivation.)
Since $d$ generates the same topology on $\CG$ as that given by the Banach space structure of
$\CB \oplus \CB_2$, balls of $\hat \CB \oplus \hat \CB_2$ are compact in $\CG$. 
The claim then follows at once from Kolmogorov's continuity test, combined with the fact 
that, given a compact metric space $(\CX, d)$ and a compact subset $\CK$ of
a Polish space $(\CY,\bar d)$, the set $\CC^\beta(\CX,\CK)$ is compact in $\CC^\alpha(\CX,\CY)$
for any $\beta > \alpha$. 
\end{proof}

\begin{proof}[of \Cref{prop:convDelta}]
We apply \Cref{prop:tightProba} with the metric space $\CZ$ given by 
$\Cr^\alpha([0,T], \CB \oplus \CB_2)$, $Z_n = \Z^{\eps,\delta_n}$ for
any given sequence $\delta_n \to 0$, and $Z_\infty = \Z^\eps$ as constructed in  
\Cref{prop:defZZeps}. The continuous maps $L_k$ appearing in the statement are given
by the collection of maps $(Z,\ZZ) \mapsto Z_t(x)$ and $(Z,\ZZ) \mapsto \ZZ_{s,t}(x,\bar x)$
for a countable dense set of times $s$ and $t$ and elements $x, \bar x \in \R^d$. 

It follows from \eqref{e:boundWanted} and \Cref{lem:compact} that the bound
\eqref{e:condTight} holds for $\Z^{\eps,\delta}$, uniformly in $\delta$ (but with $\eps$ fixed),
so that the required tightness condition holds by \Cref{lem:tightBasic}.
For any fixed $\eps > 0$, the convergences in probability $Z^{\eps,\delta}_{t}(x)_{i} \to Z^{\eps}_{t}(x)_{i}$
and $\ZZ^{\eps,\delta}_{s,t}(x,\bar x)_{ij} \to \ZZ^{\eps}_{s,t}(x,\bar x)_{ij}$ were shown in \Cref{prop:boundRP}. (It suffices to apply it with the choices  
$f = f^\eps_{i,x}$ and $g = f^\eps_{j,\bar x}$.)
\end{proof}

\subsection{Formulation of the main technical result}
\label{sec:formMain}

The main technical result of this article can then be formulated in the following way. 

\begin{theorem}\label{theo:mainRP} Let $H\in (\f 13,1)$, let 
Assumptions~\ref{ass:ergodic}--\ref{ass:integrability}, and~\ref{ass:C3} hold,
and let $\alpha$ and $\Z^\eps$ be as in Proposition~\ref{prop:defZZeps}.  Then, as $\eps \to 0$, $\Z^\eps$ converges weakly in the space of $\alpha$-Hölder continuous $(\CB,\CB_2)$-valued  
rough paths to a limit $\Z$. Furthermore, there is a Gaussian random field $W$ as in \eqref{e:defWGaussian} such that
\begin{equ}
Z_{s,t}(x) = W(x,t) - W(x,s)\;,\qquad 
\ZZ_{s,t}(x,\bar x) = \WW^\Ito_{s,t}(x,\bar x) + \Sigma(x,\bar x) (t-s)\;,
\end{equ}
where $\WW^\Ito_{s,t}(x,\bar x) = \int_s^t Z_{s,r}(x)\otimes W(\bar x,dr)$, interpreted as an Itô integral
and $\Sigma$ is as in \eqref{e:defSigma}.
\end{theorem}

The proof of this result will be given in Sections~\ref{sec:tight} and~\ref{sec:limit} below, see
Proposition~\ref{prop3.1} which is just a slight reformulation of Theorem~\ref{theo:mainRP}. We first show
in Section~\ref{sec:tight} that the family $\{ \Z^\eps\}_{\eps \le 1}$ is tight in a suitable space
of rough paths and then identify its limit in Section~\ref{sec:limit}.

\begin{corollary}\label{cor:main} 
Under the assumptions of Theorem~\ref{theo:mainRP}, the solution flow  of  \eqref{e:SDEbasic} converges weakly to that of the
Kunita-type SDE \eqref{e:KunitaSDE}.
\end{corollary}
\begin{proof} 
Define the $\CB$-valued process
\begin{equ}
Z^{0, \epsilon}_{s,t}(x)  = \int_s^t F_0(x,Y^\eps_r)\,dr\;.
\end{equ}
It follows from Assumption~\ref{ass:C3} that, for $k \le 4$ and $p \le p_\star$, \begin{equ}
\|D^k Z^{0, \epsilon}_{s,t}(x)\|_{L^p} \le \int_s^t \|D^k F_0(x,Y^\eps_r)\|_{L^p}\,dr
\lesssim |t-s| (1+|x|)^{-\kappa}\;,
\end{equ}
uniformly over $\eps$. This shows that the family $\{Z^{0, \epsilon}\}_{\eps \le 1}$
is tight in $\CC^\beta([0,1],\CB)$ for every $\beta < 1$.

Furthermore, by the ergodic theorem which holds under Assumption~\ref{ass:ergodic},  for every $x$,   
\begin{equ}
\lim_{\eps \to 0} Z^{0, \epsilon}_{s,t}(x) = (t-s)\int_{\CY} F_0(x,y)\,\mu(dy)\;,
\end{equ}
almost surely. Since we can choose $\beta$ and $\alpha$ such that $\alpha + \beta > 1$ and $2\beta > 1$,
it follows that there is no need to control any cross terms between $Z^{0,\epsilon}$ and either $Z^\epsilon$
or $Z^{0,\epsilon}$ itself 
in order to be able to solve equations driven by both \cite{Lyons,Mixed}. Furthermore, since the limit of $Z^{0,\epsilon}$
is deterministic, one deduces joint convergence from Theorem~\ref{theo:mainRP}.

By the continuity theorem for rough differential equations, the solutions of 
  \eqref{e:SDEbasic} written in the form \eqref{e:infinite} converge weakly to those 
of the rough differential equation
\begin{equ}[e:SDEfinal]
  dX_t=\delta(X_t)\, d\Z_t+ \delta(X_t)\, d\Z_t^0\;.
\end{equ}
It remains to identify solutions to this equation with those of \eqref{e:KunitaSDE}.

This is straightforward and follows as in \cite[Sec~5.1]{Book} for example. 
Since the Gubinelli derivative $x'$ of the solution $X = (x,x')$ to \eqref{e:SDEfinal} is given by $\delta (x)$,
the integral  $\int_0^t \delta(X_s) d\Z_s$ is obtained as limit of the compensated Riemann 
sum 
\begin{equs}
\sum_{[u,v]\subset {\mathcal P}} &\bigl(\delta(x_u) Z_{u,v}+  (D\delta \cdot \delta)(x_u)(\WW^{\textup{It\^o}}_{u,v} +  (u-v) \Sigma )\bigr)\\
&= \sum_{[u,v]\subset {\mathcal P}} \bigl(Z_{u,v}(x_u) +  (u-v)G(x_u)\bigr)
+ \sum_{[u,v]\subset {\mathcal P}} (D\delta \cdot \delta)(x_u)\WW^{\textup{It\^o}}_{u,v}\;,
\end{equs}
where ${\mathcal P}$ is a partition on $[0,t]$ and $D\delta \cdot \delta$ is as in \eqref{e:Ddelta}. 
Since $x$ is continuous and adapted to the filtration generated by $W$, the first term converges to
\begin{equ}
\int_0^t W(x_u,du) + \int_0^t G(x_u)\,du\;.
\end{equ}
The last term on the other hand converges to $0$ in probability since 
it is a discrete martingale and its summands are centred random variables 
of variance $\CO(|v-u|^2)$.
\end{proof}

\section{Tightness of the rough driver as \TitleEquation{\eps \to 0}{eps}}
\label{sec:tight}
The content of this section is the proof of the following tightness result. Let $\{\Z^\eps\}_{\eps \le 1}$  be given as in
\Cref{prop:defZZeps}.

\begin{proposition}\label{prop:tight}
Let  Assumptions~\ref{ass:ergodic}--\ref{ass:integrability}, and~\ref{ass:C3} hold. For $H\in (\f 13, \f 12]$, there
exists $ \alpha \in (\f13, H)$ such that the family $\{\Z^\eps\}_{\eps \le 1}$ is tight in 
$\Cr^\alpha([0,T], \CB \oplus \CB_2)$.

For $H \in (\f12,1)$,   if in addition $\int F(x,y)\,\mu(dy) = 0$ for every $x$, then
the family of rough paths $\Z^\eps$ is tight in $\CC^\alpha$ for every $\alpha \in (\f13,\f12)$.
\end{proposition}

It will be convenient to introduce the following notation.
Given $f,g \in E$, we use the shorthand
\begin{equ}
J^\eps_{s,t}(f) = \eps^{\f12 -H} Z_{s,t}^{f(Y^\eps_\cdot)}\;,\qquad
\JJ^\eps_{s,t}(f) = \eps^{1 -2H} \ZZ_{s,t}^{f(Y^\eps_\cdot),g(Y^\eps_\cdot)}\;.
\end{equ}
We then have the following tightness criterion.

\begin{lemma}\label{lemma:tight}
Let $p>d+1$.
Assume that for any $f,g\in E$,  $|s-t|\le 1$, and $\eps\in (0,1]$, 
\begin{equs}
  \|J^\eps_{s,t}(f)\|_{L^p}\le    C \|f\|_E |t-s|^{\alpha_0}\;, \qquad
 \big\|\JJ^\eps_{s,t}(f,g)\big\|_{L^p} \le C \|f\|_E\|g\|_E |t-s|^{2\alpha_0}\;,\end{equs}
 where $p > 3/(3\alpha_0-1)$.
Let furthermore Assumption \ref{ass:C3}  hold with  $\kappa > \f{8d} p$. 
 Then the family $\{\Z^\eps\}_{\eps \le 1}$ is tight in 
$\Cr^\alpha([0,T], \CB \oplus \CB_2)$ for any $ \alpha < \alpha_0-1/p$.
\end{lemma}
\begin{proof}
Recall that with the above notations, one has from \eqref{e:defZZ}
\begin{equ}
\bigl(Z^\eps_{s,t}(x)\bigr)_i = J^\eps_{s,t}\big(F_i^*(x,\cdot)\big)\;,\qquad
\bigl(\ZZ^\eps_{s,t}(x,\bar x)\bigr)_{ij} = \JJ^\eps_{s,t}\big(F_i^*(x,\cdot),F_j^*(\bar x,\cdot)\big)\;.
\end{equ}
By the assumption for $|\ell| \le 3$ one has for $|s-t|\le 1$ and $|x-x'| \le 1$,
\begin{equs}
\|D^\ell Z^\eps_{s,t}(x)-D^\ell & Z^\eps_{s,t}(x')\|_{L^p} 
\lesssim \sum_i \big\|J^\eps_{s,t}\big(D_x^\ell F_i^*(x,\cdot)-D_x^\ell F_i^*(x',\cdot)\big)\big\|_{L^p} \\
&\lesssim \|D_x^\ell F(x,\cdot)-D_x^\ell F(x',\cdot)\|_E|t-s|^{\alpha_0}\\
&\lesssim \sup_{y \in [x,x']} \|D_x^{\ell} F(y,\cdot)\|_E  |x-x'| |t-s|^{\alpha_0} \\
&\lesssim  (1+|x|)^{-\kappa}  |x-x'||t-s|^{\alpha_0}\;,
\end{equs}
where we wrote $[x,x']$ for the convex hull of $\{x,x'\}$.
Here, the last bound follows from Assumption~\ref{ass:C3}. 
It then follows from Lemma~\ref{lem:compact} that, for $\hat \CB$ as defined in the appendix,
$\E \|Z^\eps_{s,t}\|_{\hat \CB}^p
\lesssim |t-s|^{p \alpha_0}$. We choose $\zeta$ to be any number in $(0, 1-d/p)$.

It similarly follows 
that for $|k+\ell| \le 3$ and $|x-x'|\le 1$
\begin{equs}
\|D_x^k D_{\bar x}^\ell\ZZ^\eps_{s,t}&(x,\bar x)-D_x^k D_{\bar x}^\ell\ZZ^\eps_{s,t}(x',\bar x)\|_{L^{p/2}} \\
&\lesssim \sum_{i,j}  \|\JJ^\eps_{s,t}\big(D_x^k F_i^*(x,\cdot)-D_x^k F_i^*(x',\cdot), D_x^\ell F_j^*(\bar x,\cdot)\bigr)\|_{L^{p/2}}\\
&\lesssim \sup_{y \in [x,x']} \|D_x^{k} F(y,\cdot)\|_E\|D_x^{\ell} F(\bar x,\cdot)\|_E  |x-x'| |t-s|^{2\alpha_0}\\
&\lesssim  (1+|x|)^{-\kappa}(1+|\bar x|)^{-\kappa}  |x-x'||t-s|^{2\alpha_0}\;,
\end{equs}
(and analogously when varying $\bar x$).
Since $\kappa p /8 > d$ by assumption, we can again apply Lemma~\ref{lem:compact} with $p/2$, thus yielding
$\E \|\ZZ^\eps_{s,t}\|_{\hat \CB_2}^{p/2} \lesssim |t-s|^{p\alpha_0}$.
Since furthermore $p > 3/(3\alpha_0-1)$ by assumption, 
the conditions of Lemma~\ref{lem:tightBasic} are satisfied
and the claim follows for any $\alpha < \alpha_0 - 1/p$.
\end{proof}

\begin{proof}[of Proposition~\ref{prop:tight}]
The arguments are quite different for the different ranges of $H$, but they will always reduce to
verifying the assumptions of Lemma~\ref{lemma:tight}.  
 
First let $H\in (\f 13, \f 12)$. The first assumption of Lemma~\ref{lemma:tight}
follows from Proposition~\ref{prop:tight1st} below with $\alpha_0=H$ and from the trivial bound
\begin{equ}
\eps^{\f12-\alpha_0} t^\alpha_0 \vee \sqrt t \lesssim t^\alpha_0\;,\qquad \forall \eps \le 1\;,\quad t \le T\;,
\end{equ}
while the second assumption follows from Proposition~\ref{prop:tight2nd} below. Both hold for any 
$p\le p_\star/4$ where $p_\star>\max\{4d, \f{6} {(3H-1)}\}$, 
and the proofs of the propositions are  the content of Section~\ref{sec:tightLow}. 

The ingredients for showing tightness of $\Z^\eps$ where $H\in (\f 12, 1)$ are given in Section \ref{sec:tightHigh}, starting with a bound
on $J$ analogous to that of \Cref{prop:tight1st}. Unlike in the proof of that statement though,
we do \textit{not} show this by bounding the conditional variance $\E \bigl(|J_{s,t}^\eps(f)|^2\,|\,\CF^Y\bigr)$.
This is because, as a consequence of the lack of integrability at infinity of $\eta''$ when $H > \f12$, it 
appears difficult to obtain a sufficiently good bound on it, especially for $H$ close to $1$.
(In particular, the best bounds one can expect to obtain from a quantitative law of large numbers
don't appear to be sufficient when $H > \f 34$.) 
The required bounds are collected in \Cref{tight-estimate-regular}
which yields the assumptions of Lemma~\ref{lemma:tight} with $\alpha_0 = \f12$ and arbitrary $p$.

Finally we take $\alpha_0=\f 12$ when $H=\f 12$,  then
$  \|J^\eps_{s,t}(f)\|_{L^p}\le \sqrt{ \int_s^t |f(Y_r^\eps)|^2}  dr  \lesssim C \|f\|_E \sqrt{ |t-s|}\;$ by Burkholder-Davies-Gundy inequality,  and similarly the second order processes satisfies the bound:
\begin{equs}
 \big\|\JJ^\eps_{s,t}(f,g)\big\|_{L^p} &= \Big\| \int_0^t \int_0^{u} f(Y_v^\eps)g(Y_u^\eps) dB_vdB_u+\f 12 \int_0^t f(Y_r^\eps)g(Y_s^\eps)\Big\|_{L^p} \\
 & \lesssim
   \|f\|_E\|g\|_E |t-s|\;,\end{equs}
allowing us to again apply Lemma~\ref{lemma:tight} and concluding the proof.
\end{proof}

\subsection{The low regularity case}\label{sec:tightLow}

This section consists of a number of a priori moment bounds, which we then combine
at the end to provide the proof of Proposition~\ref{prop:tight2nd}. These uniform in $\eps$ moment bounds 
follow from the H\"older continuity of $Y$ in a subspace of $L^{p^*}$ for a sufficiently large $p^*$, in particular ergodicity of $Y$ does not play any role.

We will make repeated use of the following simple calculation, where we recall \eref{e:defEta}
for the definition of the distribution $\eta$.

\begin{lemma}\label{lem:boundIntEta''}
Given $t > 0$ and $H < \f12$, let $\Psi \colon [0,t]^2 \to \R$ be a continuous function such that 
for some numbers $\eps > 0$, $\beta > -2H$, $\gamma, \zeta > 1-2H$, and $C, \hat C, \bar C \ge 0$
it holds that
\begin{equ}[e:assumePsi]
|\Psi(r,r)| \le C|r|^\beta\;,\quad
|\Psi(s,r)-\Psi(r,r)| \le \hat C|r|^\beta \bigl(1 \wedge \textstyle{|s-r|^\gamma \over \eps^\gamma}\bigr)
+ \bar C \eps^{\beta-\zeta} |s-r|^\zeta\;,
\end{equ}
for all $s,r \in [0,t]$.
 Then, one has the bound
\begin{equ}[e:boundDInt]
\Big|\int_0^t \int_0^t \Psi(s,r) \eta''(r-s)\,ds\,dr\Big| \leq K\bigl(
C t^{2H + \beta} + \hat C t^{\beta+1} \eps^{2H-1} + \bar C t^{\zeta+2H} \eps^{\beta-\zeta}\big)\;,
\end{equ} 
with the proportionality constant $K$ depending only on $\beta$, $\gamma$ and $\zeta$.
The same bound holds if the upper limit of the inner integral
in \eqref{e:boundDInt} is given by $r$ instead of $t$. 
\end{lemma}

\begin{proof}
Let $I$ be the double integral appearing in \eqref{e:boundDInt}.
As a consequence of Lemma~\ref{distributional-derivative}, we can rewrite it as
\begin{equs}
I &= -2\alpha_H \int_0^t\int_0^t \bigl(\Psi(s,r) - \Psi(r,r)\bigr)|r-s|^{2H-2}\,ds\,dr \\
&\qquad + 2H \int_0^t \Psi(r,r) \bigl(|t-r|^{2H-1} + |r|^{2H-1}\bigr)\,dr \eqdef I_1 + I_2\;.
\end{equs}
We then have
\begin{equs}
|I_1| &\lesssim \hat C\int_0^t r^\beta\int_0^t  \bigl(1 \wedge (|s-r|/\eps)^\gamma\bigr)
|r-s|^{2H-2}\,ds\,dr 
\\&\qquad + \bar C \eps^{\beta-\zeta} \int_0^t \int_0^t |r-s|^{2H+\zeta-2}\,ds\,dr 
\\
&\lesssim \hat C \eps^{2H-1} \int_0^t r^\beta \int_{\R} (1\wedge |u|^\gamma) |u|^{2H-2}\,du \,dr 
+ \bar C \eps^{\beta-\zeta} t^{2H+\zeta}\\
&\lesssim \hat C \eps^{2H-1} t^{\beta+1}+ \bar C \eps^{\beta-\zeta} t^{2H+\zeta}\;.
\end{equs}
We used the conditioned imposed on $\beta, \gamma$ and $\zeta$. Regarding $I_2$, we have the bound
\begin{equs}
|I_2| &\lesssim C\int_0^t |r|^\beta \bigl(|t-r|^{2H-1} + |r|^{2H-1}\bigr)\,dr \\
&= C t^{2H+\beta} \int_0^1 |r|^\beta \bigl(|1-r|^{2H-1} + |r|^{2H-1}\bigr)\,dr\;,
\end{equs}
and the claim follows.
\end{proof}

Note that replacing $\Psi(r,s)$ by $\Psi_\tau(r,s) = \tau^{2H}\Psi(\tau r,\tau s)$ and
$t$ by $t/\tau$, the  left-hand side of (\ref{e:boundDInt}) is left unchanged. 
Regarding the bounds (\ref{e:assumePsi}), such a change leads to the substitutions $C \mapsto C \tau^{2H+\beta}$, 
$\eps \mapsto \eps/\tau$, $\hat C \mapsto \hat C \tau^{2H+\beta}$, and $\bar C \mapsto \bar C \tau^{2H+\beta}$. 
All three terms appearing in the right-hand side of (\ref{e:boundDInt}) are invariant under 
these substitutions.

\begin{remark}\label{rem:betterLemma}
The proof of Lemma~\ref{lem:boundIntEta''} works \textit{mutatis mutandis} for $\Psi$ taking values in
a Banach space, for example $L^p$. We also see that if $\Psi$ is upper bounded by a finite sum of terms of the type \eqref{e:assumePsi} with different exponents $\beta$ and $\gamma$, then 
the bound \eqref{e:boundDInt} still holds with the corresponding sum in the right-hand side.
\end{remark}

We perform a number of preliminary calculations. For this, it will be notationally
 convenient to introduce the shortcuts
\begin{equ}
I^\eps_{s,t}(f)
=  \int_s^t f(Y^\eps_r)\,dB_r\;,\qquad J^\eps_{s,t}(f) =  Z^{f(Y^\eps_\cdot)}_{s,t} = \eps^{\f12-H} I^\eps_{s,t}(f) \;,
\end{equ}
for $f \in E$ (with values in $\R^m$).

\begin{proposition}\label{prop:tight1st}
Let $H\in (\f 13, \f 12)$ and let  Assumptions~\ref{ass:continuous} and~\ref{ass:integrability} hold for some $p_\star\ge 2$. 
Then there exists a constant $C$ such that, uniformly over $s\ge 0$, $t\ge 0$, and $f\in E$,
\begin{equs}\label{e:Lp-bounds-J}
 \|J^\eps_{s,t}(f)\|_{L^{p_\star}}\le    C \|f\|_E \bigl(\eps^{{1\over 2}-H} |t-s|^{H}
\vee \sqrt{|t-s|}\bigr)\;.
\end{equs} \end{proposition}

\begin{proof} 
Let $\CF^Y$ denote the $\sigma$-algebra generated by all point evaluations of the process $Y$, and $\CF^Y_t$ the corresponding filtration. Write $p$ for $p_\star$ for brevity.
Since $B$ is independent of $\CF^Y$ and the $L^p$ norm of an element of a Wiener chaos of fixed degree
is controlled by its $L^2$ norm,  we have
\begin{equ}[e:boundLpChaos]
\|I^\eps_{s,t}(f)\|_{L^p}^2 = \big|\E \big(\E \bigl(I^\eps_{s,t}(f)^p\,|\, \CF^Y\bigr)\big)\big|^{2/p} 
\le c \|\E(I^\eps_{s,t}(f)^2\,|\, \CF^Y)\|_{L^{p/2}}\;,
\end{equ}
for some universal constant $c$ depending only on $p$ and on the degree of the Wiener chaos, so that
\begin{equ}[e:niceRHS]
\|I^\eps_{s,t}(f)\|_{L^p}^2 
\lesssim
 \Big\|  \int_s^t\int_s^t f(Y_r^\epsilon)f(Y_{r'}^\eps) \eta''(r-r')\, dr\, dr' \Big\|_{L^{p/2}}\;.
\end{equ}
Since $f\in E$  is in $L^{p}$ by Assumption~\ref{ass:integrability}, it follows from Assumption~\ref{ass:continuous}  that
\begin{equ}
\bigl\|f(Y^\eps_{u})f(Y^\eps_{u+v}) - f(Y^\eps_{u})^2\bigr\|_{L^{p/2}}
\lesssim \|f\|_E^2  \bigl({|v/\eps|^H} \wedge 1\bigr)\;.
\end{equ}
We can therefore apply Lemma~\ref{lem:boundIntEta''} with $\gamma = H$ and $\beta = 0$ so 
that, for $\|f\|_E \le 1$, one has
\begin{equ}
\|I^\eps_{s,t}(f)\|_{L^p}^2 \lesssim  \eps^{2H-1}|t-s|
+ |t-s|^{2H}\;,\label{e:Lp-bounds}
\end{equ}
whence the desired bound follows. (The condition $\gamma > 1-2H$ is satisfied since
$H > \f13$ by assumption.)
\end{proof}

We now consider the second-order process $\JJ$ given by 
\begin{equ}[e:defJJ]
\JJ^\eps_{s,t}(f,g) = \eps^{1-2H} \ZZ^{f(Y^\eps_\cdot), g(Y^\eps_\cdot)}_{s,t}
= \eps^{1-2H}\int_s^t\int_s^v f(Y^\eps_u)\, dB(u)\,  g(Y^\eps_v)\,dB(v)\;,
\end{equ}
and bound it in a similar way. 
Recalling that $ \scal{f,g}_\mu=\int_\CY \<f, g\>d\mu$, we first obtain a bound on its expectation.

\begin{proposition}\label{prop:boundExpectJJ}
Let $H\in (\f 13, \f 12)$, let  Assumptions~\ref{ass:continuous} and~\ref{ass:integrability} hold for some $p_\star\ge 2$, and let $f ,g\in E$. One has 
\begin{equ}
\| \E \big(\JJ^{\eps}_{s,t}(f,g)\,|\, \CF^Y\bigr)\|_{L^p}
\lesssim \|f\|_E\|g\|_E \bigl(  \epsilon^{1-2H} |t-s|^{2H} \vee |t-s| \bigr)\;,
\end{equ}
provided that $2p \le p_\star$.
\end{proposition} 

\begin{proof}
It follows from \eqref{e:defJJ} that we have the identity
\begin{equ}
\E \big(\JJ^{\eps}_{s,t}(f,g)\,|\, \CF^Y\bigr)
= {\eps^{1-2H} \over 2} \int_s^t\int_s^v f(Y^\eps_u)  g(Y^\eps_v)\,\eta''(u-v)\, du\, dv\;,
\end{equ}
and we conclude from Lemma~\ref{lem:boundIntEta''} and the bound
$\bigl\|g(Y^\eps_{u})(f(Y^\eps_{u+w}) - f(Y^\eps_{u}))\bigr\|_{L^{p}}
\lesssim \|f\|_E\|g\|_E  \bigl({|w/\eps|^H} \wedge 1\bigr)$ exactly as above.
\end{proof}

\begin{proposition}\label{prop:tight2nd}
Let  $H\in (\f 13, \f 12)$, let Assumptions~\ref{ass:continuous} and~\ref{ass:integrability} hold for some $p_\star\ge 2$, 
let $f,g \in E$, and let $2p\le  p_\star$. Then there exists a constant $C$ such that
\begin{equ}
\big\|\JJ^\eps_{s,t}(f,g)\big\|_{L^p}
 \le C \|f\|_E\|g\|_E \bigl(  \epsilon^{1-2H} |t-s|^{2H} \vee |t-s| \bigr)\;.
\end{equ}
If $g$ is a constant, one obtains a stronger upper bound of the form $$\|f\|_E|g| \big(\eps^{{1\over 2}-H} |t-s|^{\f12 +H} \vee \eps^{1-2H}|t-s|^{2H}\big).$$
\end{proposition}

\begin{proof}
By Proposition~\ref{prop:boundExpectJJ}, it suffices to obtain a bound on
\begin{equ}
\big\|\JJ^\eps_{s,t}(f,g) - \E(\JJ^\eps_{s,t}(f,g)\,|\,\CF^Y)\big\|_{L^p}\;.
\end{equ}
As a consequence of Proposition~\ref{prop:contractWiener} and \eqref{e:boundLpChaos}, we  have the bound
\begin{equ}
\big\|\JJ^\eps_{s,t}(f,g) - \E(\JJ^\eps_{s,t}(f,g)\,|\,\CF^Y)\big\|_{L^p}
\le \sqrt 2 \big\|\tilde \JJ^\eps_{s,t}(f,g)\|_{L^p}\;,
\end{equ}
where we set
\begin{equ}
\tilde \JJ^\eps_{s,t}(f,g) = \eps^{1-2H}\int_s^t\int_s^v f(Y^\eps_u)\, dB(u)\,  g(Y^\eps_v)\,d\tilde B(v)\;,
\end{equ}
for a fractional Brownian motion $\tilde B$ independent of $B$ (and $Y$).
We furthermore restrict ourselves to the case $s=0$ and $m=1$ without loss of generality.

At this point we note that for every $H > \f13$, one has the identity
\begin{equ}[e:VarianceJJtilde]
\E \Big(\big(\tilde \JJ^\eps_{0,t}(f,g)\big)^2\,\Big|\, \CF^Y\Big) = \f12 \int_0^t \!\int_0^t\phi_\eps(s,s') \eta''(s-s')\,ds\,ds'\;,
\end{equ}
where we have set
\begin{equ}
\phi_\eps(s,s') = {\eps^{2-4H}\over 2} g(Y_s^\eps)g(Y_{s'}^\eps) \int_0^s \int_0^{s'} f(Y_r^\eps) f(Y_{r'}^\eps)\,\eta''(r-r')\,dr\,dr'\;.
\end{equ}

As a consequence of \eqref{e:boundLpChaos}, we deduce from \eqref{e:VarianceJJtilde} the bound
\begin{equ}[e:LpJJtilde]
\|\tilde \JJ^\eps_{0,t}(f,g)\|_{L^p}^2 \lesssim  \Big\|\int_0^t \!\int_0^t\phi_\eps(s,s') \eta''(s-s')\,ds\,ds'\Big\|_{L^{p/2}}\;.
\end{equ}
We now bound $\phi_\eps$ in such a way that Lemma~\ref{lem:boundIntEta''} (combined with Remark~\ref{rem:betterLemma}) applies with $\Psi = \phi_\eps$.
In order to apply this result, we first verify that the first bound in \eref{e:assumePsi} holds.
Applying Hölder's inequality we obtain for $2p \le p_\star$,
\begin{equ}[e:boundPhi]
\big\| \phi_\eps(s,s) \big\|_{L^{p/2}}
\le    C\epsilon^{2-4H}  \|g(Y^\eps_s)\|_{L^{2p}}^2 \Big\|\int_0^s \int_0^{s} f(Y_r^\eps) f(Y_{r'}^\eps)\,\eta''(r-r')\,dr\,dr'\Big\|_{L^{p}}\;.
\end{equ}
Since the last factor is the same expression as the right-hand side of \eqref{e:niceRHS},
it is bounded as in \eqref{e:Lp-bounds}, thus yielding
\begin{equ}
\big\| \phi_\eps(s,s) \big\|_{L^{p/2}}
\lesssim  \epsilon^{2-4H}   \|f\|^2_E \|g\|^2_E  \bigl(s^{2H}\vee \eps^{2H-1} s \bigr)\;.
\end{equ}

Regarding the second bound in \eref{e:assumePsi}, we note that, for $s'\ge s$ and $\alpha>\f 13$, one has 
\begin{equs}
\phi_\eps(s,s') &- \phi_\eps(s,s) =
\eps^{2-4H} g(Y_s^\eps)g(Y_{s'}^\eps) \int_0^s \int_s^{s'} f(Y_r^\eps) f(Y_{r'}^\eps)\,|r-r'|^{2H-2}\,dr\,dr' \\
&\quad + \eps^{2-4H} g(Y_s^\eps)\bigl(g(Y_{s'}^\eps)-g(Y_{s}^\eps)\bigr) \int_0^s \int_0^{s} f(Y_r^\eps) f(Y_{r'}^\eps)\,\eta''(r-r')\,dr\,dr'\;.
\end{equs}
Since $2p\le p_\star$ and $\int_0^s \int_s^{s'} |r-r'|^{2H-2}\,dr\,dr' \lesssim |s'-s|^{2H}$,  the $L^{p/2}$ norm of the first term is of order 
$\eps^{2-4H} \|f\|^2_E\|g\|^2_E|s'-s|^{2H}$.
By Hölder's inequality, the second term is bounded similarly to before by
\begin{equ}
\epsilon^{2-4H}  \|g(Y^\eps_s)\|_{L^{2p}}\|g(Y^\eps_s)-g(Y^\eps_{s'})\|_{L^{2p}} \Big\|\int_0^s \int_0^{s} f(Y_r^\eps) f(Y_{r'}^\eps)\,\eta''(r-r')\,dr\,dr'\Big\|_{L^{p}}\;.
\end{equ}
By Assumptions~\ref{ass:continuous}
 and~\ref{ass:integrability}, the factors involving $g$ are bounded by
\begin{equ}
\|g\|_E^2 \big({|s'-s|^H\eps^{-H}}\wedge 1\big)\;,
\end{equ}
while the remaining factor is the same is in \eqref{e:boundPhi}, thus yielding a bound of the order 
$$ \|\phi_\eps(s,s') - \phi_\eps(s,s)\|_{L^{p/2}}\lesssim \epsilon^{2-4H}   \|f\|^2_E   \|g\|_E^2\bigl(s^{2H}\vee \eps^{2H-1} s \bigr)\big({|s'-s|^H\eps^{-H}}\wedge 1\big)\;.$$
Applying Lemma~\ref{lem:boundIntEta''} (and Remark~\ref{rem:betterLemma}) and inserting
the resulting bound into \eqref{e:LpJJtilde}, eventually  yields the bound
\begin{equ}[e:upperBound]
 \|\tilde \JJ^\eps_{0,t}(f,g)\|_{L^p}^2 \lesssim \|f\|^2_E\|g\|^2_E \bigl( \epsilon^{2-4H} |t|^{4H} + \eps^{1-2H} |t|^{1 +2H} + |t|^2\bigr)\;,
\end{equ}
as desired. (Note that the second term is bounded by the first and the last one which
is why it was omitted in the statement.)

In case $g$ is a constant, the second term in the expression for $\phi_\eps(s,s') - \phi_\eps(s,s)$
vanishes identically. Since this is the term responsible for the summand proportional to $|t|^2$ in 
\eqref{e:upperBound}, the claim follows.
\end{proof}

\subsection{The regular case \TitleEquation{H\in ( \f 12, 1)}{H in (1/2,1)}}
\label{sec:tightHigh}
For the case where  the slow variables are driven by a fractional Brownian motion of higher regularity, $H>\f 12$,
we exploit the ergodicity of the fast motion even for proving tightness for the first order processes.
 
  To prove the tightness of the processes $\Z_t^\epsilon$, we take a different strategy and estimate higher order
  moments of the $Z^\eps_{s,t}$ and $\ZZ^\eps_{s,t}$. This requires us to estimate the expectation of multiple integrals of the form
  \begin{equ}[e:basicIntegral]
 \int_0^t \dots \int_0^t \prod_{i=1}^{2p}  f_i(Y^\eps_{t_i} )\,dB_{t_1}\cdots dB_{t_{2p}}\;.
  \end{equ}
For the second order processes, half of the upper limits of the integrals are given by one of the $t_i$'s, 
but since we will not need to exploit any cancellations these integrals
 are controlled by the bounds on the hypercube. For $p=1$, it is easy to see that this integral is of order $\eps^{2H-1}t$,
but the case $p=2$ is already more complicated:
 \begin{equs}
 {}&\E\int_0^t\dots \int_0^t \prod_{i=1}^4  f_i(Y^\eps_{t_i} )\,dB_{t_1}\cdots dB_{t_{4}} \label{e:multipleInt}\\
 &=C \int_0^t\dots\int_0^t  \E \Big(\prod_{i=1}^4  f_i(Y^\eps_{t_i})\Big) |t_{2}-t_{1}|^{2H-2}|t_{4}-t_{3}|^{2H-2}\, dt_1\cdots dt_4\;.
 \end{equs}
If we look at the regime $t_1<t_2<t_3<t_4$ say and write $P^\eps_t = P_{t/\eps}$, the first factor is given by
$$ 
\E \Big(\prod_{i=1}^4  f_i(Y^\eps_{t_i})\Big)= \int f_1 P^\eps_{t_2-t_1}\big( f_2 P^\eps_{t_3-t_2} (f_3 P^\eps_{t_4-t_3} f_4) \big)\,d\mu\;.
$$
Since $f_{3,4} = f_3P^\eps_{t_4-t_3} f_4$ is no longer centred, we unfortunately do not have very good bounds on this expression. One can however do better than $\exp(-c|t_4-t_3|/\eps)$: subtracting and adding the mean of $f_{3,4}$, we can write the expression as
\begin{equ}[e:expression]
\int f_1 P^\eps_{t_2-t_1}\big( f_2 P^\eps_{t_3-t_2} (f_{3,4} - \bar f_{3,4}) \big)\,d\mu + \bar f_{3,4} \,\bar f_{1,2}\;,
\end{equ}
where now the first term is bounded by $\exp(-c|t_4 - t_2|/\eps)$ and the second term is bounded
by $\exp(-c|t_4 - t_3|/\eps-c|t_2 - t_1|/\eps)$. This is still not optimal: we note this time that we can 
recenter $f_{2,3,4} = f_2 P^\eps_{t_3-t_2} (f_{3,4} - \bar f_{3,4})$ ``for free'' since $f_1$ has mean zero,
so the first term is actually of order $\exp(-c|t_4 - t_1|/\eps)$.
It is then not too difficult to see that,  the contribution 
of the second term of \eqref{e:expression}  to the integral \eqref{e:multipleInt} is of order $\eps^{4H-2} t^2$, while the contribution of the first term is $\eps^{4H-1} t$, which is of lower order for $t \ge \eps$.  Our aim is to generalise such considerations to arbitrarily high moments.

In particular, the ``correct'' way of rewriting the factor $\E \big(\prod_{i=1}^{2p}  f_i(Y^\eps_{t_i})\big)$ so that it yields usable bounds is in terms of its cumulants.
Given a collection $\{X_i\}_{i \in I}$ of random variables and a subset $A \subset I$, we write $X_A$ as
a shorthand for the collection $\{X_i\}_{i \in A}$ and $X^A$ as a shorthand for $\prod_{i \in A} X_i$.
Given a finite set $A$, we write $\CP(A)$ for the set of partitions of $A$.
We also write $\E_c X_A$ for the joint cumulant, 
so that one has the identities
\begin{equ}[e:idenCumul]
\E X^I = \sum_{\Delta \in \CP(I)} \prod_{A \in \Delta} \E_c X_A\;,\qquad
\E_c X_I = \sum_{\Delta \in \CP(I)} C_\Delta \prod_{A \in \Delta} \E X^A\;,
\end{equ}
where $C_\Delta = (|\Delta|-1)! (-1)^{|\Delta|-1}$. 
 
\begin{proposition}\label{prop:boundCumul}
Let Assumptions~\ref{ass:ergodic} and~\ref{ass:integrability} hold for $H>\f 12$.
For any $k \ge 2$ and any $s_1, \dots, s_k\in [0,t]$,  there exist constants $c, C > 0$ such that the following holds.
Let $f_1,\ldots, f_k \in E $ with $\int_\CY f_i \,d\mu = 0$, let $s_1 < \ldots < s_k$, and set $X_i = f_i(Y_{s_i})$. 
Then, one has the bound
\begin{equ}
\E_c X_{[k]} \le C \exp\Big(-c \sum_{i,j\le k} |s_i-s_j|\Big)\prod_{i} \|f_i\|_E\;
\end{equ}
where $[k]$ denotes the set $\{1, \dots, k\}$.
\end{proposition}

\begin{proof}
Note first that since $c$ is allowed to depend on $k$, it actually  suffices to show that 
$\E_c X_{[k]} \le C \exp\big(-c \sup_{i < k} |s_{i+1}-s_i|\big)$. From now on we fix $i_\star \in \{1,\ldots,k\}$ to be the index
which realises that supremum. Let  $\tilde Y$ be an independent copy of $Y$ and set 
\begin{equ}
\tilde X_j = 
\left\{\begin{array}{cl}
	f_j(Y_{s_j})  & \text{if $j \le i_\star$,} \\
	f_j(\tilde Y_{s_j}) & \text{otherwise.}
\end{array}\right.
\end{equ}
The most important property of the joint cumulant of a collection of random variables is that if it can be broken into
two independent sub-collections, then the joint cumulant vanishes. As a consequence, we have
\begin{equ}
\E_c X_{[k]} = \E_c X_{[k]} - \E_c \tilde X_{[k]}
= \sum_{\Delta \in \CP(I)} C_\Delta \Big(\prod_{A \in \Delta} \E X^A - \prod_{A \in \Delta} \E \tilde X^A\Big)\;.
\end{equ}
We now put a total order on the elements of a partition $\Delta$ by postulating that $A_1 \le A_2$ whenever 
$\inf \{a \in A_1\} \le \inf \{a \in A_2\}$ (this is just for definiteness, the actual choice of order is unimportant).
We can then write the above as a telescoping sum, yielding
\begin{equ}\label{e:cumulant}
\E_c X_{[k]} = 
\sum_{\Delta \in \CP(I)} C_\Delta \sum_{A \in \Delta}  \Bigl(\E X^A - \E \tilde X^A\Bigr)  \Big(\prod_{B < A,\,B\in \Delta} \E X^B\Big)\Big(\prod_{B > A,\, B\in \Delta} \E \tilde X^B\Big)\;.
\end{equ}

We  fix $A \subset [k]$ such that $\E X^A \neq \E \tilde X^A$ and write $A = \{a_1,\ldots,a_\ell\}$ with 
$\ell = |A|$ and $i \mapsto a_i$ increasing.
We also write $j_\star < \ell$ for the index such that $a_{j_\star} \le i_\star$ and $a_{j_\star+1} > i_\star$.
(This necessarily exists since otherwise $\E X^A = \E \tilde X^A$.) For $i < \ell$ and $n \ge 1$, 
we also write $T_i \colon E_n \to E_{n+1}$ for the operator given by
\begin{equ}
T_i g = f_{a_i} P_{t_i} g\;,\qquad t_i \eqdef s_{a_{i+1}}-s_{a_i}\;,
\end{equ}
whose norm, as an operator from $E_n$ to $E_{n+1}$,  is bounded by a
(possibly $n$-dependent) multiple of $\|f_{a_i}\|_E$, since it is of order
$\|f_{a_i}\|_E \,e^{-c t_i}$ from $E \times E_n$ to $E_{n+1}$,
 when restricted to functions of vanishing mean, by \Cref{ass:ergodic}. We used the continuity of  multiplication of functions.
It then follows from the Markov property that 
\begin{equ}[e:idenXA]
\E X^A = \int_{\CY} T_1\ldots T_{\ell-1} f_{a_\ell} \,d\mu\;,
\end{equ}
(this is easily shown by induction over $\ell$) while we similarly have by the definition of $\tilde X$
\begin{equ}[e:idenXtilde]
\E \tilde X^A = \int_\CY T_1\ldots T_{i_\star-1} f_{a_{i_\star}} \,d\mu\, \int_\CY T_{i_\star+1}\ldots T_{\ell-1} f_{a_\ell} \,d\mu\;.
\end{equ}
This is because, setting $A_1 = \{a_1,\ldots, a_{i_\star}\}$ and $A_2 = A \setminus A_1$,
one has $\E \tilde X^A  = \E X^{A_1} \, \E X^{A_2}$ by the definition of $\tilde X$, so that 
\eqref{e:idenXtilde} follows from \eqref{e:idenXA}. 
Writing $g = T_{i_\star+1}\ldots T_{\ell-1} f_{a_\ell}\in E_{\ell-i_\star}$, 
it follows that
\begin{equ}
\E X^A - \E \tilde X^A = \int T_1\ldots T_{i_\star} \Bigl(g - \int g \,d\mu\Bigr)\,d\mu\;.
\end{equ}
The spectral gap assumption \eqref{e:spectralGap} and the definition of $i_\star$ then imply that 
\begin{equ}
\big|\E X^A - \E \tilde X^A\big| \le C \exp(- c |s_{i_\star+1}-s_{i_\star}|) \prod_{a \in A} \|f_a\|_E\;.
\end{equ}
Combining this with \eqref{e:cumulant} immediately leads to the claimed bound
on the corresponding cumulant.
\end{proof}

The first identity of \eqref{e:idenCumul} combined with Wick's formula for the moments of Gaussians now
suggest that we should rewrite the expectation of \eqref{e:basicIntegral} as a sum over terms 
indexed by pairs $(\Delta,\pi)$ where 
$\Delta$ is a partition of $[2p]$ arising from \eqref{e:idenCumul} and representing a product of 
cumulants of the $f(Y_{t_i})$ and $\pi$ is a pairing of 
$[2p]$ arising from Wick's formula.

Figure~\ref{fig1} for example represents the pairing $\pi$ and partition $\Delta$ given by
\begin{equs}[e:examplePart]
\pi &= \{(1,2), (3,4), (5, 6), (7,8), (9, 10)\}\;,\\
\Delta &= \{\{1,3\},\{2,4, 5\},\{6, 7, 8\},\{9,10\}\}
\end{equs}
Each pairing $(i,j)$ yields a factor $|s_i-s_j|^{2H-2}$ while each element $B$ of the partition 
yields an exponential factor of the form
$\prod_{i,j \in B}\exp\big(-c  |s_i-s_j|\big)$ thanks to Proposition~\ref{prop:boundCumul}.
Since we consider the case $H > \f12$, this yields a locally integrable function 
in the expression for the expectation of \eqref{e:basicIntegral}, so our 
analysis mainly focuses on the large-scale behaviour. 
We will show then that the terms with $\Delta = \pi$,  yield a contribution of order $\eps^{(2H-1)p} t^p$ 
which dominates our bound, while all other terms are of higher order in $\eps/t$.
We now proceed to formalising this.

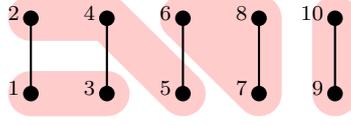
\begin{figure}
\begin{center}
\begin{tikzpicture}
\foreach \i in {1,...,10}
{
        \pgfmathtruncatemacro{\x}{(\i-1)/2};
        \pgfmathtruncatemacro{\y}{\i - 2*\x};
        \coordinate (\i) at (\x,\y);
}
\draw[partition] (1) -- (3);
\draw[partition] (2) -- (4) -- (5);
\draw[partition] (6) -- (7) -- (8) -- (6);
\draw[partition] (9) -- (10);
\foreach \i in {1,...,5}
{
    \pgfmathtruncatemacro{\pairstart}{2*\i-1};
    \pgfmathtruncatemacro{\pairend}{2*\i};
	\draw[pairing] (\pairstart) -- (\pairend);
}
\foreach \i in {1,...,10}
{
	\coordinate[label={[left]:${}^{\i}$}] (a) at (\i);
}
\end{tikzpicture}
\end{center}
\caption{Example of pairing and partition of $10$ elements.}\label{fig1}
\end{figure}

Let $\CG = (\CV,\CE)$ be a graph with vertex set $\CV$ and 
edge multiset $\CE$ (multiple edges are allowed). Edges $e \in \CE$ are oriented from $e_-$ to $e_+$
and we only consider graphs with $e_+ \neq e_-$. We also label the vertices by two exponents $\alpha_\pm\colon \CE \to \R_-$.
Finally, we assume that we have a ``kernel assignment'',
i.e.\ a collection of functions $K_e\colon \R \to \R$ (with $e \in \CE$) such that 
\begin{equ}\label{Ke}
|K_e(t)| \le C \bigl(|t|^{\alpha_-(e)}\one_{|t| \le 1} + |t|^{\alpha_+(e)}\one_{|t| \ge 1}\bigr)\;.
\end{equ}
We denote by $\|K_e\|_e$ the smallest possible constant $C$ appearing in the above expression.
For those who are not familiar with such graph representations, such a graph will be used to encode 
expressions of the type
\begin{equation}
I_\CG := \int_{\R^{nd}} \prod_{e \in \CE} K_e (s_{e_+} - s_{e_-}) \, \varphi(x_1, \dots, x_k) dx\;.
\end{equation}
 Each of its  nodes $u$ represents an integration variable $s_u$, each edge $e$ represents a
 factor $K_e$, and the resulting expression is integrated against some bounded 
function $\varphi$. The exponents $\alpha_-(e)$ and $\alpha_+(e)$ indicate the singularity 
of $K_e$ at $0$ and at infinity respectively.

\begin{definition}
We say that such a labelled graph is ``regular'' if, for every subset $\CV_0 \subset \CV$, we have
\begin{equ}[e:Weinberg]
\sum_{e \in \CE\,:\, e_\pm \in \CV_0} \alpha_-(e) + |\CV_0| > 1\;.
\end{equ}
\end{definition}The significance of this condition (also called Weinberg's condition) is that it guarantees that 
the function $K^\CG$ on $\R^\CV$ given by 
\begin{equ}\label{kernel}
K^\CG(s) = \prod_{e \in \CE} K_{e}(s_{e_+} - s_{e_-})
\end{equ}
is locally integrable \cite{Weinberg} where $s_{e_{\pm}}$ is the $e_{\pm}$ component of $s\in \R^\CV$.
 (See also \cite[Prop.~2.3]{Feynman} for a formulation closer to 
the one given here.)

We will be mainly interested in the large-scale behaviour here. To describe this, consider a partition
$\CP$ of $\CV$. We say that such a partition is \textit{tight} if there exists $A \in \CP$ such that 
$A\cap \CV_i \neq \emptyset$ for every connected component $\CV_i$ of $\CG$.  Given $\CP$, we then
also write $u \sim v$ if there exists $A \in \CP$ with $\{u,v\} \subset A$.
\begin{definition}\label{Wein}
We then say that a labelled graph as above is ``integrable'' if 
\begin{equ}[e:largeScale]
\sum_{e \in \CE\,:\, e_- \not \sim e_+} \alpha_+(e) + |\CP| < 1\;,
\end{equ}
for every tight partition $\CP$. (Note the similarity with Weinberg's condition.)
\end{definition}
The following is then an immediate consequence of \cite[Thm~4.3]{Feynman}.

\begin{proposition}\label{prop:largeScale}
Let $\CG$ be a regular and integrable graph with $m$ connected components.
Then, there exists $C$ depending on $\CG$ such that 
\begin{equ}
\int_{[-L,L]^{\CV}} |K^\CG(s)|\,ds \lesssim L^m \prod_{e \in \CE} \|K_{e}\|_{e}\;,
\end{equ} 
uniformly over $L \ge 1$, with proportionality constant depending only on the labelled graph $\CG$, where 
$ \|K_{e}\|_{e}$ is as defined by (\ref{Ke}).
\end{proposition}

\begin{remark}
Our bound on the large-scale behaviour of the kernels $K_\mft$ is weaker than the 
bound \cite[Eq.~4.1]{Feynman} since we assume no bounds on the derivatives. 
The reason why the result still holds is that we assume
local integrability, which avoids all renormalisation issues and therefore gets rid of regularity requirements.
\end{remark}

An immediate, but very useful, corollary is the following. 

\begin{corollary}\label{cor:defect}
Let $\CG$ be a regular graph with $m$ connected components and let $L \ge 1$. Let $\beta \colon \CE \to \R_+$ be such that  
\begin{equ}[e:largeScale2]
\sum_{e \in \CE\,:\, e_- \not \sim e_+} \bigl(\alpha_+(e) - \beta(e)\bigr) + |\CP| < 1\;,
\end{equ}
for every tight partition $\CP$. 
Then, there exists $C$ depending on $\CG$ and $\beta$ such that 
\begin{equ}
\int_{[-L,L]^{\CV}} |K^\CG(t)|\,dt \leq C L^m \prod_{e \in \CE} L^{\beta(e)}\|K_{e}\|_{e}\;. 
\end{equ} 
\end{corollary}

\begin{proof}
It suffices to note that we can assume that the kernels $K_e$ vanish outside of $[-2L,2L]$ since this does not affect
the value of the integral. If we then consider the graph identical to $\CG$, but with its 
labels replaced by 
$( \alpha_-, \alpha_+ - \beta)$, then
\eqref{e:largeScale2} implies integrability
for the new graph by \eref{e:largeScale}. The local integrability condition (regularity) still holds, so Proposition~\ref{prop:largeScale} applies. It remains to note that since $1 \le (L/|t|)^\beta\one_{1\le |t|\le L}$,
decreasing $\alpha_+(e)$ by $\beta(e)$ in \eqref{Ke} has the effect of 
increasing the norm $\|\cdot\|_e$ by (at most) a factor $(2L)^{\beta(e)}$,
 provided that we do consider functions supported in $[-2L,2L]$.
%
%
\end{proof}

We will make use of the following property.
\begin{lemma}\label{lem:deleteEdges}
Let $\tilde \CG$ be a
 graph obtained by deleting some of the edges of $\CG$ but without changing its connected
 components. If $\tilde \CG$ is integrable, then so is $\CG$ itself. 
\end{lemma}
\begin{proof}
This is immediate from Definition~\ref{Wein}, combined with the fact that the 
$\alpha_+(e)$ are negative by assumption.
\end{proof}

The following simple result will also be useful.
 
\begin{lemma}\label{lem:treeGraph}
If $\alpha_+(e) < -1$ for every edge $e$ of $\CG$, then it is integrable.
\end{lemma}

\begin{proof}
Let $\CP$ be a tight partition of the vertex set $\CV$ of $\CG$ and
let $\CG_\CP= (\CV_\CP,\CE_\CP)$ denote the graph obtained by removing self-loops 
from $\CG/{\sim}$, with 
$\sim$ obtained from $\CP$ as in \eqref{e:largeScale}.
Then $\CG_\CP$ is connected by the definition of tightness  
so that $|\CE_\CP| \ge |\CV_\CP|-1$,  which translates into
$|\{e\in \CE\,:\, e_- \not\sim e_+\}| \ge |\CP|-1$. Since $\alpha_+(e) < -1$ for every edge,
the bound \eqref{e:largeScale}, and therefore the
desired claim, then follow at once.
\end{proof}

We now use these preliminary results both to bound $J$ and $\JJ$ and to determine their limits
in the case $H > \f12$.

Our main technical result is the following bound. 

\begin{proposition}\label{prop:basicBound}
Let Assumptions~\ref{ass:ergodic} and~\ref{ass:integrability} hold for $H>\f 12$ and let
$\kappa \in (0,2-2H)$.
For $f,g \in E$ with $\int f d\mu=\int g d\mu=0$, 
set
$$C(f,g) = \Gamma(1-2H) \bigl(\scal{f,\CL^{1-2H} g}_\mu + \scal{g,\CL^{1-2H} f}_\mu\bigr).$$
Then, for every $p \ge 1$ and $f \in E^{2p}$ with $\int f_i \,d\mu = 0$ for every $i$, setting
$$I_{2p} (f) =\int_{L_1}^{M_1}\cdots \int_{L_{2p}}^{M_{2p}} \Big(\prod_{j=1}^{2p} f_j(Y_{t_j}) \Big)\Big( \prod_{k=1}^p |t_{2k}-t_{2k-1}|^{2H-2}\Big) \,dt_1\cdots dt_{2p},$$
there exists a constant $K>0$ such that
\begin{equ}[e:wantBoundProd]
\Big|I_{2p} (f) - \prod_{k=1}^p  |[L_{2k-1},M_{2k-1}]\cap [L_{2k},M_{2k}]| C(f_{2k},f_{2k-1}) \Big| \le K L^{p-\kappa}\;,
\end{equ}
where $L = \sup_{i} |M_i - L_i| \vee 1$.
\end{proposition}
\begin{proof}
We fix $p$ and write $I$ as a shorthand for $I_{2p} (f)$.
The properties of cumulants show that, setting $X_i = f_i(Y_{t_i})$ as previously, 
$I$ is given by
\begin{equs}
I &= \sum_{\Delta \in \CP([2p])} I_\Delta\;,\\
I_\Delta &\eqdef \int_{L_1}^{M_1}\cdots \int_{L_{2p}}^{M_{2p}} \Big(\prod_{A \in \Delta} \E_c X_A \Big)\Big( \prod_{k=1}^p |t_{2k}-t_{2k-1}|^{2H-2}\Big) \,dt_1\cdots dt_{2p}\;.
\end{equs}
Note first that since the $f_i$ are centred, we have $I_\Delta = 0$ unless $|A| \ge 2$ for every $A \in \Delta$. 
There is furthermore one special partition, namely $\Delta_\star = \big\{\{2k-1,2k\}\,:\, k\in [p]\big\}$.
For the summand generated by this `base' partition we have $I_{\Delta_\star} = \prod_{k=1}^p I(2k-1,2k)$, where we set
\begin{equ}
I(k,\ell) = \int_{L_{k}}^{M_{k}} \int_{L_{\ell}}^{M_{\ell}} \E \big(f_{k}(Y_s)f_{\ell}(Y_t)\big)\,|t-s|^{2H-2}\,dt\,ds\;.
\end{equ}
We then note that, for $a < b$ and $f,g \in E$ centred, it follows from the spectral gap assumption
and the fact that $E, E_1 \subset L^2(\mu)$ by \Cref{ass:integrability}, that
\begin{equ}\label{covariance-large-H}
\Big|\int_a^b \E \big(f(Y_0)g(Y_t)\big)\,|t|^{2H-2}\,dt - C(f,g) \one_{0 \in [a,b]} \Big| \lesssim \|f\|_E\|g\|_E e^{-c(|a| \wedge |b|)}\;,
\end{equ}
for some fixed constant $c$. 
It follows from \eqref{covariance-large-H}  that 
\begin{equ}
\big|I(k,\ell) - |[L_{k},M_{k}]\cap [L_{\ell},M_{\ell}]| C(f_{k},f_{\ell}) \big|
\lesssim \int_{L_{k}}^{M_{k}} e^{-c(|L_\ell - s| \wedge |M_\ell - s|)} \,ds \lesssim 1\;,
\end{equ}
and  that $I_{\Delta_\star}$ differs from the desired expression in the statement
by an error of at most $\CO(L^{p-1})$. 

Since $I = \sum_{\Delta \in \CP([2p])} I_\Delta$ and we already obtained \eref{e:wantBoundProd}
for $I$ replaced by $I_{\Delta_\star}$,
it remains to show that  
$|I_\Delta| \lesssim L^{p-\kappa}$ for every partition $\Delta \neq \Delta_\star$ 
with $\kappa$ as in the statement. 
Fix such a partition $\Delta$ from now on and write again $\sim$ for the equivalence relation induced by $\Delta$ on $[2p]$.
We then define a graph $\CG_\Delta$ with vertex set $\CV = [2p]$ and edge set $\CE = \CE_B \cup \CE_\Delta$, where
\begin{equ}
\CE_B = \{(2k-1,2k)\,:\, k\in [p]\}\;,\qquad \CE_\Delta = \{(u,v)\,:\, u\sim v\}\;.
\end{equ}
We furthermore assign kernels to the edges of $\CG_\Delta$ by
\begin{equ}[e:kernelAssign]
K_e(t) = 
\left\{\begin{array}{cl}
	|t|^{2H-2} & \text{for $e \in \CE_B$,} \\
	e^{-c |t|} & \text{otherwise,}
\end{array}\right.
\end{equ}
so that \Cref{prop:boundCumul} yields the bound
\begin{equ}
|I_\Delta| \lesssim \int_{[0,L]^{2p}} \big|K^{\CG_\Delta}(t)\big|\,dt\;.
\end{equ}
The kernel assignment \eqref{e:kernelAssign} is consistent with the exponents given by
\begin{equ}
\alpha_-(e) = 
\left\{\begin{array}{cl}
	2H-2 & \text{for $e \in \CE_B$,} \\
	0 & \text{otherwise,}
\end{array}\right.\qquad 
\alpha_+(e) = 
\left\{\begin{array}{cl}
	2H-2 & \text{for $e \in \CE_B$,} \\
	-2 & \text{otherwise,}
\end{array}\right.
\end{equ}
whence it immediately follows that $\CG_\Delta$ is regular.

It now remains to find a function $\beta \colon \CE \to \R_+$ allowing us to apply
\Cref{cor:defect} to $\CG_\Delta$. 
For this, we construct a set $\CT$ and set $\beta(e) = 1-\kappa$ for $e \in \CT$ and $0$ otherwise.
To construct $\CT$,
consider the graph $\hat \CG_\Delta$ which has $\hat \CV: = \Delta$ as its
vertex set and such that its edge set $\hat \CE$ is given by
\begin{equ}
\hat \CE = \{(\pi_\Delta(2k-1),\pi_\Delta(2k))\,:\, \pi_\Delta(2k-1) \neq \pi_\Delta(2k)\}\;,
\end{equ}
where $\pi_\Delta \colon [2p] \to \Delta$ maps an element to the unique element of the partition $\Delta$
that contains it. In other words, $\hat \CG_\Delta$ is obtained by quotienting $\CG_\Delta$ by the 
partition $\Delta$ and then removing self-loops. 
Let now $\CT \subset \CE_B$ be such that $\hat \CT = \pi_\Delta \CT$ is a maximal spanning forest 
for $\hat \CG_\Delta$. In the case of \eqref{e:examplePart} for example,
one could take $\CT = \{(1,2),(5,6)\}$. 
With $\kappa$ as in the statement of the proposition, we now set 
$\beta(e) = 1-\kappa$ for $e \in \CT$ and $0$ otherwise.
The reason why this choice of $\beta$ satisfies \eqref{e:largeScale2} for the graph $\CG_\Delta$
 is that by construction the labelling $\gamma = \alpha_+ - \beta$
is such that $\CG_\Delta$ contains a spanning forest $\tilde \CT$ consisting of edges $e$ with $\gamma(e) = 2H-3+\kappa < -1$.  
(To build a reduced set of edges from $\CE = \CE_B \cup \CE_\Delta$, we start with $\CT$ and then connect its components using edges in $\CE_\Delta$.)
It then remains to first apply Lemma~\ref{lem:deleteEdges} to reduce ourselves to 
considering $\CG_\Delta$ and then apply Lemma~\ref{lem:treeGraph}.


Denote now by $m$ the number of connected components of $\hat \CG_\Delta$ and note that 
since every element of $\Delta$ is of size at least $2$, $\hat \CG_\Delta$ has at most $p$ vertices.
It follows that the number of elements in $\CT$ equals at most $p-m$, so that \Cref{cor:defect} yields
the bound
$I_\Delta \lesssim L^m L^{(1-\kappa)(p-m)} = L^{p -\kappa(p-m)}$, which is bounded
by $L^{p-\kappa}$ unless $m=p$.
Since the only
partition $\Delta$ yielding $m=p$ is the complete pairing $\Delta_\star$, the claim follows at once.
\end{proof}

\begin{corollary}\label{tight-estimate-regular}
Let the assumptions of \Cref{prop:basicBound} hold.
For $H \in (\f12,1)$ and $f\in E$ with $\int f(y)\,\mu(dy) = 0$, one has for every $p \ge 1$ the bounds
$$\|J_{s,t}^\eps(f)\|_{L^{2p}}
\lesssim t^{\f 12}\;, \quad 
\|\JJ_{s,t}^\eps(f)\|_{L^{p}}
\lesssim t\;,$$
uniformly over $t \le T$ (for any fixed $T \ge 1$) and $\eps \le 1$.
\end{corollary}
\begin{proof}
For integer $p \ge 1$, we note that as a consequence of Wick's theorem and the fact that $Y$ is independent of
$B$, one has the identity,
\begin{equ}[e:boundJ0t]
\E \bigl(J_{0,t}^\eps(f)\bigr)^{2p}
= C_p \eps^p \int_{[0,t/\eps]^{2p}}
\E \bigl(f(Y_{s_1})\cdots f(Y_{s_{2p}})\bigr)\prod_{k=1}^p \eta''(s_{2k}-s_{2k-1})\,ds\;.
\end{equ}
where $C_p = 2^{-p}(2p-1)!!$.  For $t\ge \eps$ we apply \Cref{prop:basicBound} with $L = t/\eps$, so that
\begin{equs}
\Big|\E \bigl(J_{0,t}^\eps(f)\bigr)^{2p}\Big|
&\lesssim \epsilon^p  (t/\eps)^{p-\kappa}+\epsilon^p  (t/\eps)^{p}\lesssim \eps^\kappa t^{p-\kappa} + t^p \lesssim t^p.
\end{equs}
For $t<\epsilon$, \eqref{e:boundJ0t} is bounded by $C\eps^p\|f\|_{L^{2p}}^{2p}|t/\eps|^{2Hp}\lesssim  t^p$.
One similarly obtains the bound
$\E \bigl(\JJ_{0,t}^\eps(f,g)\bigr)^{p} \lesssim  t^p$, completing the proof.
\end{proof}

\section{Identification of the limit}
\label{sec:limit}

In this section we complete the proof of Theorem~\ref{theo:mainRP} by identifying the
limit (in law) of $\Z^\eps$ as $\eps \to 0$. The proof proceeds in two steps. First, in 
Proposition~\ref{prop:convFirst}, we show that the first-order process $Z$ itself converges in law to
a limit $W$ with covariance given as in \eqref{e:defWGaussian}. In a second step, we then exploit martingale
techniques, and in particular \cite{Jacod-Shiryaev}, to obtain convergence of the second-order
process $\ZZ$ to the limit described in Theorem~\ref{theo:mainRP}. 
Recall that, by \eref{e:defZZ},
\begin{equ}
\big(Z^\eps_{s,t}(x)\big)_i = \eps^{\f12 -H} Z^{f_{i,x}^\eps}_{s,t}\;,\quad
\big(\ZZ^\eps_{s,t}(x,\bar x)\bigr)_{ij} \eqdef \eps^{1 -2H} \ZZ^{f_{i,x}^\eps, f_{j,\bar x}^\eps}_{s,t}\;.
\end{equ}

\begin{proposition}\label{prop3.1}
In the setting of Proposition~\ref{prop:tight}, 
the family of random rough paths $\Z^\eps$ converges in law, as $\eps \to 0$, 
to the unique (in law) random rough path $\Z$ such that the following hold.
The process $Z$ is a $\CB$-valued Wiener process with covariance given by
\begin{equ}[e:covar]
\E \big(Z_{s,t}(x)\otimes Z_{u,v}(\bar x)\big) = |[s,t] \cap [u,v]| \,
\bigl(\Sigma(x,\bar x) + \Sigma(\bar x,x)^\top \bigr) \;,
\end{equ}
with $\Sigma$ as defined in \eqref{e:defSigma}.
The ``second-order'' process $\ZZ$ is the $\CB_2$-valued process such that for any $x,\bar x$ in $\R^d$
\begin{equ}
\ZZ_{s,t}(x, \bar x) = \int_s^t Z_{s,r}(x) \otimes dZ_{s,r}(\bar x) +  (t-s) \Sigma(x,\bar x)\;,
\end{equ}
where the integral is interpreted in the Itô sense.
  \end{proposition}
 \begin{proof}
 The convergence in distribution of any finite collections of the stochastic processes follows from \Cref{convergence-drivers} below. By \Cref{prop:tight}, $(Z^\eps, \ZZ^\eps)$ is tight in  $\Cr^\alpha([0,T], \CB \oplus \CB_2)$ for suitable $\alpha\in (\f 13, H)$,  so the weak convergence holds with respect to the rough path norm on $\Cr^\alpha([0,T], \CB \oplus \CB_2)$.
 \end{proof}
 
\subsection{Law of large numbers}

We will need the following quantitative version of the law of large numbers. 
Let $E \subset E_1\subset E_2$ be Banach spaces of functions $\CY\to \R$ containing constants and such that
pointwise multiplication from $E \times E_1$ into $E_2$ is continuous.  

\begin{lemma}\label{lem:LLN}
Let  $E\subset L^4$ and $E_2 \subset L^2$,
let the spectral gap condition (\ref{e:spectralGap}) hold for $n=1,2$, and let $f, g \in E$. 
Then, the bound
\begin{equ}[e:boundLLN]
\Big\|\int_0^T f(Y_s) g(Y_{s+t})\,ds - T \scal{f, P_t g}_\mu\Big\|_{L^2}
\lesssim \sqrt{(1+t)T}\|f\|_E \|g\|_E\;,
\end{equ}
holds uniformly over $t, T \in \R_+$ with a proportionality constant depending only on the constants 
appearing in the two assumptions.
\end{lemma}

\begin{proof}
Writing $f_s$ as a shorthand for $f(Y_s)$ and similarly for $g$, the 
square of the left-hand side of \eqref{e:boundLLN} is given by
\begin{equ}
2\int_0^T\int_0^r\big( \E \bigl(f_sg_{s+t}f_rg_{r+t}\bigr) - \E \bigl(f_sg_{s+t}\bigr)\E\bigl(f_rg_{r+t}\bigr)\big)\,ds\,dr\;.
\end{equ}
 Since $E\subset L^4$,  H\"older's inequality  shows that the integrand is bounded by some multiple of $\|f\|_E^2 \|g\|_E^2$. It thus follows from the triangle inequality that 
  the required bound follows  for $t \ge T$ so we assume $t \le T$ from now on.
Using the same bound on the integrand, we can further restrict the inner integral to impose
$s + t \le r$ at an additive cost of order at most $tT\|f\|_E^2 \|g\|_E^2$. On that smaller domain, we can then rewrite 
the integrand as
\begin{equ}
\E \big(f_s g_{s+t} \big(P_{r-s-t} (f P_t g) - \scal{f,P_t g}_\mu\big)(Y_{s+t})\big)\;.
\end{equ}
By the spectral gap assumption applied to $fP_tg\in E_2$, we have the bound
\begin{equ}
\|P_{r-s-t} (f P_t g) - \scal{f,P_t g}_\mu\|_{E_2}
\lesssim e^{-c(r-s-t)} \|f P_t g\|_{E_2} 
\lesssim e^{-c(r-s-t)} \|f\|_E \|g\|_{E_1}\;.
\end{equ}
Combining this again with Hölder's inequality, $E\subset L^4$, and $E_2 \subset L^2$,
we conclude that the integrand is of order $e^{-c(r-s-t)} \|f\|_E^2 \|g\|_E^2$, thus 
yielding a contribution to the integral of order $T\|f\|_E^2 \|g\|_E^2$ as desired.
\end{proof}

\subsection{Identification of the first-order process }


We treat separately the cases $H < \f12$ and $H>\f12$, while the case
$H = \f12$ is straightforward and will be considered when we put both cases 
together in \Cref{convergence-drivers} below.

\subsubsection{The low regularity case}

Let $H\in (\f 13, \f12)$.
Conditional on $Y$, the process $J_{s,t}^\epsilon (f)=\eps^{\f 12 -H}\int_s^t f(Y_r^\epsilon)dB(r)$ 
is centred Gaussian. In order to identify its limiting distribution, it 
thus suffices to show that its conditional covariances converge to a limit
that is independent of $Y$. This is the content of the following result. 

\begin{proposition}\label{prop:convFirst}
Let Assumptions~\ref{ass:ergodic} and ~\ref{ass:integrability} hold for $H\in (\f 13, \f 12)$, and  let \Cref{ass:continuous} hold for some $p_\star\ge 4$.
Let $f, g \in  E$    and let $u<v$ and $s < t$. Then, we have
$$ 
\lim_{\eps \to 0}\E\bigl( J^\eps_{s,t}(f)J^\eps_{u,v}(g)\,|\,\CF^Y \bigr) =  
 |[s,t] \cap [u,v]| \,C(f,g)\;,
$$
in $L^2$, where 
$C(f,g) =\f 12 \Gamma(2H+1)\bigl(\scal{f, \CL^{1-2H} g} _\mu+ \scal{\CL^{1-2H} f, g}_\mu\bigr)$.  
\end{proposition}

\begin{proof} 
We work in components it is therefore again sufficient to assume that $m=1$. 
We first consider the case $[u,v] = [s,t]$.
A straightforward calculation similar to the one given in \eqref{e:limitExpZ} shows that, setting $\alpha_H = H(1-2H)$, 
one has the identity
\begin{equation}\label{distributional-derivative2} \begin{aligned}
\E &\bigl(J^\eps_{s,t}(f) J^\eps_{s,t}(g) \,|\, \CF^{Y}  \bigr)\\
&=\f {-\alpha_H}2 \eps^{1-2H}\!\!\! \int_{s-t}^{t-s}\! |v|^{2H-2}\!\!\!  \int_{2s+|v|}^{2t-|v|}  \bigl(f(Y^\eps_{u+v\over 2})g(Y^\eps_{u-v \over 2}) - f(Y^\eps_{\f u 2})g(Y^\eps_{\f u 2})\bigr)\,du\,dv \\
&\qquad  +H\eps^{1-2H}  \int_0^{t-s}u^{2H-1} \bigl((fg)(Y^\eps_{(2t-u)/2}) + (fg)(Y^\eps_{(2s+u)/2})\bigr)\,du\;.
\end{aligned}
\end{equation}
Provided that $2p \le p_\star$, the size of the $L^p$-norm of the last term is at most of order
${\|f\|_E\|g\|_E|t-s|^{2H}\eps^{1-2H}}$, so that it does not contribute to the limit
considered in the statement.
(Note however that in the special case $H = {1\over 2}$ which we do not consider here,
this term is the only surviving one,
which contributes to the fact that the conclusion of the lemma still holds in this case.)

Regarding the first term, we note that changing the sign of $v$ is the same as swapping $f$ and $g$.
Taking $s=0$ (by stationarity) and performing a change of variables, 
it remains to show that, for any fixed $t > 0$,
\begin{equs}\label{in-case}
\epsilon\int_{0}^{\f t\epsilon } 
 \int_{v}^{{2t\over \eps}-v} v^{2H-2} &\bigl(f(Y_{(u+v)/2})g(Y_{(u-v)/2}) - f(Y_{\f u 2})g(Y_{\f u 2})\bigr)\,du \,dv\\
&\to -t \f{ \Gamma(2H+1)}{ H(1-2H)}  \<f, \CL^{1-2H} g\> _\mu
\end{equs}
 in $L^2$ as $\eps \to 0$.
 We set
$$\hat I_\eps = -\f 12 {\alpha_H} \epsilon\int_{0}^{\f t\epsilon } v^{2H-2} G_v\,dv\;,$$
 with
$$
G_{v} = \int_{v}^{{2t\over \eps}-v} \bigl(f(Y_{(u+v)/2})g(Y_{(u-v)/2}) - f(Y_{\f u 2})g(Y_{\f u 2})\bigr)\,du\;.
$$
To show that $
\lim_{\eps \to 0} \hat I_\eps =  \f{t}{2} \Gamma(2H+1)\scal{f, \CL^{1-2H} g} _\mu\;$
in $L^2$, we treat the values of $v$ close to the singularity 
separately from the others, so we fix some (eventually sufficiently small) exponent $\kappa$.
For ``small'' values of $v$, we then have the bound
\begin{equs}
\epsilon \Big\|\int_{0}^{\epsilon^\kappa} v^{2H-2} G_v\,dv\Big\|_{L^2}
&\lesssim t \|f\|_E\|g\|_E \int_{0}^{\epsilon^\kappa}v^{2H-2} \bigl(1 \wedge v^H\bigr)\,dv \\
&\lesssim \eps^{(3H-1)\kappa} t \|f\|_E\|g\|_E\;,
\end{equs}
which converges to $0$ as desired for any fixed $\kappa > 0$ since $H > \f13$.

For the remaining values of $v$, we apply Lemma~\ref{lem:LLN}, which yields the bound
\begin{equ}
\Big \|G_{v} - 2\Big({t\over \eps}-|v|\Big)\bigl(\scal{P_v f,g}_\mu-\scal{f,g}_\mu\bigr)\Big\|_{L^2}
\lesssim \sqrt{\eps^{-1}(1+|v|)(t-\eps |v|)} \|f\|_E\|g\|_E\;.
\end{equ}
For $\kappa<\f 1{ 2(1-2H)}$, we furthermore have the bound 
\begin{equs}
\eps \int_{\eps^\kappa}^{\f{t}{\eps}} & v^{2H-2}\sqrt{\eps^{-1}(1+|v|)(t-\eps |v|)}\,dv
\lesssim \sqrt{\eps t} \int_{\eps^\kappa}^{\f{t}{\eps}} \bigl(v^{2H-2} + v^{2H-\f32}\bigr)\,dv \\
&\lesssim \sqrt{\eps t} \Bigl(\int_{\eps^\kappa}^{\infty} v^{2H-2}\,dv + \int_{0}^{\f{t}{\eps}} v^{2H-\f32}\,dv\Bigr)
\lesssim \eps^{\f12 -\kappa(1-2H)}\sqrt t  + \eps^{1-2H} t^{2H}\;,
\end{equs}
which converges to $0$ as $\eps \to 0$ for every fixed $t$. 
 We conclude that 
\begin{equs}
\lim_{\eps \to 0}\hat I_\eps 
&= -\alpha_H \lim_{\eps \to 0} \int_{\eps^\kappa}^{\f t\eps} v^{2H-2} \big( t- \eps |v|\big)\bigl(\scal{P_v f,g}_\mu-\scal{f,g}_\mu\bigr)\,dv \\
&= -\alpha_H \lim_{\eps \to 0} \int_{0}^{\infty} v^{2H-2} t \bigl(\scal{P_v f,g}_\mu-\scal{f,g}_\mu\bigr)\,dv \\
&=  {t\over 2}\Gamma(2H+1)\scal{f, \CL^{1-2H} g} _\mu\;,
\end{equs} 
holds in $L^2$, as claimed.  

We now consider the case when $[s,t) \cap [u,v) = \emptyset$ and assume without loss of generality
that $t \le u$. We then have
\begin{equs}
\big|\E \bigl(I^\eps_{s,t}(f) I^\eps_{u,v}(g) \,|\, \CF^{Y} \bigr) \big|
&= \alpha_H\Big|\int_s^t \int_{u}^{v} |r-\bar r|^{2H-2} f(Y^\eps_r)g(Y^\eps_{\bar r})\,d\bar r\,dr\Big|\\
&\lesssim \int_s^t \int_{u}^{v} |r-\bar r|^{2H-2}\,d\bar r\,dr < \infty\;.
\end{equs}
Since this is multiplied by $\eps^{1-2H}$, it follows that 
$\E \bigl(J^\eps_{s,t}(f) J^\eps_{u,v}(g) \,|\,\CF^Y \bigr) \to 0$ in that case.
The general case then follows immediately since we have 
$J^\eps_{s,t}(f) = J^\eps_{s,u}(f) + J^\eps_{u,t}(f)$
for any $s\le u\le t$, so that it can be reduced to the two cases we just treated.
\end{proof}

\subsubsection{The high regularity case}
Let $H>\f 12$. We first show that the process $Z$ is Gaussian with covariance given by \eref{e:covar}.
For this we recall relations between cumulants and expectations. Fix a finite set $A$ as well as elements $f_a \in E$ and intervals $[s_a, t_a] \subset \R$
for every $a \in A$.
Given a subset $B \subset A$, we write $\CG(B)$ for the set of pairs $(\Delta,p)$ where 
$\Delta \in  \CP(B)$ is a partition of $B$ without singletons and $p$ is a pairing of $B$ 
(i.e.\ $p\in \CP(B)$ contains only sets of size two).
We also write $[s,t]_B \subset \R^B$ for the domain $\bigtimes_{a \in B} [s_a,t_a]$. 
Given $G = (\Delta, p) \in \CG(B)$, we then set
\begin{equ}
J_G^\eps = (-\eps^{\f 12-H}\alpha_H)^{|B|}\int_{[s,t]_B} \prod_{B' \in \Delta} \E_c \big(f_a(Y_{r_a}^\eps)\,:\, a\in B'\big)  \prod_{\{a,b\} \in p} |r_a - r_b|^{2H-2}\,dr\;.
\end{equ}
In order to extract a formula for the joint cumulants of the $J^\eps_{s,t}$'s, 
we note that if $B_1 \cap B_2 = \emptyset$ and
$G_i \in \CG(B_i)$, one has
\begin{equ}[e:multiplicative]
J_{G_1 \sqcup G_2}^\eps  = J_{G_1}^\eps\cdot J_{G_2}^\eps\;,
\end{equ}
where $G_1 \sqcup G_2 \in \CG(B_1 \sqcup B_2)$ denotes the natural concatenation of $G_1$ and $G_2$.
We furthermore write $\CG_c(B) \subset \CG(B)$ for the set of ``connected'' elements, namely
those pairs $G = (\Delta,p)$ such that $G^\vee \eqdef \Delta \vee p = \{B\}$,
 where we use the usual 
lattice structure of the set of partitions of $B$. 

\begin{lemma} Let \Cref{ass:integrability} holds for $H>\f 12$ and $f_a\in E$, then for every $B\subset A$,
\begin{equ}
\E_c \big(J^\eps_{s_a,t_a}(f_a)\,:\, a \in B\big) = \sum_{G \in \CG_c(B)} J_G^\eps\;.
\end{equ}
\end{lemma}
\begin{proof}
As a consequence of \eqref{e:multiplicative}, we can write
\begin{equs}
\E \prod_{a \in A} J^\eps_{s_a,t_a}(f_a) &= \sum_{G \in \CG(A)} J_G^\eps
= \sum_{\Delta' \in \CP(A)} \sum_{G \in \CG(A)\,:\, G^\vee = \Delta'} J_G^\eps \\
&= \sum_{\Delta' \in \CP(A)} \prod_{B \in \Delta'} \sum_{G \in \CG_c(B)} J_G^\eps\;.
\end{equs}
Comparing this to the first identity in \eqref{e:idenCumul}, we conclude the proof.
\end{proof}

This allows to conclude:
\begin{lemma}\label{lemma:convergence-high}
Let Assumptions \ref{ass:ergodic} and \ref{ass:integrability} hold for $H > \f12$. 
Let $f_a\in E$ with $\int f_a d\mu=0$. Then, for any index set $B$ with more than two elements,
\begin{equ}[e:convZeroCumul]
\E_c \big(\lim_{\eps\to 0} J^\eps_{s_a,t_a}(f_a)\,:\, a \in B\big) =0\;.
\end{equ}
Consequently, the limits $\{J_{s_a,t_a}(f_a)\}_{a\in B}$ are jointly Gaussian with covariance  
$|[s_a,t_a] \cap [s_b,t_b]| \,C(f_a,f_b)$.
\end{lemma} 
\begin{proof}
Let $G = (\Delta,p) \in  \CG_c(B)$ with $|B|$ even and note that $J_G^\eps$ can be written as
\begin{equ}
J_G^\eps
 = (-\alpha_H)^{|B|}\eps^{|B|/2}\int_{[\f{s}{\epsilon},\f{t}{\epsilon}]_B} \prod_{\bar A \in \Delta} \E_c \big(f_a(Y_{r_a})\,:\, a\in \bar A\big)  \prod_{\{a,b\} \in p} |r_a - r_b|^{2H-2}\,dr\;.
\end{equ}
We then note that if $|B| > 2$, elements $(\Delta,p) \in \CG_c(B)$ are always such that 
$\Delta \neq p$.
We can therefore apply  \Cref{prop:basicBound} with $2p=|B|$, which shows that
\begin{equ}
|J_G^\eps| \lesssim \eps^{|B|/2}\eps^{\kappa -\f{|B|}2}\;,
\end{equ}
which converges to $0$, thus yielding \eqref{e:convZeroCumul} as claimed. 
By \Cref{prop:basicBound},
for any  $f_i \in  E$, $u<v$ and $s < t$ we have
$$ 
\lim_{\eps \to 0}\Big|\E\bigl ( J^\eps_{s,t}(f_i)J^\eps_{u,v}(f_j)\bigr) -
 |[s,t] \cap [u,v]| \,C(f_i,f_j)\Big| = 0\;.
$$
Since Gaussian processes are characterised by the fact that their joint cumulants of order three or higher 
all vanish, the last claim follows.
\end{proof}

\subsection{Convergence of the second-order process }
\label{subsection-convergence}

In this section we assume that $H\in (\f 13,1)$. If $H=\f 12$, $J_t^\epsilon= \int_0^t f(Y_r^\epsilon) dB_r$ is already a local martingale, otherwise we make a
decomposition as follows.
Write $\bar B^t_r = \E \bigl(B_r - B_t \,|\, \CF_t^B\bigr)$ for $r > t$ and write
\begin{equ}
J_t^\eps(f) =\eps^{\f12-H} \int_0^t f(Y_r^\epsilon) dB_r=M_t^f + R_t^f\;,
\end{equ}
where, setting $\bar f = \int_\CY f(x)\,\mu(dx)$ and $\tilde f = f - \bar f$,
\begin{equs}
M_t^f &= \eps^{\f12-H}\int_0^t \tilde f(Y_r^\eps)\,dB_r + \eps^{\f12-H}\int_t^\infty \bigl(P_{r-t \over \eps}\tilde f\bigr)(Y_t^\eps)\,d\bar B^t_r\;, \label{e:defMart}\\
R_t^f &= \eps^{\f12-H}\bar f B_t - \eps^{\f12-H}\int_t^\infty \bigl(P_{r-t \over \eps}\tilde f\bigr)(Y_t^\eps)\,d\bar B^t_r\;.\label{e:defRem}
\end{equs}
The convergence of the integral is guaranteed by the fact that $|\dot {\bar B}_t^r|\lesssim {(r-t)^{(H-1)-}}$ 
and  $P_t \tilde f \to 0$
exponentially fast, thanks to the centering condition on $\tilde f$. 
This clearly illustrates why have no need to assume that the coefficient $F(x,\cdot)$ itself is centred for $H < \f12$
since the first term in $R_t^f$ obviously converges to $0$ in that case. For $H=\f 12$, we 
write $M_t^f=J_t^\epsilon(f)$ and $R_t^f=0$.

\begin{lemma}\label{lem:boundR}
Let $H\in (\f 13, 1)$, let $p>1$ and $E\subset L^{2p}$, and let (\ref{e:spectralGap}) holds for $n=1$
For every $f \in E$ and and every $t \ge 0$, one has $\lim_{\eps \to 0} R_t^f = 0$ in $L^q$ for any $q\in [1,p)$.
 Furthermore,   $\bigl \| \eps^{-H}\int_t^\infty \bigl(P_{r-t \over \eps}\tilde f\bigr)(Y_t^\eps)\,d\bar B^t_r\|_{L^q} \lesssim  \|\tilde f\|_E$.
\end{lemma}

\begin{proof}
We only need to prove if for $H\not = \f 12$. The first term in the definition of $R_t^f$ obviously converges to $0$. The scale and shift
invariance of 
fractional Brownian motion shows that, in law, the second term $\eps^{\f 12-H}\int_t^\infty \bigl(P_{r-t \over \eps}\tilde f\bigr)(Y_t^\eps)\,d\bar B^t_r$ equals
\begin{equ}
\eps^{\f12}\int_0^\infty \bigl(P_{r}\tilde f\bigr)(Y)
d\bar B^0_r\;,
\end{equ}
with $Y$ a random variable with law $\mu$, independent of $\bar B^0$. 
Note now that $\dot {\bar B}^0$ is Gaussian and
 \begin{equ}
 \E |\dot {\bar B}^0_r|^2 \propto |r|^{2H-2}\;.
 \end{equ}
Furthermore since $f\in E$, $\|P_tf\|_{L^p}\lesssim |f|_E$ and $\|P_tf\|_{L^1}\to 0$, $\|P_tf\|_{L^q}\to 0$ for any $q\in [1,p)$, so that
by Cauchy--Schwarz, 
\begin{equs}
\Bigl\|\int_0^\infty \bigl(P_{r}\tilde f\bigr)(Y)
d\bar B^0_r\Bigr\|_{L^p} &\lesssim 
\int_0^\infty \|\bigl(P_{r}\tilde f\bigr)(Y)\|_{L^{2p}}|r|^{H-1}\,dr \\
&\lesssim
\|\tilde f\|_E \int_0^\infty e^{-cr} |r|^{H-1}\,dr < \infty\;,
\end{equs}
and the claim follows.
\end{proof}

\begin{lemma}\label{lemma:martingale-limit}
Let $H\in (\f 13, 1) \setminus \{\f12\}$. Let Assumption~\ref{ass:ergodic} hold for $n=1$,
let Assumption~\ref{ass:continuous} hold for some $p_\star> 2$, and let Assumption~\ref{ass:integrability} hold. 
Let $f\in E$, then the process $M_t^f$ as defined above is an $L^p$ bounded  $\CF_t^Y\vee \CF_t^B$-martingale for every $p<p_\star$ with the convention that $p_\star=\infty$ for  $H>\f 12$).
\end{lemma}

\begin{proof}
For $T> 0$, we define the $\CF_t^Y\vee \CF_t^B$-martingale $M_t^{f,T}$ by
\begin{equ}
M_t^{f,T} = \eps^{\f12-H}\E \Bigl(\int_0^T \tilde f(Y_r^\eps)\,dB_r\,\Big|\, \CF_t^Y\vee \CF_t^B\Bigr)\;,
\end{equ}
and we note that for $T > t$ one has
\begin{equ}
M_t^{f,T} = \eps^{\f12-H}\int_0^{t} \tilde f(Y_r^\eps)\,dB_r + \eps^{\f12-H}\int_{t}^T \bigl(P_{r-t \over \eps}\tilde f\bigr)(Y_t^\eps)
d\bar B^t_r\;.
\end{equ}
Since $\|P_t \tilde f\|_{E_1} \to 0$ exponentially fast, it follows  from Lemma~\ref{lem:boundR} that
$M_t^{f} = \lim_{T \to \infty} M_t^{f,T}$ in $L^p$, so that $M_t^{f}$ is a local martingale.
Since the first term of \eqref{e:defMart} is bounded in $L^p$ (by Proposition~\ref{prop:tight1st} for $H\in (\f 13, \f 12)$, and for $H>\f 12$ from 
Corollary~\ref{tight-estimate-regular} which applies since $\tilde f$ is centred), and the second term converges in $L^p$ by Lemma~\ref{lem:boundR}, the claim follows.
\end{proof}

\begin{remark}
For $H<\f 12$, we only used the integrability condition $E\subset L^{p_\star}$ in Lemma~\ref{lemma:martingale-limit},
 the condition $E_2\subset L^2$ from Assumption~\ref{ass:integrability} is not needed (nor is it needed in Proposition~\ref{prop:tight1st}).
\end{remark}

For $H\not =\f 12$, we can then rewrite $\JJ$ as
\begin{equ}[r:idenJJ]
\JJ_{0,t}^\eps(f,g) = \int_0^t J^\eps_s(f)\, dM_s^g + J_t^\eps(f) R_t^g - \int_0^t R_s^g\,dJ_s^\eps(f)\;.
\end{equ}
Here, the first integral is an Itô integral, while the integral  $\int_0^t J_s^\epsilon(f) dR_s^f$  should be interpreted as
the limit, as $\delta \to 0$, of the corresponding expression with $B$ replaced by a mollified
version $B^\delta$, to which integration by parts is applied.

\begin{lemma}\label{lem:convMean} Let $H\in(\f 13, \f 12)\cup(\f 12, 1)$.
Let Assumptions~\ref{ass:ergodic}--\ref{ass:integrability} hold and 
let  $f \in E$. 
Then, one has 
\begin{equ}
\lim_{\eps \to 0}\E\Big( \int_0^t R_s^g\,dJ_s^\eps(f)\;\Big|\; \CF^Y \Big) = -\f t2  \Gamma(2H+1)\scal{\CL^{1-2H}g, f}_\mu\;,
\end{equ}
in probability.
\end{lemma}

\begin{proof}
Let us first write $R_t^g = J_t^\eps(\bar g) - \tilde R_t^g$ where
$$\tilde R_t^g=\eps^{\f12-H}\int_t^\infty \bigl(P_{r-t \over \eps}\tilde g\bigr)(Y_t^\eps)\,d\bar B^t_r\;,
$$
and note that 
\begin{equ}[e:idenJJ]
 \int_0^t J_s^\eps(\bar g)\,dJ_s^\eps(f)
 = J_t^\eps(\bar g)\,J_t^\eps(f) - \JJ_{0,t}^\eps(f,\bar g)\;,
\end{equ}
for $H<\f 12$ (vanishes for $H>\f 12$ since then $\bar g = 0$). Since $J_t^\eps(\bar g) = \eps^{\f12-H} \bar g\, B_t$, we conclude from Proposition~\ref{prop:tight1st}
that the first term on the right hand side converges to $0$ in probability.
Since $\bar g$ is constant, the second part of Proposition~\ref{prop:tight2nd} 
implies that the second term also converges to $0$ 
in probability for $H\in (\f 13, \f12)$, so that it remains to obtain the limit of $\int_0^t \tilde R_s^g\,dJ_s^\eps(f)$.
For $H \neq \f12$ we have the identity
\begin{equs}
\E\Big( \int_0^t \tilde R_s^g\,dJ_s^\eps(f)\,&\Big|\, \CF^Y \Big)
= 
\eps^{1-2H}\E\Bigl( \int_0^t \int_s^\infty \bigl(P_{r-s\over \eps} \tilde g\bigr)(Y_s^\eps)d\bar B_r^s  f(Y_s^\eps)\,dB_s\;\Big|\; \CF^Y\Bigr)\\
&= 
{\eps^{1-2H} \over 2} \int_0^t \int_s^\infty \bigl(P_{r-s\over \eps} \tilde g\bigr)(Y_s^\eps) f(Y_s^\eps)\,\eta''(s-r)\,dr\,ds\\
&=H(2H-1)\int_0^t \int_0^\infty \bigl(P_{r} \tilde g - \one_{H < \f12}\tilde g\bigr)(Y_s^\eps)\,r^{2H-2}\,dr f(Y_s^\eps)\,ds \\
&=\f 12 \Gamma(2H+1)
\int_0^t \big\<\CL^{1-2H}   g, f\>_\mu (Y_s^\eps)\,ds \;.
\end{equs}
We have used the fact that the difference between $B_r-B_s$ and $\bar B_r^s$ is independent of $B_s$. The claim now follows from Birkhoff's ergodic theorem (or the quantitative version given in Lemma~\ref{lem:LLN}).
\end{proof}

\begin{proposition}\label{R-convergence}
Let $H\in (\f13, \f 12)\cup (\f 12, 1)$, let Assumptions~\ref{ass:ergodic}--\ref{ass:integrability}  hold, 
and let $f\in E$. One has  
\begin{equ}
\lim_{\eps \to 0} \int_0^t R_s^g\,dJ_s^\eps(f) = - {t\over 2}\Gamma(2H-1) \scal{\CL^{1-2H}g, f}\;,
\end{equ}
in probability.
\end{proposition}

\begin{proof}
Since the  expectation converges by Lemma~\ref{lem:convMean}, it remains to show that the
 variance vanishes as $\eps \to 0$. As in the proof of Proposition~\ref{prop:tight2nd}, we can fix a realisation of
$Y$ and use Proposition~\ref{prop:contractWiener} to reduce ourselves to the case where, conditional on $Y$, $R_t^g$ and $J_t^\eps(f)$
are independent. As a consequence of the proof of \Cref{lem:convMean}, $ \int_0^t J_s^\eps(\bar g)\,dJ_s^\eps(f)\to 0$ in probability, so it suffices to bound the conditional variance of the
term with $R_t^g$ replaced by $\tilde R_t^g$.

Writing $A_t = \int_0^t \tilde R_s^g\,dJ_s^\eps(f)$ (and assuming that $\tilde R^g$ and $J^\eps(f)$ are driven by
independent fractional Brownian motions), we then have as in \eqref{e:VarianceJJtilde}
the identity
\begin{equ}
\E \big(A_t^2\,|\, \CF^Y\big) = {1\over 2} \int_0^t \int_0^t \phi_\eps(s,s') \eta''(s-s')\,ds'\,ds
\end{equ}
but this time we have
\begin{equ}[e:defphieps]
\phi_\eps(s,s') = \eps^{2-4H}
f_{s'}^\eps f_s^\eps\int_s^\infty \!\!\!\int_{s'}^\infty  \bigl(P_{r'-s'\over \eps} \tilde g\bigr)_{s'}^\eps
 \bigl(P_{r-s\over \eps} \tilde g\bigr)_s^\eps C_{s\wedge s'}(r,r')\, dr'\,dr\;,
\end{equ}
where we use the shorthand $f_s^\eps = f(Y_s^\eps)$ and where
\begin{equ}
C_{s\wedge s'}(r,r') = \E \dot{\bar B}^s_r\dot{\bar B}^{s'}_{r'}
= \int_{-\infty}^{s\wedge s'} (r-u)^{H-\f32}(r'-u)^{H-\f32}\,du\;.
\end{equ}
Note that this holds for any $ (\f13, \f 12)\cup (\f 12, 1)$ and for $r, r'\ge s$, we have the bound
$|C_s(r,r')| \lesssim |r-s|^{H-1}|r'-s|^{H-1}$,
which will be used repeatedly below. As a consequence of this, the $s \leftrightarrow s'$ symmetry
of the integrand in \eqref{e:defphieps} and Assumptions~\ref{ass:ergodic}, we obtain for $p_\star \ge 4$ the upper bound
\begin{equs}
\|\phi_\eps(s,s)\|_{L^1} &\lesssim 
\eps^{2-4H} \int_s^\infty e^{-c(r-s)/\eps} \int_{s}^r  C_{s}(r,r')\, dr'\,dr\\
&\lesssim 
\eps^{2-4H} \int_s^\infty e^{-c(r-s)/\eps} \int_{s}^r |r'-s|^{H-1} |r-s|^{H-1}\, dr'\,dr\\
&\lesssim 
\eps^{2-4H} \int_s^\infty e^{-c(r-s)/\eps} |r-s|^{2H-1} \,dr\\
&= 
\eps^{2-2H} \int_0^\infty e^{-cr} r^{2H-1} \,dr
\lesssim \eps^{2-2H}\;, \label{e:boundphi}
\end{equs}
with constant proportional to $\|f\|_E^2 \|g\|_E^2$.
When $H > \f12$, using the local integrability of $\eta''$, 
this is sufficient to conclude that 
\begin{equ}
\|\E \big(A_t^2\,|\, \CF^Y\big)\|_{L^1}
\lesssim \eps^{2-2H} \Big|\int_0^t \int_0^t \eta''(s-s')\,ds'\,ds\Big|
\lesssim \eps^{2-2H} t^{2H}\;,
\end{equ}
which does indeed converge to $0$ as $\eps \to 0$ as desired.

It remains to  consider the case $H < \f12$, so we restrict ourselves to this case from now on.
Regarding $\delta\phi_\eps(s,s') = \phi_\eps(s,s')- \phi_\eps(s,s)$ for $s' > s$ (the case $s < s'$ is analogous), 
we write
$\delta \phi_\eps = \sum \delta \phi_\eps^{(i)}$ with
\begin{equs}
\delta \phi_\eps^{(1)}(s,s') &= 
\eps^{2-4H}
f_{s'}^\eps f_s^\eps\int_s^{s'} \int_{s'}^{\infty}  \bigl(P_{r'-s'\over \eps} \tilde g\bigr)_{s'}^\eps
 \bigl(P_{r-s\over \eps} \tilde g\bigr)_s^\eps C_{s}(r,r')\, dr'\,dr\;, \\
\delta \phi_\eps^{(2)}(s,s') &= 
-\eps^{2-4H}
f_{s}^\eps f_s^\eps\int_s^{s'} \int_{s}^{s'}  \bigl(P_{r'-s\over \eps} \tilde g\bigr)_{s}^\eps
 \bigl(P_{r-s\over \eps} \tilde g\bigr)_s^\eps C_{s}(r,r')\, dr'\,dr\;, \\
\delta \phi_\eps^{(3)}(s,s') &= 
-\eps^{2-4H}
f_{s}^\eps f_s^\eps\int_s^{s'} \int_{s'}^{\infty}  \bigl(P_{r'-s\over \eps} \tilde g\bigr)_{s}^\eps
 \bigl(P_{r-s\over \eps} \tilde g\bigr)_s^\eps C_{s}(r,r')\, dr'\,dr\;, \\ 
\delta \phi_\eps^{(4)}(s,s') &= 
-\eps^{2-4H}
f_{s}^\eps f_s^\eps\int_{s'}^\infty \int_{s}^{s'}  \bigl(P_{r'-s\over \eps} \tilde g\bigr)_s^\eps
 \bigl(P_{r-s\over \eps} \tilde g\bigr)_s^\eps C_{s}(r,r')\, dr'\,dr\;, \\
\delta \phi_\eps^{(5)}(s,s') &= 
\eps^{2-4H}
\bigl(f_{s'}^\eps-f_{s}^\eps\bigr) f_s^\eps\int_{s'}^\infty \int_{s'}^{\infty}  \bigl(P_{r'-s'\over \eps} \tilde g\bigr)_{s'}^\eps
 \bigl(P_{r-s\over \eps} \tilde g\bigr)_s^\eps C_{s}(r,r')\, dr'\,dr\;, \\
\delta \phi_\eps^{(6)}(s,s') &= 
\eps^{2-4H}
f_s^\eps f_s^\eps\int_{s'}^\infty \int_{s'}^{\infty}  \bigl(P_{r'-s'\over \eps} \tilde g-P_{r'-s\over \eps} \tilde g\bigr)_{s'}^\eps
 \bigl(P_{r-s\over \eps} \tilde g\bigr)_s^\eps C_{s}(r,r')\, dr'\,dr\;, \\
\delta \phi_\eps^{(7)}(s,s') &= 
\eps^{2-4H}
f_s^\eps f_s^\eps\int_{s'}^\infty \int_{s'}^{\infty}  \big(\bigl(P_{r'-s\over \eps} \tilde g\bigr)_{s'}^\eps-\bigl(P_{r'-s\over \eps} \tilde g\bigr)_{s}^\eps\big)
 \bigl(P_{r-s\over \eps} \tilde g\bigr)_s^\eps C_{0}(r,r')\, dr'\,dr\;.
\end{equs}
We obtain the bound
\begin{equs}
\|\delta \phi_\eps^{(1)}(s,s')\|_{L^1}
&\lesssim 
\eps^{2-4H}\int_{s'}^{\infty}   e^{-c(r'-s')/\eps}|r'-s|^{H-1} \int_s^{s'} |r-s|^{H-1}\, dr\,dr' \\
&\lesssim 
\eps^{2-4H}|s-s'|^H \int_{s'}^{\infty}  e^{-c(r'-s')/\eps} |r'-s'|^{H-1}\, dr' \\
&\lesssim 
\eps^{2-3H}|s-s'|^H\;,
\end{equs}
and similarly for $\delta \phi_\eps^{(3)}$ and $\delta \phi_\eps^{(4)}$.
Regarding $\delta \phi_\eps^{(2)}$, we obtain 
\begin{equ}
\|\delta \phi_\eps^{(2)}(s,s')\|_{L^1} \lesssim  
\eps^{2-4H} \Big(\int_s^{s'} e^{-c(r-s)/\eps} |r-s|^{H-1} \,dr\Big)^2
\lesssim \eps^{2-4H} |s-s'|^{2H}\;.
\end{equ}
In view of \eqref{e:boundphi} and Assumption~\ref{ass:continuous}, we obtain for $\delta \phi_\eps^{(5)}$ the bound
\begin{equs}
\|\delta \phi_\eps^{(5)}(s,s')\|_{L^1} \lesssim \eps^{2-3H}|s'-s|^H\;,
\end{equs}
using  $\|f^\epsilon_{s'}-f^\epsilon_s\|_{L^p}\lesssim(|s-s'|/\eps)^H$ to obtain the increment in time.

In order to bound $\delta \phi_\eps^{(6)}$, we note that one has the bound
\begin{equs}
\|(P_t \tilde g-\tilde g)(Y_s^\eps)\|_{L^p} &= \bigl\|\E \bigl(\tilde g(Y_{t+s}^\eps) - \tilde g(Y_s^\eps)\,|\,\CG_s\bigr)\bigr\|_{L^p}\\
&\le  \bigl\|\tilde g(Y_{t+s}^\eps) - \tilde g(Y_s^\eps)\bigr\|_{L^p}
\lesssim \|\tilde g\|_E \bigl( 1\wedge t^H \eps^{-H} \bigr)\;.
\end{equs}
As a consequence
of Assumption~\ref{ass:ergodic}, we thus obtain the bound 
\begin{equ}
\|\delta \phi_\eps^{(6)}(s,s')\|_{L^1} \lesssim
\eps^{2-3H}|s'-s|^H \;,
\end{equ}
and similarly for $\delta \phi_\eps^{(7)}$. Collecting all of these bounds, we conclude that 
\begin{equ}
\|\delta \phi_\eps(s,s')\|_{L^1}  \lesssim \eps^{2-3H}|s'-s|^H + \eps^{2-4H} |s-s'|^{2H}\;.
\end{equ}
It suffices then to apply Lemma~\ref{lem:boundIntEta''} to $\phi_\epsilon$ with $\beta = 0$, $\hat C = 0$, and $\zeta \in \{H,2H\}$ to
conclude that $\|\E (A_t^2\,|\, \CF^Y)\|_{L^1} \lesssim \eps^{2-2H} t^{2H} + \eps^{2-4H}t^{4H} $, which converges to $0$ as $\eps \to 0$, thus concluding the proof.
\end{proof}

Collecting all of these results, we conclude that the following holds.

\begin{proposition}\label{convergence-drivers}
Let Assumptions~\ref{ass:ergodic}--\ref{ass:integrability}  hold and let 
$f_1,\ldots f_N \in E$ for some $N \ge 1$. 
The processes $\big(J_t^\epsilon(f_i),  \JJ_{s,t}^\epsilon(f_i,f_j)\big)_{i,j\le N}$ converge jointly in distribution  to
$$
\Big(W_t^{(i)} , \int_s^t W_r^{(i)}\,dW^{(j)}(r) + \f 12(t-s)\Gamma(2H+1) \scal{f_i, \CL^{1-2H} f_j}_\mu\Big)_{i,j\le N}\;,
$$
where the $W^{(i)}$ are Wiener processes with covariance 
\begin{equation}
\E W^{(i)}_s W^{(j)}_t =\f 12  (s\wedge t) \Gamma(2H+1) \bigl(\scal{\CL^{1-2H}f_i, f_j}_\mu + \scal{\CL^{1-2H}f_j, f_i}_\mu\bigr)\;.
\end{equation}
\end{proposition}
\begin{proof}
If $H<\f 12$,  by \Cref{prop:convFirst} and \Cref{lem:boundR}, $(J_t^\epsilon(f_i), M_t^\epsilon(f_i), R_t^{f_i})_{i\le N}$ all converge jointly to $(W_t^{(i)}, W_t^{(j)},0)$.
This holds similarly for $H>\f 12$, using \Cref{lemma:convergence-high}.
Since by \Cref{R-convergence} below, $\int_0^t R_s^{f_i}\,dJ_s^\eps(f_j) $ converges to a deterministic limit,
and $\int_0^t J_s^\epsilon(f_i)dM_t^\epsilon(f_j)\to \int_0^tW_r^{(i)}\,dW_r^{(j)}$,
the desired  convergence in distribution follows by combining \eqref{r:idenJJ} with the  standard convergence theorem of 
stochastic integrals, c.f.   \cite[Theorem 6.22]{Jacod-Shiryaev} and  \cite[Theorem 2.7]{Kurtz-Protter}.

The case $H=\f 12$ is straightforward. Firstly, we see that conditional on $\CF^Y$, the $J^\epsilon_{s,t}(f_i) =\int_s^t f_i(Y^\eps_r) \,dB(r)$ are  $L^2$ bounded martingales with respect to the filtration generated by $B$. By Lemma~\ref{lem:LLN}, their covariances $ \int_s^t \E (f_i f_j)(Y^\eps_{r})\, dr$ converge to $(t-s)\<f_i,f_j\>_\mu$ in $L^2$. Then,
\begin{equ}
\JJ_{s.t}^\epsilon = \int_s^t J^\eps_r(f_i)\,\circ  dJ_r^\epsilon(f_j) 
 =\int_s^t \int_s^u f(Y_r^\epsilon) g(Y_u^\epsilon) dB_r  dB_u  
+\f 12 \int_s^t \big(fg\big) (Y_r^\epsilon) dr\;,
\end{equ}
which converges in $L^2$.  Since $J^\epsilon_{s,t}(f_i)$ converge, they converge together with their integrals,
  concluding the convergence of $(J_t^\epsilon(f_i),  \JJ_{s,t}^\epsilon(f_i,f_j))_{i,j\le N}$ for $H=\f 12$. 
 \end{proof}


\appendix

\section{A compactness criterion}

Recall the definitions of $\CB$ and $\CB_2$ from Section~\ref{sec:precise}, fix $\zeta \in (0,1)$ and $\kappa > 0$, 
and define $\hat \CB \subset \CB$ as the space of 
$\CC^3$ functions $X$ on $\R^d$ such that 
\begin{equ}[e:defhatB]
\|Z\|_{\hat \CB} \eqdef \sup_{x \in \R^d }\bigl( (1+|x|)^{\kappa/2} \|Z\|_{\CC^{3+\zeta}_x}\bigr) < \infty\;,
\end{equ}
where 
\begin{equ}
\|Z\|_{\CC^{3+\zeta}_x} = \sup_{x'\,:\, |x' - x|\le 1}\sup_{|\ell|\le 3} \Big(|D^\ell Z(x)| + \f{ |D^\ell Z(x)-D^\ell  Z(x')|}{|x-x'|^\zeta} \Big)\;.
\end{equ}
Similarly, we define $\hat \CB_2$ as the space of functions on $\R^d\times \R^d$ such that
\begin{equ}
\|Z\|_{\hat \CB_2} \eqdef \sup_{x, \bar x \in \R^d }\bigl( (1+|x|+|\bar x|)^{\kappa/2} \|Z\|_{\CC^{3+\zeta}_{(x,\bar x)}}\bigr) < \infty\;.
\end{equ}
We then have the following.

\begin{lemma}\label{lem:compact}
The embeddings $\hat \CB \subset \CB$ and $\hat \CB_2\subset \CB_2$ are compact for any $\zeta \in (0,1)$ and
$\kappa > 0$.
Furthermore, there exists a constant $C$ such that for any random $\CC^3$ functions $Z$ and $\bar Z$ one has the bound
\begin{equs}
\E \|Z\|_{\hat \CB}^p &\le C  \sup_{|x-x'| \le 1}\sup_{|\ell| \le 3} (1+|x|)^{\kappa p}{\E |D^\ell Z(x)-D^\ell Z(x')|^p\over |x-x'|^p}\;, \label{e:boundZZZ1}\\
\E \|\bar Z\|_{\hat \CB_2}^p &\le C  \sup_{|x-x'| \le 1 \atop |\bar x-\bar x'| \le 1} \sup_{|k+\ell| \le 3} (1+|x|+|\bar x|)^{\kappa p}{\E |D_1^kD_2^\ell \bar Z(x,\bar x)-D_1^kD_2^\ell \bar Z(x', \bar x')|^p\over |x-x'|^p + |\bar x-\bar x'|^p}\;,
\end{equs}
provided that $p \ge d$, $\zeta < 1-d/p$,  and $\kappa > 4d/p$.
\end{lemma}

\begin{proof}
The compactness statement is a routine modification of Arzelà--Ascoli. 
Regarding the first bound, it follows from Kolmogorov's continuity criterion \cite[Thm 2.1, p26]{RevuzYor} that, writing $K$
for the right-hand-side of \eqref{e:boundZZZ1}, there is $C > 0$ such that
\begin{equ}
\E \|Z\|_{\CC^{3+\zeta}_x}^p \le C (1+|x|)^{-\kappa p} K\;,
\end{equ}
provided that $0<\zeta < 1-d/p$. We then cover $\R^d$ with balls of diameter $1$ and 
note that a norm equivalent to that of $\hat \CB$ is obtained by restricting the supremum in
\eqref{e:defhatB} over the centres of these balls. Since $\kappa > 2d/p$,
we can then simply bound the supremum by the sum, yielding
\begin{equ}
\E \|Z\|_{\hat \CB}^p
\le \sum_x \E \bigl( (1+|x|)^{\kappa p/2} \|Z\|_{\CC^{3+\zeta}_x}^p\bigr) 
\le C \sum_x (1+|x|)^{-\kappa p/2} K \le C' K\;,
\end{equ}
as claimed. The second bound is identical, except that $d$ is replaced by $2d$.
\end{proof}

\endappendix

\bibliographystyle{Martin}
\bibliography{refs-averaging-low}

\end{document}